\documentclass[a4paper]{amsart}

\usepackage{amsmath}  
\usepackage{amsfonts}
\usepackage{amssymb}
\usepackage{times}
\usepackage{helvet}
\usepackage{bbm}

\usepackage{array}
\usepackage{graphicx}
\usepackage{color}

\usepackage[swedish,english]{babel}

\usepackage[latin1]{inputenc}
\usepackage[T1]{fontenc}

\usepackage{float}
\usepackage{tikz}
\usetikzlibrary{arrows,decorations.pathmorphing,matrix,chains,positioning}

\theoremstyle{plain}
\newtheorem{thm}{THEOREM}[section]
\newtheorem{lm}[thm]{LEMMA}

\newtheorem{prop}[thm]{PROPOSITION}

\theoremstyle{definition}
\newtheorem{defi}[thm]{DEFINITION}
\newtheorem{remark}[thm]{Remark}
\newtheorem{assumption}[thm]{Assumption}
\theoremstyle{remark}

\newcommand{\bilddir}{bilder}

\newcommand{\R}{\mathbb{R}}

\newcommand{\Ee}{\mathbb{E}}
\newcommand{\Eetn}{\mathbb{E}_{\tau_n}}
\newcommand{\Pe}{\mathbb{P}}
\newcommand{\bigoh}{\mathcal{O}}
\newcommand{\symn}{\mathfrak{S}_n}

\newcommand{\one}{\mathbbm{1}}

\newcommand{\SEn}{S_{E_n}}
\newcommand{\SEns}{S_{n}}
\newcommand{\SEne}{S_{E_n,\epsilon}}
\newcommand{\SEnes}{S_{n,\epsilon_n}}

\newcommand{\di }{\mathrm{d}}
\newcommand{\pfneps}{\mathcal{Z}_{n,\epsilon}}

\newcommand{\tz}{\tilde{z}}
\newcommand{\tx}{\tilde{x}}

\title{A Kac Model with Exclusion}
\author{Eric Carlen$^{(1)}$ and Bernt Wennberg$^{(2)}$ }

\begin{document}

\maketitle

\begin{center}
(1) Department of Mathematics, 
Rutgers University \\
110 Frelinghuysen Rd., Piscataway NJ 08854-8019, USA\\
email: carlen@math.rutgers.edu \\
\bigskip

(2) Department of Mathematical Sciences, \\
Chalmers University of Technology and University of Gothenburg\\
SE41296 G\"oteborg, Sweden \\
email: wennberg@chalmers.se

\end{center}

\begin{abstract}  We consider a one dimension Kac model with
  conservation of energy and an exclusion rule:  Fix  a number of
  particles $n$, and an energy $E>0$. Let each of the particles have
  an energy $x_j \geq 0$,  with $\sum_{j=1}^n x_j = E$.   For some  
$\epsilon$, the allowed configurations $(x_1,\dots,x_n)$ are those
that satisfy $|x_i - x_j| \geq \epsilon$ for all $i\neq j$.  At each
step of the process, a pair $(i,j)$ of particles is selected uniformly
at random, and then they ``collide'',  and there is a repartition of
their total energy $x_i + x_j$ between them producing new energies
$x^*_i$ and $x^*_j$ with $x^*_i + x^*_j = x_i + x_j$, but with the
restriction that exclusion rule is still observed for the new pair of
energies.   
This process bears some resemblance to 
Kac models for Fermions in which the exclusion represents the effects
of the Pauli exclusion principle.  However, the ``non-quantized''
exclusion rule here, with only a lower bound on the gaps, introduces
interesting novel features, 
and a detailed notion of Kac's chaos is required to derive an evolution
equation for the evolution of rescaled empirical measures for the
process, as we show here. 
 \end{abstract}

\section{Introduction}
\label{sec:intro}

The first attempts to formulate kinetic equations for colliding
particles that satisfy Boson or Fermion statstics go back at least to
the works of 
Nordheim~\cite{Nordheim1928} and Uehling and
Uhlenbeck~\cite{UehlingUhlenbeck1933}. To make a rigorous derivation of
these equations starting from the Schrödinger equation for a large
system of particles has proven very
difficult. See~\cite{Benedetto_etal2007} for a review. To understand
the classical spatially homogeneous Boltzmann equation Mark Kac
introduced a Markov 
jump process to mimic a real $n$-particle system and from this model
he could rigorously derive a simplified (one dimensional) Boltzmann
equation~\cite{Kac1956}. A similar kind of jump process has been
studied by Colangeli et al.~\cite{Colangeli_etal2015}, who derive a
kinetic equation from a particle system with discretized phasespace
with exclusion.

We investigate a Kac model on the simplex with exclusion, but without
dividing the simplex into cells, paying close attention to questions
concerning Kac's notion of chaos for the model. 
Before we introduce our model with exclusion, it will be helpful to
recall Kac's notion of chaos in the context of the corresponding
model, in which states are characterized by their energy only, without
exclusion.  

Consider a system of $n$ (indistinguishable) particles with a total
energy $E_n$, and assume that the state of a particle is determined by
its energy $x_j \geq 0$. The phase space of this system is then the simplex
\begin{align}
  \label{eq_01}
  \SEn := \left\{ (x_1,\dots,x_n)\in \R^n_+ \ :\
  \textstyle{\sum_{j=1}^n x_j }= E_n\right\}\ .
\end{align}
 Let  $\sigma_n$  denote  the uniform probability measure in $\SEn$. 
 In its simplest form, the Kac walk on the simplex $\SEn$ is the
 process in which binary collisions  occur in a Poisson stream of jump
 times, with the expected waiting time between jumps being $1/n$, and when a
 jump occurs, a pair $(i,j)$, $1\leq i < j \leq n$ is selected
 uniformly at random, and then the energy of the pair is redistributed
 by the ``collision'', a new  energy $x_i^*$ for the $i$-th particle is
 selected uniformly at random from $[0,x_i+x_j]$, and then $x_j^*$ is
 fixed by $x_i^*+x_j^* = x_i+x_j$. 
It is easy to see that the uniform probability measure is the unique
invariant measure for this process, and the single particle 
marginals in equilibrium are certain beta distributions. The rate of
approach to equilibrium has been studied by Giroux and
Ferland~\cite{GirouxFerland2008}.  The original
Kac process \cite{Kac1956} takes place on the $n-1$-dimensional sphere
consisting of vectors $(v_1,\dots,v_n)$ 
such that $\sum_{j=1}^n v_j^2 = E_n$. The process described above is
the image of the process on the sphere under the change of variables
$x_j = v_j^2$. 

As Kac discovered, the Kac process on the sphere {\em propagates
  chaos}, and it follows readily that the process on the simplex does
as well.  This means the following: 
For any probability density $F_n$ with respect to the uniform
probability  measure on $\SEn$, consider the empirical distribution 
\begin{equation}\label{fk1}
\mu_n=\frac{1}{n}\sum_{j=1}^n \delta(x - \tilde{x}_j)  \quad{\rm
  where}\quad  \tilde{x}_j = \frac{n}{E_n}x_j\ , 
\end{equation}
where $(x_1,\dots,x_n)$ is distributed according to $F_n$. Note that
\begin{equation}\label{fk2}
\int_0^\infty x{\rm d}\mu_n = 1
\end{equation}
for every $(\tilde x_1,\dots\tilde x_n)\in \SEn$.
The notion of  chaos concerns sequences, indexed by $n$ of
probability measues or  
densities on $\SEn$. We use upper case letters, e.g. $F_n$, to denote
such densities.   
While in some contexts it is natural to reserve upper-case lets for
the distribution funstions of probabilitiy densities on the 
line, here such distribution functions do not come into play, and it
is more natural to use upper and lower case to distinguish between
probability densities on the high dimensional space  
$\SEn$, and probability densities on $\R_+$.

Now let $g(x)$ be any probability density on $\R_+$ with
$\int_0^\infty x g(x){\rm d}x =1$.  
A sequence $\{F_n\}$ of probability densities on $\SEn$ is called {\em
  $g$-chaotic} in the sense of Kac in case the sequence $\{\mu_n\}$ of
empirical distributions as specified above converges in probability to
 $g(x){\rm d}x$. {  The notion of chaos is often presented directly in
terms of the probability measures: Consider a sequence of probability
measures $\{ m_n \}_{n}^{\infty}$ where the $m_n$  are  symmetric probability
distributions on $E^n$, the $n$-fold product of a metric space
$E$. Then  $\{ m_n \}_{n}^{\infty}$ is said to be $m$-chaotic for some
probability measure $m$ on $E$, if for every $k\ge 1$ and
functions $\phi_j\in C(E)$, $j=1,...,k$ the following limit holds:
\begin{equation}
  \label{eq:kacchaos}
  \lim\nolimits_{n=k}^{\infty}\int_E\cdots\int_E
  \phi_1(x_1)\cdots\phi_k(x_k)m_n(dx_1,...,dx_n) - \prod_{j=1}^k\int_E
  \phi_j(x) m(dx) =0\,.
\end{equation}
In fact, this definition is equivalent to the definition given in
terms of the empirical measures, as proven e.g. in~\cite{Sznitman1991}\,.}

Kac's main result in \cite{Kac1956} (for the spherical case) is that
if one starts with a chaotic sequence $\{F_n\}$ of initial data that
is $g$-chaotic,  and if for each $t>0$ one lets  $\{F_{n,t}\}$ denote
the sequence  of densities resulting from running the evolution for a
time $t$, then this sequence is $g_t$-chaotic for some density $g_t$,
and moreover, $g_t$ is the unique solution of a certain non-linear
Boltzmann-like equations starting from the initial data $g$. Thus,
this Boltzmann-like equation gives a complete description, in the
large $n$ limit of the evolution of the scaled empirical distribution
under the Kac process provided one starts with chaotic initial data.  

Since $\SEn$ is very close to being a product space, it is possible to
construct $g$-chaotic initial data for any probability density $g$
satisfying $\int_0^\infty xg(x){\rm d}x =1$, $\int_0^\infty x^2
g(x){\rm d}x <\infty$ and $g\in L^p(\R_+)$ for some $p>1$: 
One takes $\prod_{j=1}^n g(\tilde x_j)$, and restricts it to simplex
$\SEn$, and normalizes \cite{Carlen_etal2010}. By the Central Limit
Theorem, under $\prod_{j=1}^n g(x_j)$,  $\sum_{j=1}\tilde x_j$ is with
high probability very close to $n$, and so the mass is tightly
concentrated on $\SEn$. As long as one does not look at too many
coordinates at once, one cannot see the effects of the restriction. In
the physics literature, this is known as the {\em equivalence of
  ensembles}.   A
related result can be found in~\cite{Sznitman1991}, where $g$ is a
density  on $\R^k$, and the simplex is replaced by a set $x_1+\cdots +
x_n=n a\in\R^k$.

The restrictions that $\int_0^\infty x^2 g(x){\rm d}x <\infty$ and
$g\in L^p(\R_+)$ for some $p>1$ may then be removed in a limiting
process \cite{Carlen_etal2010}, and thus one has a construction of
chaotic initial data for every meaningful initial density $g$. The
corresponding nonlinear Boltzmann-like equation that governs the
evolution of the large $n$ empirical measure may then be studied in
terms of the linear Kolmogorov equation associated to the Kac process
on $\SEn$. 

That is, Kac had found an interesting way to study, by probabilistic
means, a class of non-linear equations of a type that arise in kinetic
theory. The method relies on the introduction of a family of
stochastic processes indexed by $n$, the number of particles. Because
of constraints such as $\sum_{j=1}^n x_j = E_n$ that correspond to
conservation laws in the particle system, the $x_j$ are not
independent, but their dependence is weak enough, for a wide class of
sequences of probability measures including  $\{\sigma_n\}$, that the
empirical measure in \eqref{fk1} becomes non-random as $n\to \infty$.  

In the  model introduced next, we consider another type of kinematic
constraint. In addition to the conservation of energy, we impose an
exclusion condition. This brings dependencies of a new type into
consideration, and we show that Kac's notion of chaos is not enough to
identify the evolution of  a limiting density. Therefore a new approach
is required, and a stronger notion of chaos,  and one such  
approach is developed here.

\subsection{The incorporation of exclusion}

For Fermions, the Pauli exclusion principle asserts that a state (here
characterized by its energy) only can be occupied by at most one
particle.   In this continuous setting, we model this without
``quantizing'' the state space,  by requiring 
that for all pairs of particles, we have $|x_j-x_k|>\epsilon$ for
some $\epsilon> 0$. We define 
\begin{align}\label{state}
\SEne := \left\{ (x_1,\dots,x_n)\in \R^n_+ \ :\
  \textstyle{\sum_{j=1}^n x_j }= E_n\ , |x_j-x_k|>\epsilon \ {\rm
  for\ all} \ i\neq j\right\}\ ,
\end{align}
and assuming that $E_n > \epsilon n(n-1)/2$ so that $\SEne \neq
\emptyset$, we let $\di\sigma_{n,\epsilon}$ denote the uniform probability
measure on $\SEne$.   

The process that we consider is the following: Again, the collision
times arrive in a Poisson stream with expected waiting time equal to
$1/n$, and again, when a jump time occurs, a pair $(i,j)$, $1\leq i <
j \leq n$ is selected uniformly at random. The energy of the two
particles is then reapportioned as before,  
with $x_i^*$ chosen uniformly from $[0,x_i+x_j]$ and then $x_j^* =
x_i+x_j - x_i^*$, except the jump only occurs if the new configuration  
$(x_1,\dots,x_i^*,\dots,x_j^*,\dots,x_n)$ of energy levels satisfies
the exclusion condition; i.e., only if it belongs to $\SEne$.   It is
easy to see that $\sigma_{n,\epsilon}$ is the invariant measure for
this process,  and since the process is reversible, it is natural to
refer to it as the {\em equilibrium measure}.  

While $\SEne$ is non-empty whenever $E_n > \epsilon n(n-1)/2$,  if
$E_n$ is not too much larger than this value, the spacing between most
levels will be very close to $\epsilon$.  Think of a long line of
parked cars with no marked spaces. For a new pair of cars to park,
they must both find gaps of sufficient width. If there is a constraint
on the sum of their distances from the start of the line, there may be
no way for them to park.  In terms of our model, if two cars pull out
and look for different spaces, it may be that their only option is to
return to the spaces they had (or to swap).  

We shall find interesting large $n$ limits only if the energies $E_n$
grow with $n$ in a certain way. Define 
%
\begin{equation}\label{fk3}
\alpha_n :=  \frac{\epsilon n (n-1)}{E_n}\,. 
\end{equation}
Then 
\begin{equation}\label{fk4}
  E_n - \frac{\epsilon n(n-1)}{2}= \left(1 - \frac{\alpha_n}{2}\right)E_n 
\end{equation}
is the {\em excess energy}, the  difference between the minimum energy
for a configuration of $n$ particles satisfying the exclusion
constraint and the available energy. Clearly we must have $0 \leq
\alpha_n\leq 2$. 
We shall be studying sequences of probability measures
$\{F_n\sigma_n\}$ on $\SEne$ with $E_n$ and $n$ related by 
\begin{equation}\label{fk5}
\lim_{n\to\infty}\alpha_n = \alpha \in ]0,2[\ .
\end{equation}
As before, we rescale the variables with
the average energy, 
\begin{align}\label{scale}
  \tx_j &= \frac{n}{E_n}x_j\,,
\end{align}
and define the empirical distribution
\begin{align}\label{empir}
\mu_n := \frac{1}{n}\sum_{j=1}^n \delta (x- \tilde{x}_j)\ .
\end{align}
We also need to rescale
$\epsilon$, and set 
\begin{equation}\label{fk8}
 \tilde{\epsilon}_n =\frac{\epsilon n}{ E_n}  = \frac{\alpha_n}{n-1}\, .
\end{equation}
Because, $\sum_{j=1}^n \tx_j = n$ for every $(x_1,\dots,x_n)\in
\SEne$,  one always has that 
\begin{equation}\label{fk6}
\int_0^\infty x{\rm d}\mu_n =1\ .
\end{equation}
The exclusion limits the amount of  mass the $\mu_n$ can assign to any
half open interval  $]a,b]$ in $\R_+$:   
There can be at most $(b-a)/\tilde{\epsilon}_n$ particles in this
interval, and hence 
\begin{equation}\label{fk8_12}
\int_{]a,b]}{\rm d}\mu_n  \leq  \frac{1}{n\tilde{\epsilon}_n}(b-a)  =
\frac{n-1}{n\alpha_n} (b-a) \ , 
\end{equation}

It follows from \eqref{fk8_12} that  if $\mu_n$ almost surely
converges vaguely  to $g(x){\rm  d}x$ along a sequence with $\alpha_n
\to \alpha$, then   
\begin{equation}\label{fk8b}
g(x) \leq \frac{1}{\alpha}\  \ ,
\end{equation}
and provided no mass escapes,
\begin{equation}\label{fk8c}
\int_0^\infty g(x){\rm d}x = 1 \ .
\end{equation}

{\em In what follows we will only use the rescaled variables $\tx_j$ and
$\tilde{\epsilon}_n$, but suppress the tildes from the
notation. { Moreover, in this scaling  $E_n=n$, and therfore
  $\SEne$ becomes $\SEnes$.}}

\medskip

At this point we can define a notion of chaos for our class of models:

\begin{defi}
  \label{def1}
Let $\alpha>0$ and let $f(x)$ be a probability density on $\R_+$. We
define a sequence $\{F_n\}_{n \geq 2}$ of probability measures on
$\SEnes$ to be $(\alpha,f)$-{\em   chaotic} if $(x_1,...,x_n)$ is
random with distribution $F_n$, and  
the empirical measures $\mu_n = \frac1n \sum_{j=1}^n \delta(x-x_j)$
converge in probability to $f(x){\rm d}x$ as $n\to \infty$ and 
$\alpha_n:= \epsilon n (n-1)/E_n \rightarrow \alpha$.   

Let ${\mathcal P}_t$ be the semigroup associated to a Markov process
on $\SEnes$, { that is,  $F_n(x,t)= {\mathcal P}_t F_n(x,0)$.}
Following Kac,
we say that the semigroup ${\mathcal P}_t$ {\em propagates chaos with
  parameter   $\alpha$} in case whenever $\{F_n(x,0)\}$
$(\alpha,f_0)$-chaotic, then  $\{F_n(x,t)\}$ is  $(\alpha,f_t)$-chaotic
for some probability density $f_t$ on $\R_+$.  
\end{defi}

In the Kac process that we study here, pairs of particle will interact
by redistributing their energies $x_i$ and $x_j$ to a new pair $x_i^*$
and $x_j^*$ with $x_i+x_j = x_i^*+x_j^*$ provided the gaps around
$x_i^*$ and $x_j^*$ are large enough for the exclusion constraint to
be satisfied. Let $x\in \R_+$. Then for all sufficiently large $n$,
and all $(x_1,\dots x_n)\in \SEnes$, $x < \max_{1\leq k \leq
  n}\{x_k\}$. Let  $x_{(j)}$  and $x_{(j+1)}$ be the pair of {\em
  consecutive} energies such that $x\in [x_{(j)},x_{(j+1)}[$. Define
the {\em gap at energy $x$} to be   
$$\zeta(x) := x_{(j+1)} - x_{(j)} - \frac{\alpha}{n-1}\ .$$
Only when $\zeta \geq  \frac{\alpha}{n-1} $ is it possible for an
interaction  to result in either $x_i^*\in [x_{(j)},x_{(j+1)}[$ or
$x_j^*\in [x_{(j)},x_{(j+1)}[$   since only in this case is the
minimum spacing $\frac{\alpha}{n-1}$ (in the scaled variable)
available above and below some energy in this interval.

It is probably intuitively clear, and will be shown later on, that the
evolution of the empirical density depends strongly on distribution of
the energy gaps: For a given probability density $f(x)$ as in
Definition\ref{def1}, and any $0 < \alpha < 2$, there are different
$(\alpha,f)$-chaotic sequences  $\{F_n\}_{n \geq 2}$ that have very
different gap distributions, and this will result in different sorts
of interactions being favored in the process, and thus to  different 
results for $f_t$ under the time evolution.  Thus, this definition as
it stands will not lead to a well-defined evolution equation for the
limiting density $f_t$. We must bring in information on the gaps.

\begin{defi}
  \label{def2}
Let a sequence $\{F_n\}_{n \geq 2}$  be $(\alpha,f)$-{\em
  chaotic} according to Definition~\ref{def1}. 
We say that 
 $\{F_n\}_{n \geq 2}$  is  $(\alpha,f)$-{\em
   chaotic in detail}  if for any $x\in R_+$, the random interval   
   $]x_{(j)}, x_{(j+1)}[$  that contains $x$,  the gap length
   $\zeta_{x,n}= 
  x_{(j+1)}- x_{(j)}- \alpha/(n-1)$ satisfies
  \begin{align}
    \label{eq:79gapdistA}
    \lim_{n\to\infty}\Pe[ (n-1) \zeta_{x,n}/\alpha > r]
    \rightarrow e^{-\frac{\alpha 
    f(x)}{1-\alpha f(x)} r}\,.
  \end{align}

We say that the semigroup ${\mathcal P}_t$ {\em propagates 
    detailed chaos with
  parameter   $\alpha$} in case whenever $\{F_n(x,0)\}$ is $(\alpha,
f_0)$ chaotic in detail, then the same holds for 
$\{F_n(x,t)\}$ for some probability density $f_t$ on $\R_+$.   
\end{defi}

As we shall show below, this particular gap distribution specified in
\eqref{eq:79gapdistA} is the only one that is possible: If the gap
lengths are asymptotically exponential, and the empirical distribution
is asymptotically deterministic with density $f$, then the exponential
rates must be related to $f$ as specified in
\eqref{eq:79gapdistA}. Thus one could formulate the definition less
specifically, only requiring that the gap lengths are asymptotically
exponential with {\em some} rate.

This is probably the simplest generalization of the notion of
chaos to our class of models with the exclusion constraint. We consider four
questions concerning the Kac model on the simplex with exclusion: 

\smallskip
\noindent{\it (1)}  Does the $\{\sigma_{n,\epsilon}\}$ of
equilibrium measures satisfy the detailed  chaos condition when $\alpha_n \to
\alpha$?  If so, what 
is the limiting density $f_\alpha$ for which this sequence is
$(\alpha,f_\alpha)$-chaotic, and
how does $f_\alpha$
 compare with the Fermi-Dirac distribution, which
one might expect  in a ``quantized'' model; i.e., one in which parking
spaces are marked with lines?

\smallskip
\noindent{\it (2)}  For which probability densities $g$ on $\R_+$
that satisfy \eqref{fk8b}  and \eqref{fk8c} do there exist sequences
that satisfy $(\alpha,g)$-chaos and detailed $(\alpha,g)$-chaos conditions?

\smallskip
\noindent{\it (3)} Is detailed chaos propagated, and if so, what is
the  equation that governs the evolution of the limiting marginal
densities? 

\smallskip
\noindent{\it (4)}  At which energy levels in equilibrium do
collisions occur a rate bounded away from zero, and at which energy
levels are the collisions ``frozen out''?  
\medskip

Theorem~\ref{main} gives a positive answer to the first question,
explicitly identifying $f_\alpha$, which is not the Fermi-Dirac
distribution; see Figure 1. Theorem~\ref{main} provides quantitative
bounds on the rate at which 
$W_1(\mu_n,f_\alpha{\rm d}x) \to 0$ in probability, where $W_1$ is the
Kantorovich-Rubinstein transport metric.  Mass transport methods are
the basis of a number of our proofs.  

Theorem~\ref{chacon} answers the second question -- such chaotic
sequences exist for {\em all} densities satisfying the two necessary
requirements \eqref{fk8b} and \eqref{fk8c}. Such sequences can be
constructed in qualitatively different ways, and we provide two
examples of constructions, the second one given in
Theorem~\ref{thm:chaosexp}.   Other results in this 
section provide quantitative chaos estimates, again in the $W_1$
metric for a broad class of densities $g$ satisfying mild regularity
hypotheses.

In section 4  we derive under the assumption of propagation of
detailed chaos, the  Boltzmann-like equation that
governs the evolution of the limiting empirical measure. The equation
resembles the Uehling-Uhlenbeck equation of quantum kinetic theory,
but with a different ``exclusion factor'' corresponding to our different
exclusion model. But this exclusion factor turns out to depend on the
chaotic sequence: Definition~\ref{def1} is not restrictive enough to
uniquely determine the evolution of the limiting empirical measure, though the additional information on the gaps provided in 
Definition~\ref{def2}  is enough.

 We prove that the limiting densities
$f_{\alpha}$ obtained from equilibrium measures
$\{\sigma_{n,\epsilon}\}$ are stationary solutions to the
Boltzmann-Kac equation. We do not prove that propagation of chaos
according to either Definition~\ref{def1} or Definition~\ref{def2} holds, but
we do provide numerical evidence that  detailed
$(\alpha,f)$-chaoticity is 
propagated, and also that if initial data are only
$(\alpha,f_0)$-chaotic, without the correct exponential  
gap distribution \eqref{eq:79gapdistA} for $\alpha$ and $f_0$, this
is actually improved by the evolution: The gap distribution converges
rapidly to the correct exponential distribution,  
so that in this sense it appear that not only is chaos propagated, but
it strengthens.   The numerical evidence for this is presented in
Section~\ref{sec:sim}, and further results are available as
supplementary material.

\section{The empirical distribution with exclusion}
\label{sec:empirical}

Equip the rescaled state space, still denoted $\SEnes$ and
defined in \eqref{state}, with the uniform probability measure
$\sigma_{n,\epsilon}$.  Let $\Ee$ denote
expectation with respect to this probability measure.  Then
$x_1,\dots,x_n$ become random variables.

For two probability measures $\mu$ and $\nu$ on $\R_+$, let
$W_1(\mu,\nu)$ denote the Kantorovich-Rubinstein distance between
$\mu$ and $\nu$.  
Recall that
\begin{equation}\label{KR}
W_1(\nu,\nu) = \sup\left\{ \left|\int_0^\infty  \chi {\rm d}\mu -
    \int_0^\infty  \chi {\rm d}\nu\right|\ :\ \chi \in {\rm Lip}_1
\right\} \,.
\end{equation}
${\rm Lip}_1$ denotes the class of $1$-Lipschitz functions; i.e.,   functions $\chi$ 
such  that  $|\chi(x) - \chi(y)| \leq |x-y|$ for all $x,y$.  Note that
we may restrict to $\chi\in {\rm Lip}_1$ and $\chi(0) =0$ without
changing anything.

\begin{thm}\label{main}
 For $\alpha_n= \epsilon n (n-1)/E_n \rightarrow \alpha$, the sequence
 of uniform probability measures on $\SEnes$  is 
 $(\alpha,f_\alpha)$-chaotic in detail,  where 
  \begin{align}\label{fk94}
  f_\alpha(x) &= \frac{d}{dx} \phi^{-1}(x) = \frac{1}{\phi'\left( \phi^{-1}(x)\right)}\,.
\end{align}
and
\begin{align}
  \label{eq:20.0}
 \phi(\xi)  := (1-\alpha/2) \log\left(\frac{1}{1-\xi}\right)
  + \alpha \xi \ .
\end{align}
Moreover, the sequence of 
 empirical measures  $\{\mu_n\}$, defined as in \eqref{empir},  is such that there is a constant $C$ such that for
              any $\delta>0$ and all sufficiently large $n$,  
  \begin{align}\label{fk94c}            
\Pe\{ W_1(\mu_n,f_\alpha {\rm d}x) > \delta\} \leq
\frac1\delta\left(\frac{C}{\sqrt{n}} +\frac32|\alpha_n-
  \alpha|\right)\ .
\end{align}
\end{thm}

  The theorem is a corollary of Theorem~\ref{mainX}, except for the
  statement of detailed chaoticity, which is proven in
  Section~\ref{subsec:strongchaos}. The function $\phi$ in
  equation~(\ref{eq:20.0}) is known as the quantile function for the
  distribution with density $f_{\alpha}$ and it is derived as a limit
  of explicitly constructed quantile functions for each $n$. We refer
  to~\cite{CiprianiZeindler2015} for similar functions related to
  Young diagrams and shape functions associated with random permutations.

Theorem~\ref{main} shows that although the exclusion introduces new
dependencies between the random variables $x_1,\dots,x_n$ that are far
more complicated that those induced by $\sum_{j=1}^n x_j = n $ which
would be the only constraint in the absence of exclusion, these new
dependencies are not an obstacle to chaos in the sense of Kac: If
$\sigma_{n,\epsilon}$ denote the law of $(x_1,\dots,x_n)$, and $\alpha_n\to \alpha$, then
$\{ \sigma_{n,\epsilon}\}$ is $(\alpha,f_\alpha)$-chaotic.  

While the form of $\phi(\xi)$ is simple, it seems difficult to express
the function $f$ in closed form, but it
clearly differs from the Fermi-Dirac density that is the relevant
expression in a quantized setting, although it does resemble it for
large values of $\alpha$. The function $f$ is plotted for some 
different values of $\alpha$ in Figure~\ref{fig:1}.

\begin{figure}[!h]
  \centering
  \includegraphics[width=0.5\textwidth]{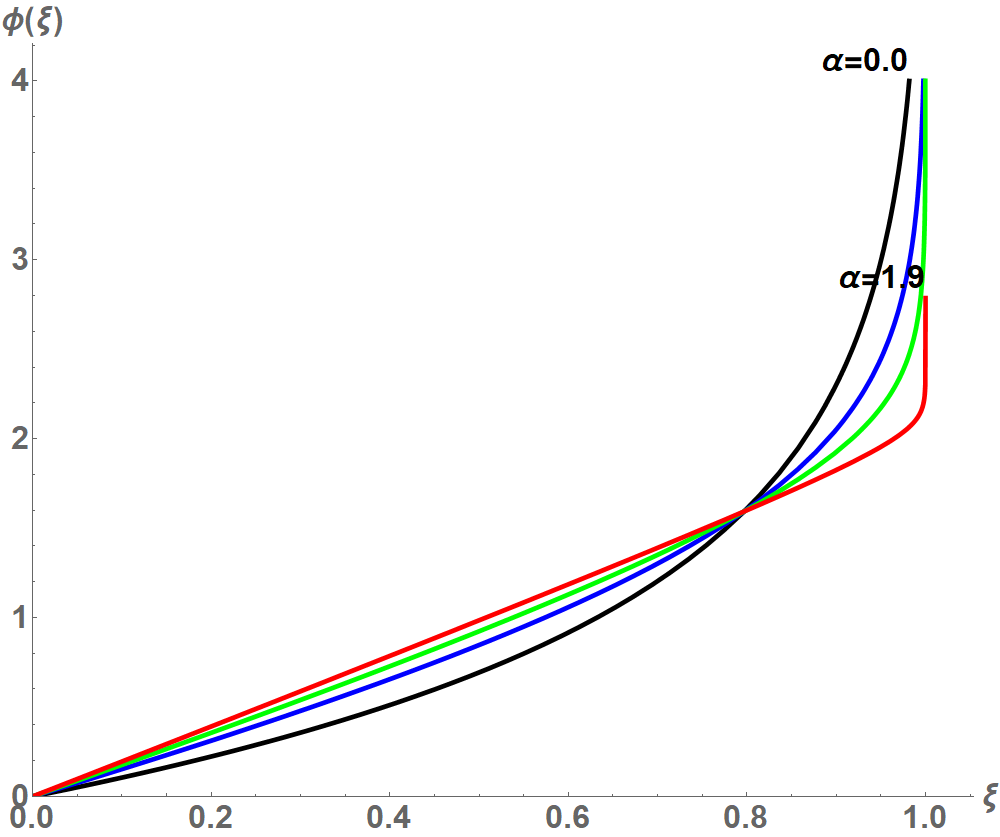}\;%
   \includegraphics[width=0.5\textwidth]{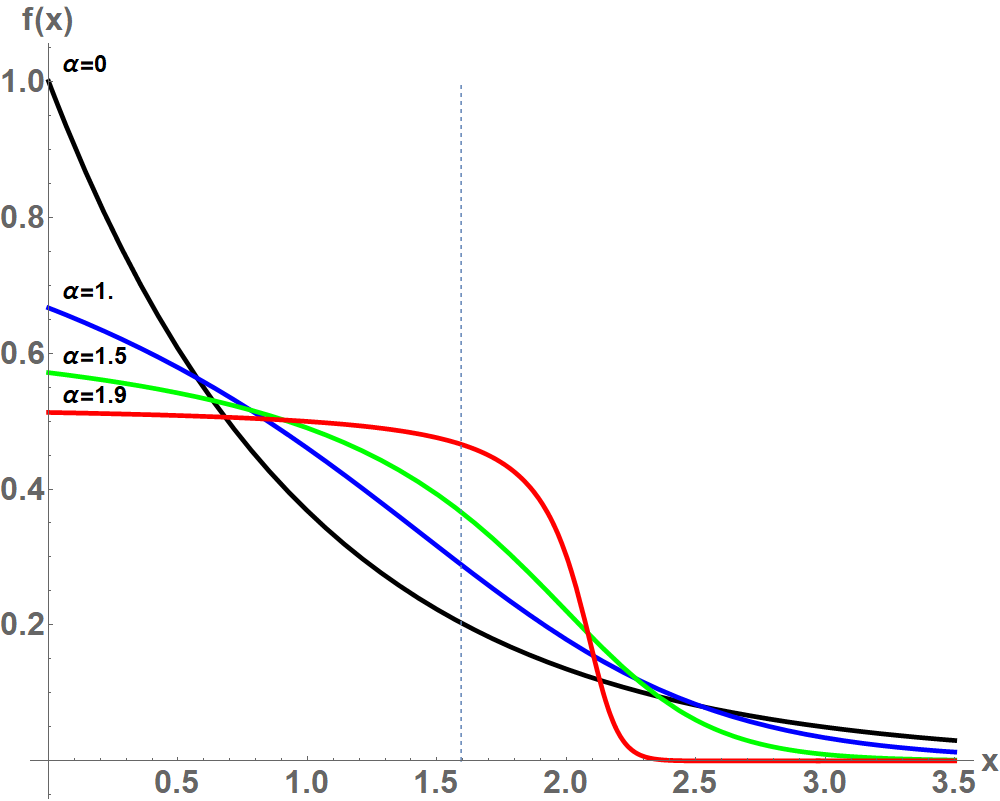}\;%
  \caption{\small The function $\phi(\xi)$ for $\alpha=0.0,1.0,1.5,1.9$
    (left), and the density $f(x)$ for the same values of
    $\alpha$. Curiously, for $\xi=\xi_0\approx 0.797$ (the solution of
    $1-\xi-e^{-2\xi}=0$), $\phi(\xi)$ is independent of $\alpha$, and
    hence the fraction of the mass of $f(x)$ in $0\le x\le 2\xi_0$
    (indicated by the dashed line) is
    $\xi_0$ for all values of $\alpha$}. 
  \label{fig:1}
\end{figure}

\subsection{Parameterization of $\SEnes$ by the standard simplex}

We shall make use of a parameterization of the state space $\SEnes$ in
terms of the standard simplex 
\begin{align}
  \label{eq_01b}
  S_1 := \left\{ (x_1,\dots,x_n)\in \R^n_+ \ :\
  \textstyle{\sum_{j=1}^n x_j }= 1\right\}\ .
\end{align}
We first define $\SEnes^*$ to be the subset consisting of all $(x_1,\dots,x_n)$ with $x_1 < x_2 < \cdots < x_n$. Up to a set of measure zero, one recovers $\SEnes$ by taking the union over all permutations 
$$
\bigcup_{\pi\in\symn} \{ (x_{\pi(1)}, \dots, x_{\pi(n)})\ :\ (x_1,\dots , x_n)\in \SEnes^*\ \}. 
$$
The measures we study are all symmetric under interchange of particles, and hence it suffices to parameterize $\SEnes^*$.

\begin{lm}\label{param} For $(\tz_1,\dots,\tz_n)\in S_1$, define $T_n(\tz_1,\dots,\tz_n)$ to be the vector in $\R_+^n$ whose $j$th component $x_j$ is given by
\begin{equation}\label{Tndef}
  x_j = n \left(1-\frac{\alpha_n}{2}\right) \left(\frac{\tz_1}{n}+
  \frac{\tz_2}{(n-1)}+\cdots+\frac{\tz_j}{n+1-j}\right) +
  \frac{j-1}{n-1}\alpha_n\,.
\end{equation}
Then $T_n$ provides a one-to-one parameterization of $\SEnes^*$ by $S_1$, and moreover $\sigma_{n,\epsilon}$ is the push-forward of the uniform probability measure on $S_1$ under $T_n$, averaged over permutations. 
\end{lm}

\begin{proof} First note that
$$\sum_{j=1}^n \left(\frac{\tz_1}{n}+  
  \frac{\tz_2}{(n-1)}+\cdots+\frac{\tz_j}{n+1-j}\right)= 1\quad{\rm and}\quad \sum_{j=1}^n \frac{j-1}{n-1}\alpha_n =\frac{n\alpha_n}{2}\ ,
  $$
  so that $\sum_{j=1}^n x_j= n$, and for $j>1$,
\begin{equation}\label{fk10}
  x_{j} - x_{j-1} = n
  \left(1-\frac{\alpha_n}{2}\right)\frac{\tz_j}{n+1-j}  +\epsilon_n
  \geq \epsilon_n\ . 
 \end{equation}
  Thus the image of $T_n$ lies in $\SEnes$.  Moreover, \eqref{fk10}
  shows that $T_n$ is invertible, and gives an explicit formula for
  the inverse from which one sees, by the same computations that
  $T_n^{-1}(\SEnes) \subset S_1$. This proves the statements about the
  parameterization. The proof of the description of
  $\sigma_{n,\epsilon}$ in terms of $T_n$ is somewhat more involved.  

We begin by considering the case with no exclusion ($\epsilon =0$): 
 The uniform
density is also  the equilibrium distribution of a set of particles at
equilibrium, so that for $\phi\in C(\R^n)$,
\begin{align}
\label{eq:1}
   \Ee( \phi(x_1,...,x_n))  &=
   \frac{1}{\mathcal{Z}} \int\limits_{0< x_1+...+x_{n-1}< n}
                              \phi(x_1,...,x_n)\; \di x_1 \di x_2
                              \cdots \di x_{n-1}\nonumber \\
   \nonumber \\
   &=\frac{1}{\tilde{\mathcal{Z}}} \int\limits_{\substack{0<
     x_1+...+x_{n-1}< n\\x_1<\cdots < x_n}} \sum_{\pi}
  \phi_{\pi}(x_1,...,x_n)\; \di x_1 \di x_2 \cdots \di x_{n-1}\,,
\end{align}
where $x_n=n-x_1-\cdots-x_{n-1}$, and, in the second row, $\phi_\pi$
denotes the composition of $\phi$ with the permutation operator $\pi:
(x_1,...,x_n) \mapsto (x_{\pi_1},x_{\pi_2}, ..., x_{\pi_n})$, and the
sum is taken over all permutations. The normalizing factor
$\tilde{\mathcal{Z}}$ is given by
\begin{align}\,.
 \int\limits_{\substack{0 <x_1+...+x_{n-1}< n\\0<x_1<\cdots < x_n}}
  \; \di x_1 \di x_2 \cdots \di x_{n-1}\,.
\end{align}
Here we have parameterized $\SEns$ with its projection on $\{
(x_1,...,x_{n-1}) \, | \,  x_j>0\,,\, x_1+\cdots+x_{n-1} < n\}$, and
set $\di \sigma(x_1,...,x_{n-1}) = \di x_1\cdots\di x_{n-1}$ without
the factor $\sqrt{n}$ which may anyway be absorbed into
$\tilde{\mathcal{Z}}$.

Now consider the case $\epsilon>0$: The
expectation in eq.~(\ref{eq:1}) can then be computed with the same
integrals, but adding the restriction that $x_{j}-x_{j-1}>\epsilon_n$
for all $j>1$. Therefore we set $z_{j} =  x_{j}-x_{j-1}-\epsilon_n>0$
for $1 < j < n$ and set $z_1 = x_1$. 
This yields the following change of variables:
\begin{equation}
\label{eq:2}
\begin{split}
x_1&= z_1\\
x_2&= x_1 + \epsilon_n + z_2 = z_1+z_2+\epsilon_n\\
\cdots \\
x_{n-1} &= z_1+\cdots +z_{n-1} + (n-2)\epsilon_n \\
x_n & = n-(n-1)z_1-(n-2)z_2-\cdots - z_{n-1} -  \epsilon_n (n-1)(n-2)/2.
\end{split}
\end{equation}
The Jacobian of $(z_1,...,x_{n-1}) \mapsto (x_1,...,x_{n-1})$ has
determinant one, and hence to compute the integrals in
eq.~(\ref{eq:1}), it is enough to find the domain of
$(z_1,...,x_{n-1})$.  For each $j < k \leq n$, and each $1 \leq m \leq
n-j$, we have 
$x_{j+m} \geq x_j + m\epsilon_n$, and evidently this is the smallest
value $x_{j+m}$ can take, given $x_j$.  
Therefore,
$$n - \sum_{k=1}^j x_k= \sum_{m=1}^{n-j}x_{j+m}  \geq (n-j)x_j + \frac{(n+1-j)(n-j)}{2}\epsilon_n\ .$$
Rearranging terms,
\begin{align}
       x_j < \frac{1}{n+1-j}\left( n - x_1-x_2-\cdots -x_{j-1} -
  \epsilon_n \frac{(n+1-j)(n-j)}{2} \right)\,.
\end{align}
Since $x_k = z_1 + \cdots + z_k + (k-1)\epsilon_n$,
$$\sum_{k=1}^{j-1} x_k= \sum_{k=1}^{j-1}(j-k)z_k  +\epsilon_n \frac{(j-1)(j-2)}{2}\ ,
$$
and since
$$
(j-1)(j-2)+ (n+1-j)(n-1) + 2(j-1)(n+1-j) = n(n-1)\ ,
$$
\begin{align}\label{fk61}
z_j &< \frac{1}{n+1-j}\left(n - \sum_{k=1}^{j-1}(j-k)z_k
      -\epsilon_n \frac{(j-1)(j-2)}{2}-\epsilon_n
      \frac{(n+1-j)(n-j)}{2}\right) \nonumber \nonumber \\
    &\quad - z_1-\cdots-z_{j-1} - (j-1)\epsilon_n \nonumber\\
    &=\frac{1}{n+1-j}\left(n \left(1 -\frac{\alpha_n}{2}\right)-
      \sum_{k=1}^{j-1}  (n+1-k)z _k   \right)\,. 
\end{align}
Define 
\begin{equation}\label{fk62}
\tz_j := \frac{n+1-j}{n}\frac{1}{1- \alpha_n/2}z_j\ .
\end{equation}
Then \eqref{fk61} becomes
$\displaystyle{\tz_j  \leq  1 - \sum_{k=1}^{j-1}\tz_k}$.
Using this notation,  the  version of Equation~(\ref{eq:1}) with exclusion  can
be written
\begin{align}
\label{eq:8}
  \Ee( \phi(x_1,...,x_n))  = & \frac{\mathcal{J}_{n,\epsilon}}{\pfneps}
   \int\limits_{0}^{1} \mathrm{d}\tz_1
                               \int\limits_{0}^{1 - \tz_1} \mathrm{d}\tz_2 \cdots
   \int\limits_{0}^{1 -\tz_1 -\cdots -\tz_{j-1}}\mathrm{d}\tz_j\cdots \nonumber \\
    & \cdots \int\limits_{0}^{1 -\tz_1 -\cdots -\tz_{n-2}}\mathrm{d}\tz_{n-1}\qquad \sum_{\pi}
      \phi_{\pi}(x_1,x_2,...,x_n)\, ,
\end{align}
where $\mathcal{J}_{n,\epsilon}$ is the Jacobian corresponding to the
change of variables given in \eqref{fk62}.  Taking $\varphi$ to be the
constant function $1$, it is evident that  
$$\pfneps =\mathcal{J}_{n,\epsilon}  =\frac{1}{n!}\left(n - \frac{n\alpha_n}{2}\right)^{n-1}\ ,$$
which gives the value of $\pfneps$. However, we only need to know that
$\mathcal{J}_{n,\epsilon}/\pfneps =1$, and then observe that the
substitution \eqref{fk62} transforms  
$x_{j} = z_1+\cdots +z_{n-1} + (j-1)\epsilon_n$ into \eqref{Tndef} for all $j < n$. 
\end{proof}

\begin{remark}
  \label{sampling}
Lemma~\ref{param} provides a convenient method for sampling
$(x_1,...,x_n)$:
simply take $ (\tz_1,...,\tz_n) $ uniformly from the standard
$n$-simplex, i.e. $\tz_1+\cdots+\tz_n=1$, and compute the $x_j$
according to the formula \eqref{Tndef}.
\end{remark}

\begin{remark}\label{partition} Because $T_n$ is  continuous and
  invertible with a continuous inverse, it sets up a one-to one
  correspondence between symmetric Borel probability measures on
  $\SEnes$ and  Borel  probability measures on $S_1$. This
  correspondence provides a useful way to think about  symmetric Borel
  probability measures on $\SEnes$ in terms of partitions of the excess
  energy. In the rescaled variables, the excess energy is 
$$n\left(1 - \frac{\alpha_n}{2}\right) =   \sum_{j=1}^n n\left(1 - \frac{\alpha_n}{2}\right)\tz_j\ .$$
Thus one may think of $\{\tz_1,\dots,\tz_n\}$ as specifying a
partition of the excess energy into $n$ components 
\begin{align}\label{ranpar}
 \left\{ n \left(1-\frac{\alpha_n}{2}\right)\tilde{z}_j\right\}_{j=1}^n\ .
\end{align}
The first term in the partition may be understood as making an equal contribution of  
$$
n \left(1-\frac{\alpha_n}{2}\right)\frac{\tilde{z}_1}{n}
$$
to the energy of each particle,  and in the same way the second
component  makes equal contributions of  
$$
n \left(1-\frac{\alpha_n}{2}\right)\frac{\tilde{z}_2}{n-1}
$$
to the energy of each of the last $n-1$ particle and so forth. Adding
these up, together with the total  excluded energies, one arrives at
\eqref{Tndef}.
\end{remark}

\section{Pre-chaotic sequences of probability measures on $S_1$}  
\label{sec:prechaos}

We now identify a class of sequences of measures on $S_1$ whose push-forwards under $T_n$ will be shown to be $(\alpha,g)$-chaotic on  $\SEnes$  in the sense of Definition~\ref{def1}.

For each $n$, let $\tau_n$ be a Borel probability measure on $S_1$.  Also define the function 
\begin{equation}\label{wdef}
w_n(\xi) = n \Eetn[\tilde z_k]  \quad{\rm for} \quad \frac{k-1}{n} < \xi \leq \frac{k}{n}\ . \qquad 1 \leq k \leq n\ ,
\end{equation}
and $w(0) =n\Eetn[\tilde z_1]$. Here $\Eetn$ is the expectation with respect to the measures $\tau_n$. 
That is, for $\xi>0$,
$w_n(\xi) = n\Ee[\tilde z_{\lceil n\xi\rceil}]$
where $\lceil n \xi \rceil$ is the least integer $k$ such that
$n\xi \leq k$.  Note that for each $n$,
$$\int_0^1 w_n(\xi){\rm d}\xi = \sum_{k=1}^n \Ee_{\tau_n}[\tilde z_k] = 1\ ,$$
so that $w_n$ is a probability density.

\begin{defi}\label{s1chaos}  Let $w:[0,1[\to \R_+$ be a continuous probability density. A sequence $\{\tau_n\}$ of probability measures on $S_1$ is $w$-{\em pre-chaotic} in case 
\begin{equation}\label{wa5}
n\Eetn[\tilde z_j]  =   w(j/n) + r(j/n)
\end{equation}
{ for a continuous function $r(s)$ decreasing to zero when
  $s\rightarrow 0$},  where for each $0 < \xi_*< 1$, and each
  $\epsilon>0$, there is an $n_\epsilon$ so that  
\begin{equation}\label{wb5}
| r(j/n)| < \epsilon \qquad{\rm for\ all}\quad n> n_\epsilon, \ j<  n\xi_* \ . 
\end{equation}
and moreover, for some constant $C<\infty$ depending only on $\xi_*$,
\begin{equation}\label{w4}
 \mathrm{Var}[\tz_i] \leq \frac{C}{n}\epsilon \qquad{\rm and}\qquad| {\rm Cov}(\tilde z_j,\tilde z_k)| \leq \frac{C}{n^2}\epsilon   \qquad{\rm for\ all}\quad n> n_\epsilon, \ j,k<  n\xi_* \ . 
\end{equation}
for all $n$, $j$ and $k$.
\end{defi}

\begin{remark}\label{equil}
 By Lemma~\ref{param}, the equilibrium distribution $\sigma_{n,\epsilon}$
arises when the random partition in \eqref{ranpar} is determined by
choosing $(\tz_1,\dots,\tz_n)$ from a flat Dirichlet distribution;
i.e., the  
uniform density on $S_1$, and then the random variables $\tz_j$
satisfy
\begin{align}
\label{eq:15}
\Ee[\tz_i]&= \frac{1}{n}\,, \nonumber\\
\mathrm{Var}[\tz_i] &=  \frac{(n-1)}{n^2(n+1)}\,,  \\
\mathrm{Cov}[\tz_i,\tz_j] &=  \frac{-1}{n^2 (n+1)}\,.  \nonumber
\end{align}
Moreover, it is clear that for each $n$, $w_n(\xi) =1$ for all $n$ and $\xi$.  In this case, $w$ is continuous on the closed interval $[0,1]$ and hence is bounded at $1$ also, though the definition allows for $w(t)$ to diverge as $t\uparrow 1$.  Later, we shall see that we need this generality.
\end{remark}

The next lemma will be used several times in what follows.

\begin{lm}\label{cut}  Let $f$ and $g$ be two non-negative integrable functions on $[0,1]$ such that
$$
\left|\int_0^1 (f(\xi) - g(\xi)){\rm d}\xi \right| \leq a
$$
 Then for all $0 < \xi_* < 1$, 
$$
\int_0^1 | f(\xi) - g(\xi)|{\rm d}x \leq  2\int_0^{\xi_*} | f(\xi) - g(\xi)|{\rm d}x  + 2\int_{\xi_*}^1 g(\xi){\rm d}\xi  + a\ .
$$
\end{lm}

\begin{proof} We have
\begin{eqnarray*}
\int_0^1 | f(\xi) - g(\xi)|{\rm d}x  &=&  \int_0^{\xi_*} | f(\xi) - g(\xi)|{\rm d}x + \int_{\xi_*}^1 | f(\xi) - g(\xi)|{\rm d}x\\
 &\leq&  
 \int_0^{\xi_*} | f(\xi) - g(\xi)|{\rm d}x + \int_{\xi_*}^1 f(\xi){\rm d}\xi + \int_{\xi_*}^1 g(\xi){\rm d}\xi
\end{eqnarray*}
Next,
\begin{eqnarray*}
\int_{\xi_*}^1 f(\xi){\rm d}\xi  &=& \int_{\xi_*}^1 g(\xi){\rm d}\xi +  \int_{\xi_*}^1( f(\xi)- g(\xi){\rm d}\xi \\
 &\leq&  \int_{\xi_*}^1 g(\xi){\rm d}\xi  + a -  \int_0^{\xi_*}( f(\xi)- g(\xi)){\rm d}\xi\\ 
 &\leq&  \int_{\xi_*}^1 g(\xi){\rm d}\xi +  a + \int_0^{\xi_*}| f(\xi)- g(\xi)|{\rm d}\xi\,.
\end{eqnarray*}
\end{proof}

Our first application is the following:

\begin{lm} \label{L1lm}  Let $\{\tau_n\}$ be a $w$ pre-chaotic
  sequence, and let $w_n$ be defined in terms of $\tau_n$ as in
  \eqref{wdef}. Then  
\begin{equation}\label{w5}
\lim_{n\to\infty}\int_0^1 |w_n(\xi) - w(\xi)|{\rm d}\xi = 0\ .
\end{equation} 
\end{lm}

\begin{proof}
Pick $\epsilon>0$, and choose $0 <\xi_* < 1$ such that $\int_{\xi_*}^1
w(t){\rm d}t  < \epsilon$.   By hypothesis   $w$ is continuous
on $[0,\xi_*]$, and for all $n > n_\epsilon$, 
\begin{eqnarray*}
\int_0^{\xi_*} |w_n(\xi) - w(\xi)|{\rm d}\xi  &=& \sum_{k <
    n\xi_*}\int_{(k-1)/n}^{k/n}|w(k/n) + r(k,n) -   w(\xi)|{\rm d}\xi\\ 
&\leq& \sum_{k < n\xi_*}\int_{(k-1)/n}^{k/n}|w(k/n)  - w(\xi)|{\rm
       d}\xi  +\epsilon\\
\end{eqnarray*}
If $\omega$ denotes the modulus of continuity of $w$ on $[0,\xi_*]$, 
$$|w(k/n)  - w(\xi)| \leq \omega(1/n) \quad{\rm on} \quad \left[\frac{k-1}{n}, \frac{k}{n}\right]\ .$$
Thus  
$$
\int_0^{\xi_*} |w_n(\xi) - w(\xi)|{\rm d}\xi    \leq \omega(1/n) +\epsilon $$
 for all sufficiently large $n$.  By Lemma~\ref{cut}, for all sufficiently large $n$,
 $$
 \int_0^1 |w_n(\xi) - w(\xi)|{\rm d}\xi   \leq  2\omega(1/n) +4\epsilon  \ .
 $$
 Since $\epsilon>0$ is arbitrary, the lemma is proved.
\end{proof}

\subsection{Chaotic sequences of probability measures on $\SEnes$}

In this section we prove the following:

\begin{thm}\label{mainX}  Let $w$ be a  probability density on $[0,1]$
  that is continuous on $[0,1[$, and let $\{\tau_n\}$  be a $w$
  pre-chaotic sequence of probability densities on $S_1$.  Fix a
  sequence of energies $\{E_n\}$ with $\alpha_n= \epsilon n (n-1)/E_n
  \rightarrow \alpha$, 
and define the maps $T_n$ in terms of $\alpha_n$.  Let $\hat\tau_n$
denote the push forward of $\tau_n$ onto $\SEnes$, averaged over
permutations. Let $\{\mu_n\}$ be the sequence of empirical measures
associated to  
$\{\hat\tau_n\}$. Then 
\begin{equation}
  \label{fk95}
\lim_{n\to\infty}W_1(\mu_n,g(x){\rm d}x) = 0,
\end{equation}
where $g$ is a probability density on $\R_+$ related to $w$ as
follows: Define the increasing function $\phi$ on $[0,1]$ by 
\begin{equation}
  \label{fk46}
\phi(\xi) = (1-\alpha/2)\int_0^\xi \frac{w(t)}{1-t}{\rm d}t  + \alpha \xi 
\end{equation}
and then 
\begin{equation}\label{mainX2}
g(x) = \frac{1}{\phi'(\phi^{-1}(x))}\ .
\end{equation}
\end{thm}

Theorem~\ref{mainX} gives conditions for  $\{\hat{\tau}_n\}$ to be
$(\alpha,g)$ chaotic for a probability density $g$ on $\R_+$ that is
determined by $\alpha$ and $w$.   Notice that as long as 
  $w(1) \neq 0$ we have $\lim_{\xi\to1}\phi(\xi) = \infty$, and if in
addition  $\alpha >0$ 
or if $w$ does 
not vanish on any interval, then $\phi$ is strictly increasing, so
that $\phi$ is invertible from $[0,1]$ to $[0,\infty[$, and evidently
it is differentiable.  It is also  possible to invert the relation
between $g$ and $w$, so that given an appropriate density $g$, one can
find the $w$ for which \eqref{fk46} and \eqref{mainX2} yield $g$: 

\begin{thm}\label{minX} Let $\alpha \in ]0,2[$. Let $g(x)$ be a
  probability density on $\R_+$ such that   
\begin{equation}\label{fk44}
g(x) < \frac{1}{\alpha} \ {\rm a.s.}\quad{\rm and}\quad \int_0^\infty
xg(x){\rm d}x = 1\ . 
\end{equation} 
Let $G(x) = \int_0^x g(t){\rm d}t$ denote the distribution function of $g$, and for $\xi\in [0,1]$ define 
\begin{equation}\label{fk45}
w(\xi) := \frac{1}{1-\alpha/2}\left(\frac{1}{g(G^{-1}(\xi))} -\alpha \right)(1-\xi)\ .
\end{equation} 
Then $w$ is a probability density on $[0,1]$, and with $\phi$ defined
as in~(\ref{fk46}) 
\begin{equation}\label{fk47}
g(x) = \frac{1}{\phi'(\phi^{-1}(x))}\ .
\end{equation} 
and
\begin{equation}\label{fk47V}
\frac{\alpha g(x)}{1- \alpha g(x)} =  \frac{2\alpha}{2- \alpha} \frac{1 - \phi^{-1}(x)}{w(\phi^{-1}(x))}\ .
\end{equation} 

\end{thm}

\begin{proof} We compute, using the change of variables $x := G^{-1}(\xi)$, 
\begin{eqnarray}
(1- \alpha/2)\int_0^1w(\xi){\rm d}\xi &=&   \int_0^\infty (1- \alpha g(x))(1- G(x)){\rm d}x\nonumber\\
&=&  \int_0^\infty (1- G(x)){\rm d}x  -\alpha  \int_0^\infty  \alpha g(x)(1- G(x)){\rm d}x\nonumber\\
&=&  \int_0^\infty x g(x){\rm d}x - \frac12 \alpha\ .
\end{eqnarray}
Thus, whenever, $g(x) < 1/\alpha$ almost everywhere and $\int_0^\infty xg(x){\rm d}x = 1$,  $w(x)$ is a probability density on $[0,1]$.

With this choice of $w(\xi)$ in \eqref{fk46}, we find
\begin{equation}\label{fk42}
\phi(\xi) = \int_0^\xi \left(\frac{1}{g(G^{-1}(t))} - \alpha \right){\rm d}t + \alpha \xi  =  \int_0^\xi \frac{1}{g(G^{-1}(t))}{\rm d}t\,. 
\end{equation}
It follows that $\phi'(x) = 1/g((G^{-1}(\xi)) = (G^{-1}(\xi))'$ and then since $\phi(0) = G^{-1}(0) = 0$, $\phi(\xi) = G^{-1}(\xi)$. Thus, $G(x) = \phi^{-1}(x)$, and  \eqref{fk47} is valid.

Finally, by \eqref{fk47},
$$
\frac{\alpha g(x)}{1- \alpha g(x)}  =  \frac{\alpha}{\phi'(\phi^{-1}(x))- \alpha}\ ,
$$
and then since ${\displaystyle \phi'(\xi) = (1- \alpha/2)\frac{w(\xi)}{1-\xi} + \alpha}$, \eqref{fk47V} follows. 
\end{proof}

As an example, consider $g(x) = e^{-x}$, which satisfies \eqref{fk44} as long as $\alpha \leq 1$. Then $G(x) = 1- e^{-x}$, and then $G^{-1}(\xi) = -\log(1-\xi)$.
Therefore,
$$
w(\xi) = \frac{1}{1-\alpha/2}(1- \alpha(1-\xi))\ ,
$$
which is bounded on all of $[0,1]$. By Theorem~\ref{mainX} and Theorem~\ref{minX}, for all $\alpha\leq 1$, there exists a  $(\alpha,g)$-chaotic sequence.

We now prepare to prove Theorem~\ref{mainX}. The first step is to encode the empirical distribution into a random function as follows:
Define a random function $\psi_n : [0,1]\rightarrow \R^+$ by setting $x_0 = 0$, and then
$$\psi_n(\xi) := x_{k-1} \quad{\rm for} \quad \frac{k-1}{n} \leq \xi < \frac{k}{n}\ . \qquad 1 \leq k \leq n\ .
$$
Explicitly, 
\begin{align}
\label{eq:18}
    \psi_n(\xi) &= x_{\lfloor n\xi\rfloor} = \left(1-\frac{\alpha_n}{2}\right)
                  \sum_{j=1}^{\lfloor \xi n\rfloor } \frac{\tz_j}{
                  1-\frac{j-1}{n}}    + \frac{\alpha_n(\lfloor n
                  \xi\rfloor-1)_+}{n-1}\,,
\end{align}
where $\lfloor n \xi \rfloor$ is the largest integer $k$ such that
$k \leq n\xi $.  The point of the definition is this:  Let $\chi$ be
any $1$-Lipschitz function on $\R_+$ with $\chi(0) = 0$. Then on
account of \eqref{fk6}, $\chi$ is, with probability $1$,   integrable
with respect to the empirical distribution $\mu_n$, 
and one has
\begin{equation}\label{fk81}
\int_0^\infty \chi {\rm d}\mu_n =  \int_0^1 \chi(\psi_n(\xi)){\rm d}\xi + \frac1n \chi(x_n)\ .
\end{equation}
Define $\rho_n$ to be the push-forward under $\psi_n$ of the  uniform measure
on $[0,1]$,  so that  we can rewrite \eqref{fk81} as  
$$\mu_n = \rho_n + \frac1n \delta(x- x_n) -\frac1n\delta(x)\ .$$  Had
we used the ceiling function $\lceil \cdot \rceil$ in place of the
floor function $\lfloor \cdot \rfloor$, we would have had $\mu_n =
\rho_n$, and then we would have
$$
\int_0^1 \psi_n(\xi){\rm d}\xi = \frac1n \sum_{j=1}^n x_j = 1\ ,
$$
so that $\psi_n$ would be a random probability distribution.  This would be convenient, but
then some estimates that follow would be more
complicated. It is easy to estimate the small difference:

\begin{lm}\label{ceilfl}  We have
\begin{equation}\label{xntoz}
\lim_{n\to\infty}\frac1n \Ee[x_n] = 0\ ,
\end{equation}
and for all $\delta >0$,
\begin{equation}\label{fk93}
\lim_{n\to\infty}\Pe\{ W_1(\mu_n,\rho_n) > \delta\} = 0\,.
\end{equation}
\end{lm}

\begin{proof}
Since $\chi$ is $1$-Lipschitz with $\chi(0) = 0$, 
$$\left|\int_0^\infty \chi {\rm d}\mu_n - \int_0^\infty \chi {\rm d}\rho_n \right|
=\frac1n |\chi(x_n)| \leq \frac1n x_n\ ,$$ 
and hence
$W_1(\mu_n,\rho_n) \leq \frac{1}{n} x_n$.  Thus, once we have proved \eqref{xntoz}, \eqref{fk93} follows by Markov's inequality.

Now  \eqref{eq:15} yields
\begin{equation*}
 \Ee x_n
=   \Ee\left(n \left(1-\frac{\alpha_n}{2}\right) \left(\frac{\tz_1}{n}+
  \frac{\tz_2}{(n-1)}+\cdots+\frac{\tz_n}{1}\right) +
  \alpha_n\right)\ .
  \end{equation*}
Pick $0 <\xi_* < 1$, and split the sum into two pieces
$$
\sum_{k=1}^n \frac{\Ee[\tz_k]}{n-k+1}  =  \sum_{k\leq  \lfloor n\xi_*\rfloor } \frac{\Ee[\tz_k]}{n-k+1}  +   \sum_{k>  \lfloor n\xi_*\rfloor} \frac{\Ee[\tz_k]}{n-k+1}\ .
$$
The last term satisfies
$$
\sum_{k>  \lfloor n\xi_*\rfloor} \frac{\Ee[\tz_k]}{n-k+1}  \leq \sum_{k> \lfloor n\xi_*\rfloor } \Ee[\tz_k] \leq  \int_{\xi_*-1/n}^1  w_n(\xi){\rm d}\xi\ ,
$$
and by Lemma~\ref{L1lm}, for  and $\epsilon> 0$ sufficiently large
$n$, this is bounded above by  
$$
\int_{\xi_*- 1/n}^1  w(\xi){\rm d}\xi  + \epsilon
$$
uniformly in $\xi_*$. Because
$$
\lim_{\xi_*\uparrow 1} \int^1_{\xi_*}  w(\xi){\rm d}\xi\ =0\ .
$$
we can choose $\xi_* < 1$ so that 
$$
\sum_{k>  \lfloor n\xi_*\rfloor} \frac{\Ee[\tz_k]}{n-k+1}  < \epsilon$$ for all sufficiently large $n$.  Next, 
$$
 n\sum_{k\leq n\xi_*} \frac{\Ee[\tz_k]}{n-k+1} =    \sum_{k\leq
   n\xi_*} \frac{n\Ee[\tz_k]} {1-(k-1)/n}  \frac1n =   \sum_{k\leq
   n\xi_*} \frac{(w_n(k-1)/n)} {1-(k-1)/n}  \frac1n  
$$
By \eqref{wa5}, for all $\epsilon>0$ this is bounded by 
$$
\int_0^{\xi_*} \frac{w(t)}{1-t}{\rm d}t + 2\epsilon
$$
for all sufficiently large $n$.   Altogether, for all $\epsilon>0$ and all sufficiently large $n$, 
$$\frac1n \Ee[x_n] \leq (1- \alpha_n/2)\left(\frac{1}{n}\left(\int_0^{\xi_*} \frac{w(t)}{1-t}{\rm d}t + 2\epsilon + \alpha_n \right) + \epsilon\right)
$$
for all sufficiently large $n$. Since $\epsilon>0$ is arbitrary, this proves \eqref{xntoz}.
\end{proof}

Now let $n\to \infty$ with $\alpha_n \to \alpha$. We shall show below  that if  $\{\hat\tau_n\}$ is the push forward of a $w$-pre-chaotic sequence $\{\tau_n\}$  of probability densities on $S_1$, and $\mu_n$ is the corresponding sequence of empirical measures,
then along this limit,  the variance of $\psi_n(\xi)$  converges to zero, and moreover, its expectation
$\phi_n(\xi) := \Ee[ \psi_n(\xi)]$ converges to a limiting function
$\phi := \lim_{n\to\infty}\phi_n$.  In this case
$$
\lim_{n\to \infty} \int_0^\infty \chi {\rm d}\mu_n =  \int_0^1
\chi(\phi(\xi)){\rm d}\xi  =  \int_0^\infty \chi(x) f_\alpha(x){\rm
  d}x \ . 
$$
with convergence in probability,  where $f_\alpha(x) :=
1/\phi'(\phi^{-1}(x))$, as in Theorem~\ref{main}.

Computing the expectation of $\psi_n(\xi)$, we see that
\begin{align}
\label{eq:19}
\phi_n(\xi) &= \left(1-\frac{\alpha_n}{2}\right) \sum_{j=1}^{\lfloor \xi n\rfloor }
              \frac{\Ee[\tz_j]}{ 1-\frac{j-1}{n}}    + \frac{\alpha_n (\lfloor
              n \xi\rfloor-1)_+}{n-1} \nonumber \\
            &= \left(1- \frac{\alpha_n}{2}\right)
              \frac{1}{n}\sum_{j=1}^{\lfloor \xi n\rfloor}
              \frac{w_n(j/n)}{1-\frac{j-1}{n}} +\frac{(\lfloor \xi
              n\rfloor-1)_+}{(n-1) \xi} \alpha_n \xi \ .
\end{align}
For $\frac{j-1}{n} < t \leq \frac{j}{n}$
\begin{equation}\label{ffk53}
 \frac{1}{1-t} -  \frac{1}{n} \frac{1}{(1-t)^2} \leq
 \frac{1}{1-\frac{j-1}{n}}   \leq   \frac{1}{1-t} \ ,
\end{equation}
 and therefore
\begin{equation}\label{ffk52}
 \int_0^\xi \frac{w_n(t)}{1-t}{\rm d}t  -\frac1n  \int_0^\xi \frac{w_n(t)}{(1-t)^2}{\rm d}t \leq  \frac{1}{n}\sum_{j=1}^{\lfloor \xi n\rfloor} \frac{w_n(j/n)}{1-\frac{j-1}{n}}    \leq  \int_0^\xi \frac{w_n(t)}{1-t}{\rm d}t\ .
\end{equation}
Setting
\begin{equation}\label{eq:19B}
\widetilde{\phi}_n(\xi) := (1-\alpha_n/2)\int_0^\xi \frac{w_n(t)}{1-t}{\rm d}t  +\frac{(\lfloor \xi
              n\rfloor-1)_+}{(n-1) \xi}\alpha_n \xi   \geq \phi_n(x)\ ,
\end{equation}
we have
\begin{equation}\label{eq:19C}
|\phi_n(\xi) - \widetilde{\phi}_n(\xi)| \leq  (1-\alpha_n/2) \frac1n  \int_0^\xi \frac{w_n(t)}{(1-t)^2}{\rm d}t  \leq (1-\alpha_n/2) \frac{C}{n}\frac{1}{1-\xi}\ .
\end{equation}
Note also that 
\begin{equation}\label{ffk55}
 1- \frac{2}{\xi}\frac{1}{n-2} \leq \frac{(\lfloor \xi n\rfloor-1)_+}{(n-1)\xi}\le  1
\ ,
\end{equation}
 and therefore, if we assume that $\alpha_n \rightarrow \alpha$,
\begin{align}
  \label{eq:20}
\phi_n(\xi) &\rightarrow (1-\alpha/2)\int_0^\xi \frac{w(t)}{1-t}{\rm d}t
              + \alpha \xi  =: \phi(\xi)
\end{align}
when $n\rightarrow\infty$. 

We then
have from \eqref{ffk52}, \eqref{eq:19C} and \eqref{ffk55}    for
all $\xi$, 
\begin{eqnarray}\label{ffk56}
| \phi(\xi) -\phi_n(\xi)|  &\leq&   | \phi(\xi) -\widetilde{\phi}_n(\xi)|  +  |\widetilde{\phi}_n (\xi) -\phi_n(\xi)|\nonumber\\
&\leq& (1-\alpha/2) \int_0^\xi \frac{|w(t)- w_n(t)|}{1-t}{\rm d}t  + \left( 1 - \frac{(\lfloor \xi n\rfloor-1)_+}{(n-1)\xi}\right)\alpha  \nonumber\\\
&+& \widetilde{\phi}_n (\xi) -\phi_n(\xi)\ . 
\end{eqnarray}

Now define   $\nu_n$ to be the probability measure on $\R_+$ that is
the push-forward of the uniform probability measure on $[0,1]$  under
$\phi_n$, and let $\nu$ be determined by $\phi$ in the same way.  

\begin{lm}\label{detlem}  We have
\begin{equation}\label{philim}
\lim_{n\to\infty}\int_0^1 |\phi_n(\xi) -
\phi(\xi)|{\rm d}\xi  = 0\  .
\end{equation}
and
\begin{equation}\label{delm1}
\lim_{n\to\infty} W_1(\nu_n,\nu) =0\ .
\end{equation}
\end{lm}
\begin{proof}[Proof of Lemma~\ref{detlem}] Let  $\chi\in {\rm Lip}_1$. Then
$$\left| \int_0^1 \chi(\phi_n(\xi)){\rm d}\xi -  \int_0^1
  \chi(\phi(\xi)){\rm d}\xi \right| \leq \int_0^1 |\phi_n(\xi) -
\phi(\xi)|{\rm d}\xi$$ 
It remains to show  \eqref{philim}.
We estimate each of the terms coming from \eqref{ffk56}.

Suppose first  that $\alpha_n = \alpha$ for all $n$.   To bound the integral of  the first term on the right in \eqref{ffk56},
 change the order of integration:
\begin{eqnarray*}
\int_0^1\left( \int_0^\xi \frac{|w(t)- w_n(t)|}{1-t}{\rm d}t\right){\rm d}\xi &=& \int_0^1\left( \int_t^1 \frac{|w(t)- w_n(t)|}{1-t}{\rm d}\xi\right){\rm d}t\\
  &=& \int_0^1 |w(t)- w_n(t)|{\rm d}t\ ,
\end{eqnarray*}
and by Lemma~\ref{L1lm}, the right side vanishes in the limit $n\to\infty$.  Making the obvious addition and subtraction argument, we see that the same conclusion holds under the assumption that $\lim_{n\to \infty}\alpha_n = \alpha$. 

Next, to estimate $\int_0^1\left( 1 - \frac{(\lfloor \xi n\rfloor-1)_+}{(n-1)\xi}\right){\rm d}\xi$  we break the integral up into two pieces, and use \eqref{ffk55} away from $\xi =0$:
$$
\int_0^1\left( 1 - \frac{(\lfloor \xi n\rfloor-1)_+}{(n-1)\xi}\right){\rm d}\xi \leq \int_0^{1/n}1{\rm d}\xi  + \int_{1/n}^1 \frac{2}{\xi}\frac{1}{n-2}{\rm d}\xi  = \frac1n + \frac{2\log n}{n-2}\ ,
$$
and this too vanishes in the limit $n\to\infty$.

Finally, to estimate $\int_0^1(\widetilde{\phi}_n (\xi) -\phi_n(\xi)){\rm d}\xi$, we break the integral up into two pieces, but at the other end, and use  \eqref{eq:19C}:
$$
\int_0^{1-1/n}  (\widetilde{\phi}_n (\xi) -\phi_n(\xi)){\rm d}\xi  \leq (1-\alpha/2) \frac{C}{n}\int_0^{1-1/n} \frac{1}{1-\xi} {\rm d}\xi =(1-\alpha/2) \frac{C\log n}{n}
$$
while
\begin{eqnarray*}
\int_{1-1/n}^1  (\widetilde{\phi}_n (\xi) -\phi_n(\xi)){\rm d}\xi  &\leq&  \int_{1-1/n}^1  \widetilde{\phi}_n (\xi) {\rm d}\xi \leq (1-\alpha/2)\int_{1-1/n}^1 \log(1-\xi){\rm d}\xi  +\frac1n\alpha \\
&=& (1-\alpha/2)\frac{\alpha  + 1 + \log n}{n}\,.
\end{eqnarray*}

To pass to the general case,  let $\tilde \phi$ denote the function
$\phi$ with $\alpha$ replaced by some $\alpha_n\in ]0,2[$.  Then it is
easy to see that $\int_0^1 |\phi - \tilde \phi|{\rm d}\xi \leq
\frac32|\alpha - \alpha_n |$. 
Now  one more application of the triangle inequality yields
\eqref{philim} in general.
\end{proof}

\begin{lm}\label{detlm2} 
${\displaystyle \lim_{n\to\infty}\Pe\{ W_1(\rho_n, \nu_n) > \delta\}
  =0}$. Recall here that $\rho_n$ is the pushfoward of the
  uniform measure on $[0,1]$ by the map $\psi_n$ from eq.~(\ref{eq:18}).
\end{lm}.

\begin{proof}
Take $\chi\in {\rm Lip}_1$ and estimate
$$\left| \int_0^1 \chi(\psi_n(\xi)){\rm d}\xi -  \int_0^1
  \chi(\phi_n(\xi)){\rm d}\xi\right| 
 \leq  
\int_0^1 |\chi(\psi_n(\xi))- \chi(\phi_n(\xi))|{\rm d}\xi\ ,$$
uniformly in $\chi$, and hence 
${\displaystyle W_1(\rho_n,\nu_n) \leq \int_0^1 |\psi_n(\xi)-
  \phi_n(\xi)|{\rm d}\xi}$. 
By  Markov's inequality, for any $\delta>0$,
\begin{equation}\label{Mark}
\Pe\{ W_1(\rho_n, \nu_n) > \delta\} 
 \leq \frac1\delta\Ee\left( \int_0^1 |\psi_n(\xi)- \phi_n(\xi)|{\rm
     d}\xi\right)\ . 
 \end{equation}
 
 Now note that 
${\displaystyle
 \int_0^1 \psi_n(\xi){\rm d}\xi  =  \frac1n \sum_{j=0}^{n-1} x_j = 1 -
 \frac1n x_n}$, 
 and likewise 
 $$
 \int_0^1 \phi_n(\xi){\rm d}\xi  =  \frac1n \sum_{j=0}^{n-1} \Ee[x_j]
 = 1 - \frac1n \Ee[x_n] \,.
 $$
 Therefore,
 $$
 \left|  \int_0^1 \psi_n(\xi){\rm d}\xi   - \int_0^1 \phi_n(\xi){\rm d}\xi    \right|  \leq  \frac1n |x_n- \Ee[x_n]|\ .
 $$
 Then by Lemma~\ref{cut}, 
 $$
  \int_0^1 |\psi_n(\xi)- \phi_n(\xi)|{\rm d}\xi  \leq 2  \int_0^{\xi_*} |\psi_n(\xi)- \phi_n(\xi)|{\rm d}\xi  +2 \int_{\xi_*}^1 \phi_n(\xi){\rm d}\xi +  \frac1n |x_n- \Ee[x_n]|\ .
 $$
 Next, by Lemma~\ref{detlem}
 $$
 \limsup_{n\to \infty} \int_{\xi_*}^1 \phi_n(\xi){\rm d}\xi  \leq
 \int_{\xi_*}^1 \phi(\xi){\rm d}\xi  +  \limsup_{n\to \infty} \int_0^1
 |\phi_n(\xi) - \phi(\xi)|{\rm d}\xi  = \int_{\xi_*}^1 \phi(\xi){\rm
   d}\xi \ , 
 $$
 and 
 $\frac1n\Ee[|x_n - \Ee[x_n]|] \leq  \frac2n \Ee[x_n]$ which tends to
 zero by Lemma~\ref{ceilfl}. 
Therefore,
\begin{eqnarray*}
\limsup_{n\to\infty} \Ee\left[  \int_0^1  |\psi_n(\xi)- \phi_n(\xi)|{\rm d}\xi \right]  &\leq& 2  \limsup_{n\to\infty} \Ee\left[  \int_0^{\xi_*}  |\psi_n(\xi)- \phi_n(\xi)|{\rm d}\xi \right]\\
&+& 2  \int_{\xi_*}^1 \phi(\xi){\rm d}\xi\ .
\end{eqnarray*}
We next show that the first term on the right is zero. Pick $\epsilon>0$. Then by \eqref{w4}, there is a constant $C$ depending only on $\xi_*$ such that  for some $n_\epsilon$
$$
 \mathrm{Var}[\tz_i] \leq \frac{C}{n}\epsilon \qquad{\rm and}\qquad| {\rm Cov}(\tilde z_j,\tilde z_k)| \leq \frac{C}{n^2}\epsilon   \qquad{\rm for\ all}\quad n> n_\epsilon, \ j,k<  n\xi_* \ . 
$$

We then have from Eq.~(\ref{eq:18} and~(\ref{eq:19}) that for all $\xi < \xi_*$
\begin{align}
\mathrm{Var}[ \psi_n(\xi)] &=
                             \Ee\left[(\psi_n(\xi)-\phi_n(\xi)])^2
                             \right]  \nonumber \\
&= \Ee\left[ \left( \left(1-\frac{\alpha_n}{2}\right) \sum_{j=1}^{\lfloor \xi n\rfloor }
  \frac{\tz_j-\Ee[\tz_j]}{ 1-\frac{j-1}{n}}  \right)^2
\right]\nonumber\\
&= \left(1-\frac{\alpha_n}{2}\right)^2 \left( \sum_{j=1}^{\lfloor n \xi \rfloor}
  \frac{\mathrm{Var}[\tz_j]}{\left(1-\frac{j-1}{n}\right)^2}
+ \sum_{\substack{ j,k = 1\\ j\ne k}}^{\lfloor n \xi \rfloor}
  \frac{\mathrm{Cov}[\tz_j,\tz_k]}{\left(1-\frac{j-1}{n}\right)
  \left(1-\frac{k-1}{n}\right)}
 \right)\,.
\end{align}
Using the bounds on $\mathrm{Var}[\tz_j]$ and
$\mathrm{Cov}[\tz_j,\tz_k]$ from Eq.~(\ref{w4}), 
for all sufficiently large $n$,
\begin{align}
  \mathrm{Var}[ \psi_n(\xi)]
  &\leq \left(1-\frac{\alpha_n}{2}\right)^2
    \Bigg( \sum_{j=1}^{\lfloor n \xi \rfloor}
    \frac{C\epsilon}{n\left(1-\frac{j-1}{n}\right)^2}+   \nonumber \\
&\qquad\qquad\qquad\qquad  
\sum_{\substack{ j,k = 1\\ j\ne
  k}}^{\lfloor n \xi \rfloor}
  \frac{C\epsilon}{n^2\left(1-\frac{j-1}{n}\right)\left(1-\frac{k-1}{n}\right)} 
 \Bigg)\,.
\end{align}
The first of the terms in the parentheses is smaller than
\begin{align}
\epsilon C
  \int_{0}^{\xi} \frac{1}{(1-\xi)^2} {\rm d}\xi   = \epsilon C \frac{1}{1-\xi}  \,,
\end{align}
and the second is smaller than 
\begin{align}
\epsilon C\left(\frac{1}{n}\sum_{k=1}^{\lfloor n \xi \rfloor}
  \frac{1}{1-\frac{k-1}{n}}  \right)^2 \le \epsilon C \left(\log
  (1-\xi) \right)^2 \,,
\end{align}
where, as above, we have used  \eqref{ffk52} and its  analog for $(1-\xi)^{-2}$. It follows that for all $\xi < \xi_*$,
\begin{align}
\mathrm{Var}[\psi_n(\xi)] &\le \epsilon C((1-\xi)^{-1} + (\log(1-\xi))^2)\ .
\end{align}
Therefore,  for all $n> n_\epsilon$.
\begin{eqnarray*}
\Ee\left[ \int_0^{\xi_*} |\psi_n(\xi)- \phi_n(\xi)|{\rm d}\xi\right]  &=&   \int_0^{\xi_*}\Ee\left[ |\psi_n(\xi)- \phi_n(\xi)|\right] {\rm d}\xi\\
&\leq&   \int_0^{\xi_*}\left(\Ee\left[ |\psi_n(\xi)- \phi_n(\xi)|^2\right] \right)^{1/2}{\rm d}\xi\\
&\leq& \left(\epsilon C\left((1-\xi_*)^{-1} + (\log(1-\xi_*))^2\right)\right)^{1/2}\,.
\end{eqnarray*}
Since $\epsilon>0$ is arbitrary, this proves that 
$$
\limsup_{n\to\infty} \Ee\left[  \int_0^1  |\psi_n(\xi)- \phi_n(\xi)|{\rm d}\xi \right]  \leq  \int_{\xi_*}^1 \phi(\xi){\rm d}\xi
$$for all $0 < \xi_* < 1$. However, $\phi$ is integrable, we can choose $\xi_*$ to make this arbitrarily small. Thus, 
$$
\lim_{n\to\infty} \Ee\left[  \int_0^1  |\psi_n(\xi)- \phi_n(\xi)|{\rm d}\xi \right]  = 0\ .
$$
The main assertion now follows from \eqref{Mark}. 
\end{proof}

\begin{proof}[Proof of Theorem~\ref{mainX}]
By the triangle inequality,
$$W_1(\nu, \mu_n) \leq   W_1(\nu,\nu_n) + W_1(\nu_n,\rho_n) + W_1(\rho_n,\nu_n)  \ .$$
Now applying  Lemma~\ref{ceilfl}, Lemma~\ref{detlem} and Lemma~\ref{detlm2}  yields the result  yields \eqref{fk95}.
Then since 
$$
\lim_{n\to \infty}\mu_n([0,x]) = \lim_{n\to\infty}\int_0^1  1_{[0,x]}(\psi_n(\xi){\rm d}\xi  = \lim_{n\to\infty} \int_0^1 1_{[0,x]}(\phi(\xi){\rm d}\xi  = \phi^{-1}(x)\ ,
$$
the cumulative distribution function of the limiting empirical measure is $\phi^{-1}(x)$ and hence limiting empirical measure has the density 
\begin{align}
  \label{eq:28}
   g(x) &= \frac{d}{dx} \phi^{-1}(x) = \frac{1}{\phi'\left( \phi^{-1}(x)\right)}\,.
\end{align}
\end{proof}

\subsection{ Detailed  chaoticity of the equilibrium sequence}
\label{subsec:strongchaos}

As our first application of Theorem~\ref{mainX},  we identify the
limiting equilibrium density  $f_\alpha$, and  prove the 
  detailed $(\alpha,f_\alpha)$-chaoticity of the equilibrium
sequence: 

\begin{proof}[Proof of  Theorem~\ref{main}]
By Lemma~\ref{param}, the sequence of uniform probability measures on
$\SEnes$ are obtained by averaging over permutations the push-forwards
under the map $T_n$ described there of the flat Dirichlet measure on
the standard simplices of the same dimension. In this case we have at
fixed $n$ that $w_j = 1/n$ for all $j$, and $w(t) = 1$.  By
Remark~\ref{equil} , the sequence of flat Dirichlet measures on the
standard simplices is $w$-chaotic in the sense of
Definition~\ref{s1chaos} for $w =1$. 
For $w(t) =1$ for all $t$, $\phi(\xi)$ is given by \eqref{eq:20.0}, and
this identifies the limiting density $f_\alpha$.  

We next show that detailed  $(\alpha,f_\alpha)$ chaoticity holds for  the
sequence.  The  gap length  $\zeta_{x,n}= 
  x_{(j+1),n}- x_{j,n}- \alpha/(n-1)$ satisfies
  $$
 \frac{n-1}{\alpha}
 \zeta_{x,n}=\frac{2-\alpha}{2\alpha}\frac{(n-1)\tilde{z}_j}{1-
   \frac{j-1}{n}} \qquad{\rm where}\qquad j = \lfloor
 n\phi^{-1}(x)\rfloor\ . 
  $$
  Each $\tilde{z}_j$ has a ${\rm Beta}(1,n-1)$ distribution,  and
  hence the probability density for $(n-1)\tilde{z}_j$ is
  $\frac{n}{n-1}(1-z/(n-1))^{n-1}$ which converges to $e^{-z}$ as
  $n\to\infty$. Therefore
  \begin{align}
    \label{eq:79gapdistAB}
    \lim_{n\to\infty}\Pe[ (n-1) \zeta_{x,n}/\alpha > r]
    \rightarrow e^{-\frac{2\alpha(1-\phi^{-1}(x))}{2-\alpha} r}\,.
\end{align}
and by \eqref{fk47V}, this is equivalent to \eqref{eq:79gapdistA}. with $g = f_\alpha$.

The rate information is easily extracted from the Lemmas since all but
Lemma~\ref{detlem} give rates.
However, one can extract a rate from the proof of
Lemma~\ref{detlem} by considering the last two displayed inequalities
in the proof.  The details are left to the reader.  
\end{proof}

\subsection{Chaotic sequences by non-flat Dirichlet measures}
\label{sec:chaoticA}

The construction just provided leads to other chaotic  sequences of
probability measures on $\SEnes$:  Instead of pushing forward the flat
Dirichlet measure on $S_1$, we can push forward   
more general Dirichlet distributions, and as we show in this section
this leads to the construction of $(\alpha,g)$-chaotic sequences for
all probability densities $g$ on $\R_+$ that satisfy \eqref{fk8b} and
\eqref{fk8c}.  However, except in the case of the flat Dirichlet
measures, these sequences will not satisfy the detailed chaoticity. 

 Let $w = (w_1,\dots,w_n)$ be a probability measure on $\{1,\dots,
 n\}$. That is, $w_j \geq 0$ for all $j$ and $\sum_{j=1}^nw_j =1$.  
Then
\begin{equation}\label{dir1}
h_w(\tz_1,\dots\tz_n) = \left(\prod_{j=1}^n
  \Gamma(nw_j)\right)^{-1}\Gamma(n)\prod_{j=1}^n  \tz_j^{nw_j -1} \,.
\end{equation}
 is the density for a Dirichlet distribution on $S_1$ with
concentration parameters $w$. Random variables
$(\tilde{z}_1,...,\tilde{z}_n)$ with this distribution satisfy
\begin{align}
\label{eq:15b}
\Ee[\tz_j]&= w_j\,, \nonumber\\
\mathrm{Var}[\tz_i] &=  \frac{w_j(1-w_j)}{n+1}\,,  \\
\mathrm{Cov}[\tz_i,\tz_j] &=  -\frac{w_iw_j}{n+1}\,.  \nonumber
\end{align}

Now let $w(x)$ be a continuous probability density on $[0,1]$, and
suppose that for each $n$, we produce $(w_1,\dots,w_n)$ by taking
$w_j$ to be the mass  
assigned by $w(x){\rm d}x$ to the $j$th interval on in the uniform
partition of $[0,1]$.   Let $\tau_n$ denote the Dirichlet distribution
on the standard simplex in $n$ dimensions with the distribution given
by \eqref{dir1} and this choice of $(w_1,\dots,w_n)$. Then \eqref{w5}
holds on account of the continuity of $w$, and \eqref{w4} holds with
$C = \max_{\xi\in [0,1]}\{w(\xi)\}$ which is finite by the continuity
of $w$. Thus $\{\tau_n\}$ is a { $w$}-pre-chaotic sequence.

Now fix $\alpha\in ]0,2[$, and a sequence $\alpha_n \subset [0,2]$
with $\alpha_n \to \alpha$. Let $\hat{\tau}_n$ denote the probability
measure on $\SEnes$ obtained by pushing forward $\tau_n$ under the map
specified in Lemma~\ref{param} at the value $\alpha_n$, and then
averaging over permutations.  
By Theorem~\ref{mainX},  $\{\hat{\tau}_n\}$ is $(\alpha,g)$ chaotic where 
\begin{equation}\label{fk41}
g(x) = \frac{1}{\phi'(\phi^{-1}(x))}  \quad{\rm and}\quad  \phi(\xi) = (1-\alpha/2)\int_0^\xi \frac{w(t)}{1- t}{\rm d}t + \alpha \xi\ .
\end{equation}
Provided $\int_0^1(1-t)^{-1}w(t){\rm d}t = \infty$, $\phi$ increases strictly from $0$ at $\xi =0$ to $\infty$ at $\xi = 1$, and in fact, $\phi'(\xi) \geq \alpha$ for all $\xi$.   

 As before, let $\nu_n$ denote the push-forward of
  the uniform probability measure on $[0,1]$ under $\phi_n$, and let
  $\nu$ be defined in the same way in terms of $\phi$, so that $\nu =
  g(\xi){\rm d}\xi$ where $g(\xi) =1/\phi'(\phi^{-1}(\xi))$.  Thus we have: 

\begin{thm}\label{chacon} Let $w$ be a continuous probability density
  on $[0,1]$.  For each $n$ and each $1 \leq j \leq n$, define $w_j =
  \int_{(j-1)/n}^{j/n}w(\xi){\rm d}\xi$. Let the energies $E_n$ be
  chosen so that $\alpha_n \to \alpha$.  Equip $\SEnes$ with the
  probability measure   
that is the push-forward under $T_n$ (See Lemma~\ref{param}) of the
Dirichlet measure specified in \eqref{dir1} using these weights,
averaged under permutations.  Then this sequence is
$(\alpha,g)$-chaotic where $g$ is given by \eqref{fk41}. 
\end{thm} 

However, The chaotic sequences obtained in this manner do not satisfy
detailed chaos. In this construction, for each $n$ and
$x_{j(n)}$, Each $\tilde{z}_{j(x)}$ has a ${\rm Beta}(nw_{j(x)}
,n-nw_{j(x)})$ distribution,  and hence the probability density for
$(n-1)\tilde{z}_{j(x)}$ converges to a non-exponential Gamma
distribution unless $w_j \to 1/n$. To obtain sequece that satisfies
detailed chaos, we must push forward a different class of pre-chaotic
measures on the standard simplices. In the next subsection, we
describe one way of doing this.

\begin{remark}
  \label{rem:37}
  If the concentration parameters in the Dirichlet distribution are
  multiplied by a common factor   $K$, and hence eq.~(\ref{dir1}) is
  replaced by
  \begin{equation}\label{dir1b}
    h^K_w(\tz_1,\dots\tz_n) = \left(\prod_{j=1}^n
      \Gamma(K nw_j)\right)^{-1}\Gamma(K n)\prod_{j=1}^n  \tz_j^{K nw_j -1} \,,
\end{equation}
we still have a Dirichlet distribution with the same expected
  values but with the variance covariance multiplied by the factor
  $1/K$. Equation~(\ref{eq:15b}) becomes
  \begin{align}
    \label{eq:15bK}
    \Ee[\tz_j]&= w_j\,, \nonumber\\
    \mathrm{Var}[\tz_i] &=  \frac{w_j(1-w_j)}{K(n+1)}\,,  \\
    \mathrm{Cov}[\tz_i,\tz_j] &=  -\frac{w_iw_j}{K(n+1)}\,.  \nonumber
  \end{align}
  Therefore all estimates leading to the proof of Theorem~\ref{chacon}
  are still valid, and therefore these measures are
  $(\alpha,g)$-chaotic as well. But increasing $K$ implies that the
  random variables $\tz_i$ become more concentrated around their mean
  $w_i$. And while this does not change the limiting density $g$, it
  changes the gap distribution for all finite $n$, and we will see in
  Section~\ref{sec:kac} that this is a fundamental difference for the
  limiting dynamics of the particle system. 
\end{remark}

\subsection{Detailed chaoticity via order statistics}   

The construction that we now give uses another probability density $h(\eta)$ on $[0,1]$ that  has to do with the excess energy distribution.  

The distribution of the random  points $x_j\in \R_+$ is determined by the
distribution of empty intervals $]0,x_1[$, $]x_1+\epsilon, x_2[$, or equivalently, as we have seen in
 Lemma~\ref{param},  by the random variables $z_j$ in~(\ref{eq:2}). These specify a random partition
$$\{[0,a_1], (a_1,a_2],\dots,(a_{n-1}, 1]\}$$
of $[0,1]$ into $n$ parts with $a_j - a_{j-1} = \tz_j$ and $a_0 = 0$.
This random partition is closely related to a partition of the excess
energy. Recall that the fraction of the total energy that is excess
energy is $(1- \alpha/2)$.   
Given a probability density $g(x)$ on $\R_+$ that satisfies \eqref{fk44}, 
$(1- \alpha g(x))$ represents probability that the interval $[x,x+{\rm
  d}x]$ is unoccupied.  Opening up a gap in  $[x,x+{\rm d}x]$ would raise the energy
of all the particles with energy higher than $x$ by ${\rm d}x$. Thus
this would make a contribution of 
$$(1-\alpha g(x))(1- G(x))$$ to the total excess energy, where $G$ is
defined as in Theorem~\ref{minX}. Therefore,
the fraction of the excess energy that can be ascribed to gaps in
$[x,x+{\rm d}x]$ is
\begin{align}
\label{eq:hdx}
  h(x){\rm d}x = \frac{1}{1-\alpha/2}(1- \alpha g(x)) (1 - G(x)){\rm
  d}x\ .
\end{align}
One readily checks that $h(x)$ is indeed a probability density. Let $H(x)$ denote its cumulative distribution function.   Out of $G$ and $H$ we define two maps from $[0,1]$ to $[0,1]$, namely $G\circ H^{-1}$ and $H\circ G^{-1}$.
We may use these two maps to push forward the uniform distribution on $[0,1]$ onto $[0,1]$ itself, producing two new probability measures on $[0,1]$. 

Define
\begin{equation}\label{psiform}
\psi(\eta) =  \frac{g(H^{-1}(\eta))}{h(H^{-1}(\eta))}\ ,
\end{equation}
and note that the cumulative distribution function of $\psi$ is $\Psi(\eta) = G(H^{-1}(\eta))$.   Likewise define
\begin{equation}\label{wform}
w(\xi) =  \frac{h(G^{-1}(\xi))}{g(G^{-1}(\xi))}
\end{equation}
and note that the cumulative distribution function of $w$ is 
$H(G^{-1}(\eta))$.   Also, note that
\begin{equation}\label{wpsiform}
w(\xi) = \frac{1}{\psi(\Psi^{-1}(\eta))}\,.
\end{equation}

From \eqref{psiform}, and the definition of $h$ in terms of $g$, 
\begin{equation}\label{psiform2}
\psi(\eta) =  \frac{1}{1-\alpha/2}\left(\frac{1}{g(G^{-1}(\Psi(\eta)))} - \alpha\right)(1 - \Psi(\eta))\ .
\end{equation}
When there is an $\epsilon> 0$ such that $g(x) \leq (\alpha+ \epsilon)^{-1}$,
\begin{equation}\label{psilower}
\psi(\eta) \geq \frac{\epsilon}{1-\alpha/2}(1 - \Psi(\eta))\ ,
\end{equation}
and this provides a lower bound on $\psi$ on any interval $[0,\eta_*]$, for any $\eta_* < 1$. 

Likewise, we may recover $w(\xi)$ as in equation~(\ref{fk45})  from the formula for $h$.
When there is an $\epsilon>0$ so that $1/g(x) \geq \alpha +\epsilon$ for all $x$,
$$w(x) \geq \epsilon (1-\xi)\ ,$$
 which provides a uniform  lower bound on $w(\xi)$ on $[0,\xi_*]$ for any $\xi_* < 1$.
 We see that $w$ is the probability density on $[0,1]$  associated to
 $g$ though  Theorem~\ref{minX}.

Let $\Phi_j$, $j=1,2$ be the cumulative distribution functions of two strictly positive probability densities $\phi_j$, $j=1,2$ respectively, on intervals $[a_j,b_j]$, $j=1,2$.  Then  $\Phi_1^{-1}\circ \Phi_2: [a_2,b_2] \to [a_1,b_1]$, and as is readily checked, 
$\Phi_1^{-1}\circ \Phi_2$ pushes $\phi_2(x){\rm d}x$ onto $\phi_1(x){\rm d}x$. 
In particular, if $\phi_2(x){\rm d}x$ is the uniform distribution on $[0,1]$; i.e.; $\Phi_2(x) = x$ for all $x\in [0,1]$, $\Phi_1^{-1}$ pushes forward the uniform distribution on $[0,1]$ onto $\phi_1(x){\rm d}x$ on $[a_1,b_1]$. That is, for all continuous functions $\chi$ on $[a_1,b_1]$,
$$
\int_{a_1}^{b_1} \chi(x) \phi_1(x){\rm d}x = \int_0^1 \chi(\Phi_1^{-1}(y)){\rm d}y\ .
$$
In particular if $\xi$ is a random variable that is uniformly distributed on $[0,1]$, $\Phi_1^{-1}(\xi)$ is a random variable with the law $\phi_1(x){\rm d}x$.  Therefore,
if $\xi_1,\dots,\xi_{n-1}$  are the order statistics of $n-1$ i.i.d
uniformly distributed  random variables,  $\Phi^{-1}_1(\xi_1),\dots, \Phi^{-1}_1(\xi_{n-1})$ are the order statistics of $n-1$ independent samples from the law $\phi_1(x){\rm d}x$.

\begin{lm}\label{push}
Let $F(\eta)$ be the cumulative distribution function of a continuous  probability density
$f(\eta)$  on $[0,1]$ that is uniformly positive on $[0,\eta_*]$ for all $0 < \eta_* < 1$; i.e., for some $a>0$,  $f(\eta) \geq a$ for $0 \leq \eta \leq \eta_*$. 
Let $\xi_1,\dots,\xi_{n-1}$ be the order statistics of $n-1$ i.i.d
uniformly distributed  random variables, and set $\xi_0=0$,
$\xi_n=1$ and $\lambda_j=\xi_j-\xi_{j-1}$. Moreover let $\eta_j=
F^{-1}(\xi_j)$ and set 
$\tz_j=\eta_j-\eta_{j-1}$.  Then for $\eta_j,\eta_k \leq \eta_*$ 
\begin{align}
  \label{varvar1}
  \Ee[  \tz_j ] &= \Ee\left[\frac{\lambda_j}{f (\eta_j)}\right] + \frac1n r_1\,,\\
  \label{varvar2}
  {\rm Var}[\tz_j] &= \Ee\left[\left(\frac{\lambda_j}{f (\eta_j)}\right)^2\right] + \frac{1}{n^2}r_2\,,\\
  \label{varvar3}
  {\rm Cov}[\tz_j, \tz_{k}] &=    \frac{1}{n^2}r_3\,,                
\end{align}
where  $r_1, r_2,$ and $r_3$ converge to zero as $n\rightarrow\infty$
with a rate depending on the modulus of continuity of $f$ and the lower bound on $f$ on $[0,\eta_*]$. If
$f$ is Lipschitz continuous and uniformly bounded below on all of $[0,1]$, then $|r_i|< C/n$, with $C$ depending
on $f$. 
\end{lm}

\begin{proof}
Related results for order statistics can be found {\em e.g.}
in~\cite{Ahsanullah_etal_2013}. The present result is preciely adapted
for our situation.  We know that $\lambda_j=\xi_j-\xi_{j-1}$ are
distributed as a flat Dirichlet distribution, and hence that
$\Ee[\lambda_i]=1/n$, and that 
$ {\rm Var}[\lambda_j] = 1/n^2 + \bigoh(1/n^3)$ and ${\rm
  Cov}[\lambda_j,\lambda_k] = -1/n^3+\bigoh(1/n^4)$.

First, since 
${\displaystyle ( F^{-1})'(t) = \frac{1}{\ f( F^{-1}(t))}}$,
\begin{align}\label{tzform}
  \tz_j &= \int_{\xi_{j-1}}^{\xi_j} \frac{1}{\ f( F^{-1}(s))}\,ds  = \lambda_j
         \frac{1}{f(\eta_j)} + u_j \,,
\end{align}
where
\begin{align}
  u_j&:=  \int_{\xi_{j-1}}^{\xi_j}
            \left(\frac{1}{f( F^{-1}(s))}-\frac{1}{f( F^{-1}(\xi_j))}\right)\,ds\,. 
\end{align}
If $f(s)$ is Lipschitz continuous and bounded below by $a>0$ for $F^{-1}(\xi_j)\leq \eta_*$, then  
$$
\left|\frac{1}{ f( F^{-1}(s))}-\frac{1}{ f( F^{-1}(\xi_j))}\right| \leq \frac {\lambda_j^2}{a^2}\ ,
$$
Therefore ${\displaystyle \Ee[u_j] \le C \Ee[\lambda_j^2]\le \frac{1}{a^2n^2} }$. Otherwise, if
$f$ is only continuous, there is a function $\omega(\delta)>0$, the
modulus of continuity,  with
$\lim_{\delta\rightarrow 0} \omega(\delta) =0$ such that
$$\sup_{|t_1- t_2|<\delta}| f(t_1)- f(t_2)| \le \omega(\delta)\ .$$
Then, for $\delta$ small enough,  and $\eta_j \leq \eta_*$ and $f\geq a>0$ on $[0,\eta_*]$,
\begin{align}
  \Ee[ |u_j|] &=   \Ee[ |u_j| \one_{\lambda_j<\delta}] +
                         \Ee[ |u_j| \one_{\lambda_j\ge\delta}]
                       \nonumber \\
              &\le \Ee[\lambda_j \omega(\delta)  \one_{\lambda_j<\delta} ]+ \frac1a \Pe[
    \lambda_j>\delta] \nonumber \\
  &\le \frac{1}{n} \omega(\delta) +\frac{ 1}{a} \frac{\Ee[\lambda_j^2]}{\delta^2}  \nonumber\\
  &\le \frac{1}{n} \omega(\delta) +\frac{ 1}{a} \frac{1}{n^2 \delta^2}   \nonumber
  \,.
\end{align}  
Choosing $\delta = n^{-1/3}$, we find,
$$
r_1 \leq  \frac{1}{n} \left(\omega(n^{-1/3}) +\frac{ 1}{a }n^{-1/3} \right)\ .
$$
The proof of (\ref{varvar2}), which we omit,  is very
similar.  To estimate the covariance we write
\begin{align}
  \label{eq:cov24}
  {\rm Cov}[\tz_j,\tz_k]
    &= \Ee[ \tz_j\tz_k] - \Ee[\tz_j]\Ee[\tz_k] \,, 
\end{align}
and
\begin{align}
  \label{eq:cov25}
  \tz_j\tz_k
  &= \int_{\xi_{j-1}}^{\xi_j} \int_{\xi_{k-1}}^{\xi_k}
                 \frac{1}{ f( F^{-1}(s))} \frac{1}{ f( F^{-1}(t))}\, {\rm d}s {\rm d}t 
  \nonumber \\
  & = \lambda_j \lambda_k  \frac{1}{ f(\eta_j) f(\eta_k)}  \nonumber
  \\
  &\qquad + \int_{\xi_{j-1}}^{\xi_j}
    \left(\frac{1}{ f( F^{-1}(s))}-\frac{1}{ f(\eta_j)}\right)\,{\rm d}s       \int_{\xi_{k-1}}^{\xi_k}\frac{1}{ f( F^{-1}(t))}\,{\rm d}t  \nonumber\\
    & \qquad+     
    \lambda_j \frac{1}{ f(\eta_j)}
   \int_{\xi_{k-1}}^{\xi_k} \left(\frac{1}{ f( F^{-1}(t))}-\frac{1}{ f(\eta_k)}\right)\,{\rm d}t\,.   
\end{align}
For the first term we note that $\Ee[\lambda_j \lambda_k] = n^{-2} +
{\rm Cov}{[\lambda_j, \lambda_k]}$, and by estimates like the ones
used to estimate $r_1$, we find that the remaining terms are
$o\left(n^{-2}\right)$, or even $\bigoh\left(n^{-3}\right)$ if $ f(s)$ is
Lipschitz. Hence computing the covariance in~(\ref{eq:cov24}) by
taking the expectation of ~(\ref{eq:cov25}) and using~(\ref{varvar1})
yields~(\ref{varvar3}). 
\end{proof}

For the following theorem, recall the constraints for $g$ as stated in
equation~(\ref{fk44}), and the definition of $h$ in (\ref{eq:hdx}),
and that $G$ and $H$ are their cumulative distribution functions.

\begin{thm}
  \label{thm:chaosexp}
  Let $\psi(\eta)$ be defined by \eqref{psiform}, and let
  $\eta_{j}, \, j=1,...,n-1$ be the order statitics of $n-1$
  i.i.d. random variables with distribution
  $\psi(\eta){\rm d}\eta$, and let
  $\eta_0=0$, $\eta_{n}=1$. With $\tz_i=\eta_{i}-\eta_{i-1},\; i=1,...,n$
  this induces a measure on the standard $n-1$ dimensional
  simplex $S_1$, whose push-forward to $\SEnes$ is 
  $(\alpha,g)$-chaotic. Let $x>0$, and let $]x_{(j)}, x_{(j+1)}[$ be
  the random interval that contains $x$. If the density
  $\psi(\eta)$ is continuous, then the gap length  $\zeta_{x,n}=
  x_{(j+1)}- x_{(j)}- \alpha/(n-1)$ satisfies

  \begin{align}
    \label{eq:79gapdist}
    \lim_{n\to\infty}\Pe[ (n-1) \zeta_{x,n}/\alpha > r]
    \rightarrow e^{-\frac{\alpha 
    g(x)}{1-\alpha g(x)} r}\,.
  \end{align}
\end{thm}

\begin{proof}  By Lemma~\ref{push} applied with $f =\psi$, together
  with \eqref{wpsiform}, which implies that 
$$
 \frac{1}{ \psi(\eta_j)}  = w(\xi_j)\ .
$$ This permits us to rewrite \eqref{varvar1},  \eqref{varvar2} and
\eqref{varvar3}  in terms of $w$, we see that  
  the sequence of laws of
$(\tilde z_1, \dots,\tilde z_n)$, averaged over permutations, are  $w$
pre-chaotic family on $S_1$. Here $w$ is given as in (\ref{fk45}).
Then  by Theorem~\ref{mainX}, if we fix a sequence of energies
$\{E_n\}$ with $\alpha_n= \epsilon n (n-1)/E_n \rightarrow \alpha$, 
and define the maps $T_n$ in terms of $\alpha_n$ as in
Lemma~\ref{Tndef}, the sequence of their push forwards onto
$\SEnes$, averaged over permutations, is $(\alpha, g)$-chaotic. 
The statement about the gap distributions then follows from
\eqref{tzform} with $f(\eta) =
\frac{g(H^{-1}(\eta))}{h(H^{-1}(\eta))}$ so that 
$$f(\eta_j) = \frac{g(G^{-1}(\xi_j))}{h(G^{-1}(\xi_j))}\ .$$

\end{proof}

It follows from \eqref{wform} that $w(\xi){\rm d}\xi$ results from
pushing the excess energy distribution forward onto $[0,1]$ using  
the distribution function $G$.  That is $h(x) = w(G(x))g(x)$.  In
equilibrium,  this excess energy density is 
uniform; i.e., $w(\xi) =1$,   
and the approach to equilibrium for our process can be thought of as
the approach of the excess energy distribution to uniform. 
This is illustrated in Figure~\ref{fig:excess}, which shows the
cumulative excess energy for a couple of different densities $g(x)$,
and the excess energy per particle as a function of the position $x$
of a particle, $(1-\alpha g(x)) G(x)/g(x)$.

\begin{figure}[!h]
  \centering
  \includegraphics[width=0.5\textwidth]{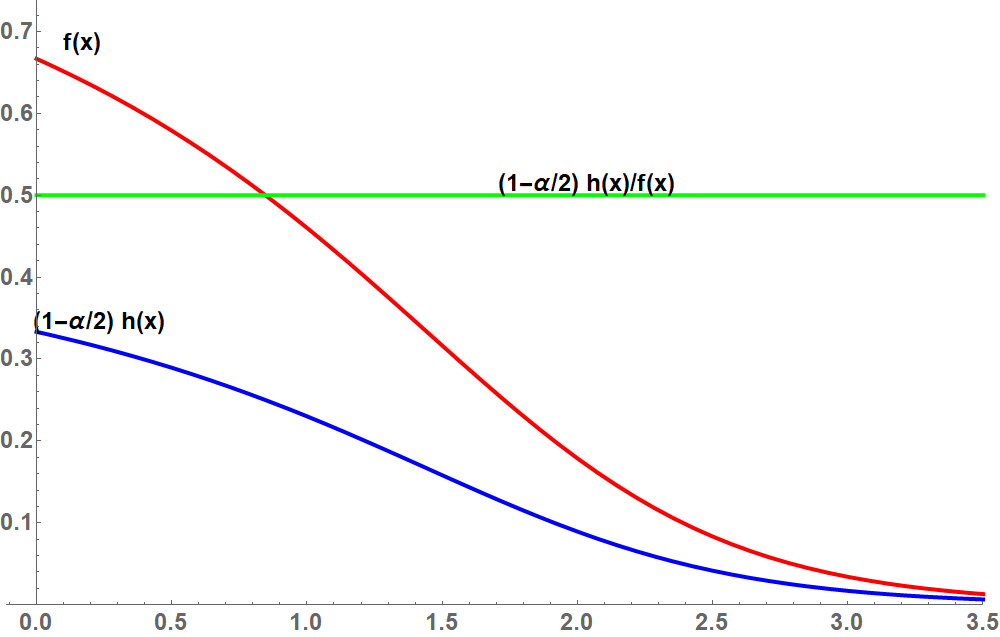}\;%
   \includegraphics[width=0.5\textwidth]{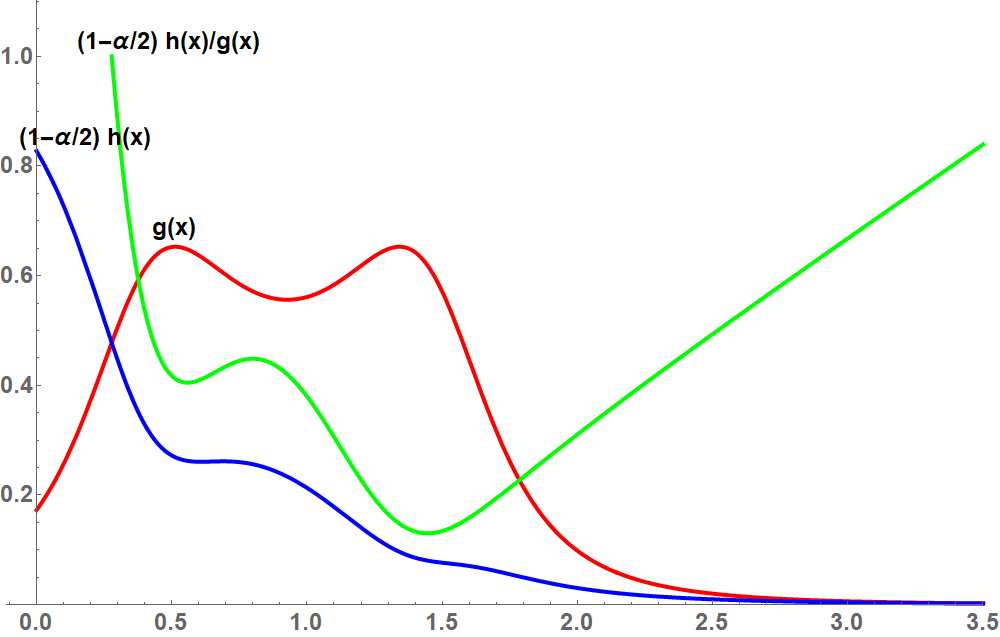}\;%
  \caption{\small The graphs show the particle density $g(x)$ (red),
    the corresponding  excess energy distribution $w(x)$ corresponding
    to $\alpha=1$ (blue), and
    the excess energy per particle, $w(x)/g(x)$ (green). To the left,
    the density is the equilibrium density $f_{\alpha}(x)$ as derived in Section
    2, and to the right $g(x) = c_1 /((1 + (c_2 x  - 1)^2) (1 + (c_2 x
    - 4)^2))$ with the constants $c_1$ and $c_2$ chosen to make $g(x)$
  a probability density with mean~$1$. We see that the density
  $f_{\alpha}(x)$ is equivalent to distributing the excess energy
  uniformly among the particles.}    
  \label{fig:excess}
\end{figure}

Lemma~\ref{push} and Theorem~\ref{thm:chaosexp} provide a means of
sampling the empirical distributions $\mu_n$: Let $\xi_{(i)},...,\xi_{(n-1)}$ be the order statistics of
  $n-1$ independent samples from the uniform distribution on $[0,1]$
 and then form
$\eta_1,\dots,\eta_n$ through  $\eta_j = \Psi^{-1}(\xi_{(j)})$, which
then gives us the order statistics of $n-1$ independent samples from
$\psi(\eta){\rm d}\eta$. Then with $\eta_0 =0$ and $\eta_n = 1$, we
define $\tz_j = \eta_j - \eta_{j-1}$, from which we recover a sample
of $(x_1,\dots,x_n)$.  
We illustrate this way of sampling the empirical distribution in
Figure~\ref{fig:chsampl1}, where $g$ is as in Figure~\ref{fig:excess},
but with $\alpha=1.5$. Here the density is close to the maximal
density $2/3$, which leads to slow convergence of the empirical
measures.

\begin{remark}
    The construction of the uniform measure on the simplex $S_1$ by
    independent i.i.d uniform random variables on the unit interval
   is not new here, and can be found for example
   in~\cite{VershikYakubovich2003}, where the authors study the
   asymptotics of partitions of a number  $n$ into a sum of $m$
   integers. 
  \end{remark}

  \begin{remark}
This construction also illustrates the difference between
$(\alpha,g)$-chaoticity and detailed $(\alpha,g)$-chaoticity, and why
detailed $(\alpha,g)$ chaos is difficult to express in terms of
marginal distributions in the way Kac defined chaos (see
eq.~(\ref{eq:kacchaos})). Assuming that $g$ 
is continuous as before, we have a one to
one map between a point $(x_1,...,x_n)\in\SEne^*$  and points
$\xi_j = \Psi(\eta_j)$ where $\eta_j=\sum_{i=1}^j \tilde{z}_i \in
[0,1]$, with $\eta_n=1$. After symmetrization the  push forward of a measure
$\sigma_{n,\epsilon}$ on $\SEne$ by this map  gives rise to a  symmetric measure
$\lambda_n$ on $[0,1]^n$, and one could compute the marginal
distributions of $\sigma_n$ and  $\lambda_n$ and see that if the sequence
$\{\sigma_n\}$ is $(\alpha,g)$ chaotic, then $\{\lambda_n\}$ is
chaotic with respect to the uniform measure on $[0,1]$, and the other
way around, and the fact that $(\alpha,g)$-chaoticity of the sequence
$\{\sigma_{n,\epsilon}\}$ corresponds to the usual notion of chaos $\{
\lambda_n\}$ is encoded in the maps $T_n$. Detailed chaos can be
expressed as saying that at the scale $1/n$, the points $x_j$ after
symmetrization behave as a Poisson point process: take any point
$\bar{\xi}\in ]0,1[$, and $a,b>0$. Then number of points $\xi_j\in
]\bar{\xi}-a/n,\bar{\xi}+b/n[$ will converge to a Poisson distribution
with parameter $b-a$ when $n\rightarrow\infty$. Hence detailed chaos
says in a sense that the set of points $x_1,...x_n$ in the limit are
as random as possible, given the constraint that their laws
$\sigma_{n,\epsilon}$ are $(\alpha,g)$-chaotic. 
  \end{remark}

\begin{figure}[!h]
  \centering
  \includegraphics[width=0.5\textwidth]{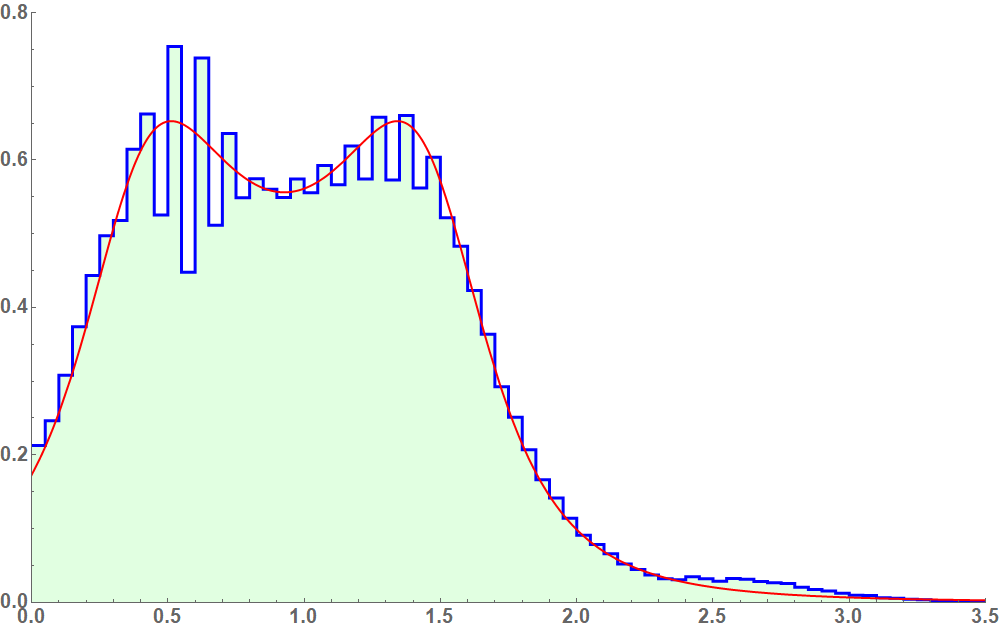}\;%
   \includegraphics[width=0.5\textwidth]{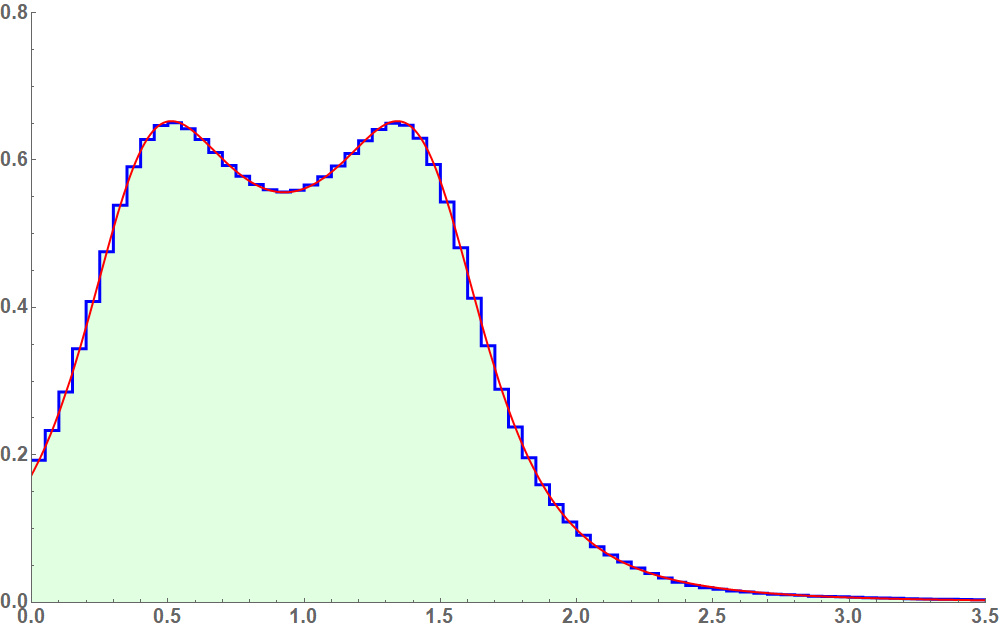}\;%
  \caption{\small The graphs show the particle density $g(x)$ (thin, red),
    and empirical histograms of samples (blue, thick) as defined in
    Theorem~\ref{thm:chaosexp}.
    To the left $n=50$, and to the right $n=1000$. The number of
    independent samples is 5000 and the bin width in the histograms is $0.05$.}
  \label{fig:chsampl1}
\end{figure}

\section{The Kac process}
\label{sec:kac}

\subsection{Specification of the Master Equation}

Kac's result~\cite{Kac1956} for the original Kac model is  that
propagation of chaos as described in 
the introduction is sufficient to identify an evolution equation for
the limiting densities, the Kac-Boltzmann equation. When the jumps
are constrained by the exclusion principle the situation is more subtle,
and propagation of chaos according to the definition~\ref{def1} is not
enough to identify a limiting equation. In this section we will
present the Kac-process, and derive a limiting Kac-Boltzmann equation
that is valid under the assumption of chaos according to
Definition~\ref{def2}, with an 
exponential gap distribution as in Theorem~\ref{thm:chaosexp}. 

The jump process is then as follows: With
$x=(x_1,...,x_n)\in \SEnes$,

\begin{itemize}
\item[(1)] pick a random waiting time $t$, exponentially distributed with
  rate $n$
\item[(2)] pick $1\le j < k \le n$ uniformly among possible pairs, and
  let $\bar{x}_{j,k}=\frac{x_j+x_k}{2}$. 
\item[(3)]  Let 
$(x_j^*,x_k^*)=\left(\bar{x}_{j,k}(1-\xi),\bar{x}_{j,k}(1+\xi)\right)$,
where $\xi$ is chosen uniformly in the $[-1,1]$.

\item[(4)] If $(x_1,....,x_j^*,...,x_k^*,...,x_n) \in \SEnes  $, then let $x^*=
  (x_1,....,x_j^*,...,x_k^*,...,x_n)$, else do nothing, {\em i.e.} let
  $x^*=x$ 
\end{itemize}
Note that the distribution of two particle energies after a collision
would be exactly the same if the step (3) ware replaced by

\begin{itemize}
\item[(3b)]   Pick   $\xi$  randomly from  $[-1,1]$, and let
  \begin{align*}
    x_j^* &= x_j + \xi \bar{x}_{j,k}\quad  (\hspace{-1ex}\mod x_j+x_k)\\
    x_k^* &= x_k - \xi \bar{x}_{j,k} \quad  (\hspace{-1ex}\mod x_j+x_k)\,.
  \end{align*}
  where $\mod$ here simply means that if $x_j+\xi
  \bar{x}_{j,k}>x_j+x_k$, then $(x_j+x_k)$ is added or subtracted to map
  $x_j^*$ back into the interval $0\le x_j^*< (x_j+x_k)$.
\end{itemize}
This collision process is reversible for any fixed $\xi$, and can also
be naturally generalized to collision models which favor small energy
exchanges in the collision, or for ``grazing collision limits'', which
are interesting in the classical setting.

However, for the purpose of writing down the generator of the process,
the version as originally described is simplest.   
Let $\mathcal{L}$ denote the generator. Then for any continuous function $F$ on $\SEnes$
\begin{align}\label{gen}
\mathcal{L}F(x) =  \frac{1}{n-1}\sum_{j<k} \int_{-1}^1
  1_{\SEnes}(x^*_{i,k,\xi})[F(x^*_{j,k,\xi}) - F(x)]{\rm d}\xi  
\end{align}
where $x= (x_1,\dots,x_n)$,  and $x^*_{j,k,\xi}  =
(x_1,....,\bar{x}_{j,k}(1-\xi)
,...,\bar{x}_{j,k}(1+\xi),...,x_n)$. Recall that the process satisfies
detailed balance, and is reversible, so that the operator $\mathcal{L}$ on the $L^2$ space given by the invariant measure is self-adjoint. 
The Kolmogorov forward equation, or what is the same thing, the Master
Equation, of the process is then

\begin{align}\label{Mas}
\frac{\partial}{\partial t}F(x,t) = \mathcal{L}F(x,t)\ .
\end{align}
Let ${\mathcal P}_t$ denote the semigroup associated to \eqref{Mas},
so that if $F(x,t)$ denotes the solutions with initial data $F(x,0)$,  
 $F(x,t) = {\mathcal P}_tF(x,0)$. 

\subsection{The exclusion factor}

To compute exactly the probability that the outcome from step (3)
results in a jump as defined in (4) is difficult, but it is possible
to derive formula for the limit as $n\rightarrow\infty$ under the
assumption that the limiting gap distribution is known and that the
events that $x_j^*$ and $x_k^*$ are admissible positions for particles
are independent. We also assume here that the density $g(x)$ is continuous.

First, to see why  propagation of chaos in the sense of Kac is not
enough to identify a limiting  equation we compare two different
chaotic sequences that are $(\alpha,f_{\alpha})$-chaotic, where
$f_{\alpha}$ is the equilibrium  density as found in
Theorem~\ref{main}. We take the empirical measures 
with the $x_j$ defined as in~(\ref{Tndef}), and $\alpha_n=\alpha$ for
simplicity. On the other hand taking $\tilde{z}_j=1/n$, for
$j=1,...,n$ provides another chaotic sequence. In this latter sequence
the gaps between particles are deterministic,
$x_{j+1}-x_{j}=(1-\alpha/2)/(n-j)$, and this means that to fit a new
particle of size $\alpha/(n-1)$ into an interval we must have
$\frac{j}{n} > 1-\frac{2-\alpha}{2\alpha}$, which is positive when
$\alpha>2/3$, and therefore for all $x$ smaller than
\begin{align*}
  x_j &\ge  \left(1-\frac{\alpha}{2}\right) \sum_{k=1}^{\lfloor
        1-\frac{2-\alpha}{2\alpha} 
  \rfloor}\frac{1}{n+1-k} + \frac{3-\alpha}{2}
        -\frac{\alpha}{n}
\end{align*}
which converges to 
\begin{align*}
   & \bar{x}_{\alpha}= \log\left(\frac{2\alpha}{2-\alpha}\right)
     +\frac{3\alpha-2}{2}\,.   
\end{align*}
when $n\rightarrow\infty$. So if $\alpha>2/3$ this
$(\alpha,f_\alpha)$-chaotic sequence does not 
allow any jump into an interval $[0, \bar{x}_{\alpha}]$. On the other hand, for
the sequence constructed  in Section~2, where the
$(\tilde{z}_1,...,\tilde{z}_n)$ taken from the flat Dirichlet
distribution, the $\tilde{z}_j$ are close to being exponentially
distributed with mean $1/n$. Hence for all $j$ there is a positive
probability that the $j$-th gap is bigger than $\alpha/(n-1)$, and
therefore jumps are possible to any point in the interval
$[0,\infty[$, although the probability will be very small in intervals
near the origin if $\alpha$ is large.

In the following calculation we  neglect the probability
that $x_j^*$ belongs to one of the gaps created when $x_j$ and $x_k$
are lifted out, {\em i.e.} when the particles fall back into nearly
the same point as where they started. The probability that this
happens converges to zero at the order $1/n$, and can be neglected
unless the excess energy is very small. Also in this case, the effect
of such a jump on the density will be very small, and therefore to see
an effect of this one would need to consider the process over a very
long time scale. It could be interesting to study this situation in a
diffusive scaling, and to analyze as certain models for competing
particle systems and rank based interacting
diffusions~\cite{PalPitman2008,Shkolnikov2012,Reygner2015}.

For any $x$, consider  an interval $[x-\delta/2,x+\delta/2]$, where
$\delta$ is assumed to be small and eventually converging to 0. 
We will call a point $x_*$ in this interval admissible if it satisfies
the exclusion constraint, given the particles that are already present
in the interval. The expected number of
points $x_j$ belonging to this interval will be $m+1= n \int_{x-\delta/2}^{x+\delta/2}
g(y)\,{\rm d}y\sim n \delta \, g(x)$ due to the assumption that $g$ is
continuous. Now let $x-\delta/2<x_{(0)}<x_{(2)},...,<x_{(m)} <x+\delta/2$ be the positions of
the $m$ particles belonging to this interval, renumbered for
convenience here, and let $\zeta_{j+1}= x_{(j)}-x_{(j-1)} -\alpha/(n-1)$ be the 
gaps between particles. For fixed $\delta$ we have that $x_{(0)} \rightarrow x-\delta/2$ and
$x_{(m)} \rightarrow x+\delta/2$ in probability, and therefore the error in
considering only the interval $[x_{(0)},x_{(m)}]$ will vanish in the limit as
$n\rightarrow \infty$. For a given gap $\zeta_i$, the interval
available for putting a new particle $x_*$ is $(\zeta_i-\alpha/(n-1))
\one_{\zeta_i>\alpha/(n-1)}$. In a jump,  $x_*$ is chosen
  uniformly over any interval , and therefore
  \begin{align}
    \label{eq:probadmi}
    \Pe\big[ x_* \, \mbox{ is admissible }\;\big|\;  x_*\in
    ]x_{(0)},x_{(m)}[ \big]
    = \frac{\frac{1}{m}\sum_{i=1}^{m}(\zeta_i -\frac{\alpha}{n-1})
   \one_{\zeta_i>\alpha/(n-1)}   }{\frac{1}{m}(x_{(m)}-x_{(0)})} 
  \end{align}
which holds for any particle configuration, if the interval
$[x_{(0)},x_{(m)}] $ does not contain the to particles that are
selected for collision.     The probability that $x_*$
is admissible can now be computed by taking the expectation of the
right hand side of equation~(\ref{eq:probadmi}) with respect to the
other particles. To continue we make
the following assumption:
\begin{assumption}
  For $n\rightarrow\infty$, one may take
  $\delta=\delta_n\rightarrow 0$ such that $m\rightarrow\infty$ in
  probability, and such that $(n-1) \zeta_i /\alpha$ are
  asymptotically i.i.d with a density $\rho_x$.
\end{assumption}

This holds for the two constructions of chaotic sequences given in
Theorem~\ref{chacon} and Theorem~\ref{thm:chaosexp} if the density
$g$ is continuous. By the law of large numbers, the denominator of the righthand side of
equation~(\ref{eq:probadmi}) is asymptotically $\Ee[\zeta_j] +
\frac{\alpha}{n-1} \sim \frac{1}{(n-1) g(x)}$,
 and the enumerator is asymptotic to
\begin{align}
  \frac{\alpha}{n-1} \Ee\left[ \frac{n-1}{\alpha}\zeta_i -1\right] &=
      \frac{\alpha}{n-1} \int_{1}^{\infty}(s-1)\rho_x(s)\,ds\,,
\end{align}
and therefore, for any interval not containing $x_j$ or $x_k$ we have
\begin{align}
  \lim_{\delta\rightarrow 0}\lim_{n\rightarrow\infty}
           \Pe\big[ x_* \, \mbox{ is admissible }\;\big| x_*\in
  [x-\delta/2,x+\delta/2]\big] = \alpha g(x) \int_{1}^{\infty}
  (s-1)\,\rho_x(s)\,ds\,.  
\end{align}

For example, with the chaotic sequence from
Theorem~\ref{thm:chaosexp}, $\rho_x(s) = \frac{\alpha g(x)}{1-\alpha
  g(x)} \exp\left(-\frac{\alpha g(x)}{1-\alpha g(x)}s\right)$, and we
find in the limit that
\begin{align}
  \label{eq:exclexp}
  \Pe\left[ x_* \; \mbox{is admissible} \,\right]
  & = \frac{\alpha^2 g(x_*)^2}{1-\alpha g(x_*)}\int_{1}^{\infty}(s-1)\,
    e^{-\frac{\alpha g(x_*)}{1-\alpha g(x_*)} s}\,ds \nonumber \\
  &= (1-\alpha g(x_*)) \exp\left(-\frac{\alpha g(x_*)}{1-\alpha
    g(x_*)}\right)\,. 
\end{align}
Here we recognize the first factor $1-\alpha g(x_*)$ as the exclusion
factor in the Uehling-Uhlenbeck equation for discrete energy levels,
and the second exponential factor reflects the fact that the
continuous spacing of gaps is a less efficient use of the available
excess energy.

With the chaotic sequences constructed through the Dirichlet
distribution as in Theorem~\ref{chacon} the distribution of the
gaps are Beta-distributions, as shown in equation~(\ref{eq:15bK}),
which gives the density for the distribution of a gap in the partition
of excess energy as 
\begin{align}
  \label{eq:beta1}
  \frac{\Gamma(K n)}{\Gamma( K n w_j ) \Gamma( K n (1-w_j)}
       \tz ^{Kn w_j-1} (1-\tz)^{Kn(1-w_j)-1}\,,
\end{align}
and hence, because $\zeta_j =
(1-\alpha/2))\frac{\tz_j}{1-\frac{j-1}{n}}$ that
$s = (n-1) \zeta_j /\alpha$ has density
\begin{align}
  \rho_{x,n}(s)=  c s^{K n w_j} \left(1-\frac{s}{\lambda_j}\right)^{K
  n (1-w_j) -1}\,, 
\end{align}
where $c$ is a normalizing constant and $\lambda_j = \frac{(n-1) 2
  \alpha }{2-\alpha}\frac{1}{1-G(x)}$. Because $\frac{j}{n}\sim G(x)$
we have asymptotically
\begin{align}
  n w_j = n \int_{(j-1)/n}^{j/n} w(\xi)\,{\rm d}x \sim w(G(x))=
  \frac{2 \alpha}{2-\alpha} \frac{1-\alpha g(x)}{\alpha g(x)}
  \left(1-G(x)\right)\equiv w_{g(x),G(x)}\,, 
\end{align}
and $K n / \lambda_n \sim K\frac{2-\alpha}{2\alpha}
\left(1-G(x)\right))$\,.
It follows that
\begin{align}
  \rho_{x,n}(s) \rightarrow \rho_{g(x),G(x)}(s) = s^{ K w_{g(x),G(x)}
  -1} \exp\left(-K\frac{2-\alpha}{2\alpha} \left(1-G(x)\right)) s \right)\,.
\end{align}
One can now obtain a formula similar to equation~(\ref{eq:exclexp})
corresponding to the density $\rho_{g(x),G(x)}$.  The notable
difference is that with this density the probability that a point
$x_*$ is admissible asymptotically does not only depend on the
limiting density $g(x)$ but also on the cumulative distribution
function $G(x)$.

Hence, when analyzing the limiting behavior of the Kac process for
this $n$-particle system, it is important to take the gap distribution
into account. We formulate this asymptotic result for the exponential
gap distribution as a proposition:

\begin{prop}
  \label{prop:exp}
  Let $g(x)$ be a continuous probability density on $[0,\infty[$, and
  let $\left( (x_1,...,x_n)  \right)_{n=2}^{\infty}$ be a chaotic
  sequence constructed as in Theorem~\ref{thm:chaosexp}. There is a
  sequence $\delta_n\rightarrow0$ such that if $x_*$ is chosen
  uniformly in an interval $[x,x+\delta_n]$, then
  \begin{align}
    \lim_{n\rightarrow\infty} \Pe[ x_* \; \mbox{ is admissible }]
    &=\left( 1-g(x)\right) \exp\left( -\frac{\alpha g(x)}{1-\alpha g(x)}\right)
  \end{align}

\end{prop}

\subsection{The Boltzmann equation}

For the original Kac process, it is enough to prove propagation of
chaos to identify an equation that describes the evolution of a density in
the limit of infinitely many particles. Here the situation is more
complicated, because the asymptotic gap distribution is important. We
conjecture that the process defined here propagates chaos with
exponential gap distribution, but we do not have a proof. The
conjecture is supported by numerical simulations that are presented in
Section~\ref{sec:sim}. Under the assumption that detailed chaos is
propagated, it is then possible to write down the  
corresponding kinetic equation, and compare this with the
corresponding Kac and Uehling-Uhlenbeck equations; a formal proof
would follow along the same lines as Kac's original derivation.

\begin{thm}\label{mKac}  Suppose that the evolution specified by
  \eqref{Mas} propagates chaos with parameter $\alpha$, and that the
  asymptotic gap distribution is exponential as in
  Theorem~\ref{thm:chaosexp}. Then the 
  limiting empirical distribution $g_t$ evolves according  
\begin{equation}\label{MkacEq}
\frac{\partial}{\partial t} g(x,t) = Q [g](x,t)\ ,
\end{equation}
where
\begin{align}
  \label{eq:35}
  Q[g](x)= \frac{1}{2}\int_{0}^{\infty} \int_{-1}^1
     & \bigg(g(x')g(y') \Pi(\alpha g(x) )\Pi(\alpha g(y) )- \nonumber \\
  &  g(x)g(y) \Pi(\alpha g(x') )\Pi(\alpha g(y'))\bigg)\, {\rm d}\xi {\rm d}y\,,
\end{align}
\begin{equation}
    \label{eq:35b}
  \begin{split}
  x' &= (1-\xi) (x+y)/2\,,\\
  y' &= (1+\xi) (x+y)/2\, ,    
  \end{split}
\end{equation}
and
\begin{align}\label{exclu}
\Pi(u) 
  &=( 1-u) \exp\left(-\frac{u}{1-u}\right)\,.
\end{align}

\end{thm}

The function $\Pi(u)$ specifies the effects of the exclusion
constraint which slows down the evolution.   The function is plotted
here in figure~\ref{fig:2}, together with the function $\Pi(u)= 1-u$,
which is the corresponding factor in the Uehling-Uhlenbeck
equation. With the car parking analogy from the introduction, this
factor quantifies how much less efficient it is to let cars park at
will along a road compared to using fixed parking slots. In Kac's
original paper there is no exclusion factor, and 
in that case $\Pi(u)$ is constant, equal to 1.

\begin{figure}[H]
  \centering
  \includegraphics[width=0.6\textwidth]{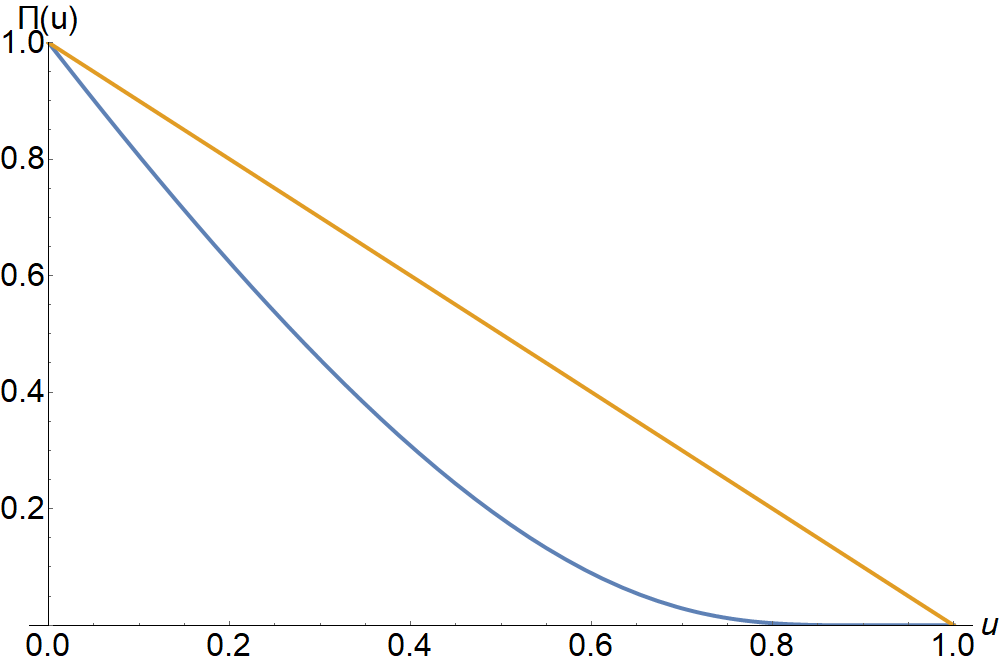}
  \caption{The exclusion factor as a function of $u$ (blue), compared
    with the fermionic factor $1-u$} 
 \label{fig:2}
\end{figure}

Note that $\alpha^{-1}$ is the maximum density possible, and hence
that $\alpha f(x)=1$ implies that the particles are  densely
packed near $x$. The exclusion factor reduces the effective jump rate
much more strongly than the usual factor $1-u$ from the Boltzmann
equation for Fermions, and is a significant difference between the
continuous setting that we study here, and the discrete, quantized
models. In fig.~\ref{fig:3} we plot the function $\Pi(\alpha f(x) )$,
{\em i.e.} the exclusion factor evaluated at the equilibrium density
as a function of the energy $x$, which indicates that particles will
very seldom get a new energy close to $x=0$, and therefore that the
rate of convergence to equilibrium could be very low.

\begin{figure}[H]
  \centering
  \includegraphics[width=0.7\textwidth]{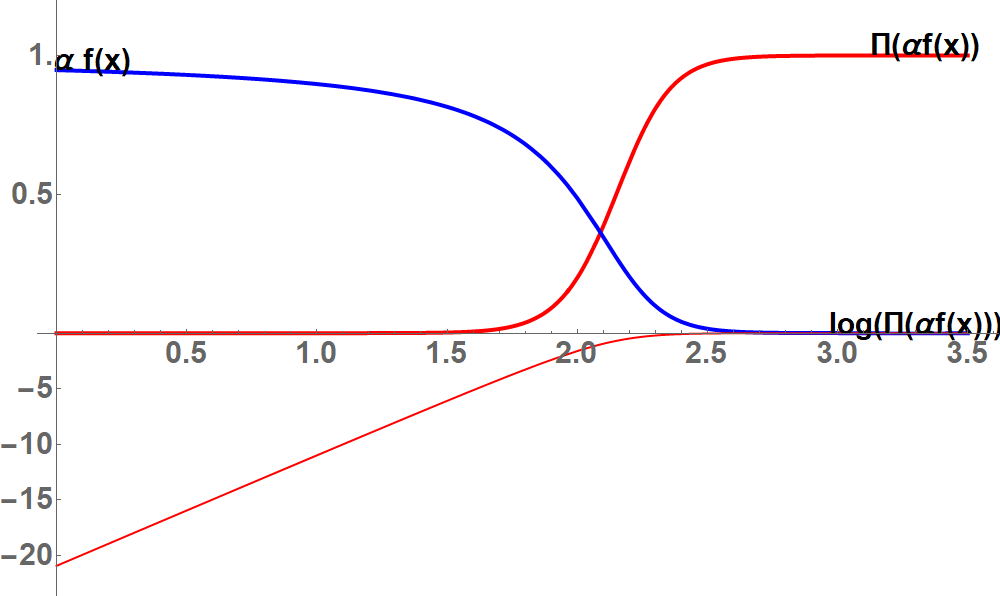}
  \caption{The exclusion factor  $\Pi(\alpha f_{\alpha}(x))$ for
    $\alpha=1.8$, as a function of  
    $x$. The thin red curve shows the logarithm of the same function,
    and the equilibrium distribution $f_{\alpha}(x)$ scaled by a factor
    $\alpha$ is given in blue.}
    \label{fig:3}
\end{figure}

\subsection{Properties of the collision operator}

The collision operator $Q[g]$ as defined in equation~(\ref{eq:35}) is
amenable to very much the same manipulations as the ordinary
collision operator for the Boltzmann equation, except that, in addition
to the mass, there is
only one conserved quantity, the energy.

\begin{thm}
  Let $Q[g]$ be defined as in equation~(\ref{eq:35}). Then the
  following holds:
  
  For any $a,b\in\R$, and any $g(x)$ satisfying
    $\int_0^{\infty} x g(x)\, {\rm d}x =1$
    \begin{align}
      \label{eq:thm44:1}
      \int_{0}^{\infty} (a+b x) Q[g](x)\,{\rm d}x =0\,,
    \end{align}

    Let  $f_\alpha(x)$ defined by equation~(\ref{fk94}) and
    (\ref{eq:20.0}). Then
    \begin{align}
    \label{eq:thm44:2}
        Q[f_{\alpha}](x) &= 0\,.
      \end{align}

      If  $g(x,t)$ is a solution to equation~(\ref{MkacEq}), then
      \begin{align}
        \label{eq:thm44:3}
      \frac{d}{dt}\int_{0}^{\infty} g(x,t) \log\left(\frac{\alpha
      g(x,t)}{1-\alpha g(x,t)}\right)\, {\rm d}x & \le 0\,.
    \end{align}
\end{thm}

\begin{proof}
  Let
  \begin{align}
    R(x',y',x,y) &= (g(x') g(y')\Pi(\alpha g(x)) \Pi(\alpha g(y))-
    g(x) g(y)\Pi(\alpha g(x')) \Pi(\alpha g(y'))\,.
  \end{align}
  Here $x'$ and $y'$ depend on a parameter $\xi$ as defined in
  equation~(\ref{eq:35b}). Formally, for any $h(x)$, a change of variables gives
  \begin{align}
    \label{eq:100}
    \frac{1}{2}\int_{0}^{\infty} \int_{0}^{\infty}\int_{-1}^{1}
    &  R(x',y',x,y)  h(x) \,{\rm d}\xi  {\rm d}x {\rm d}y = \nonumber \\
    & \int_{0}^{\infty} \frac{1}{u}\int_{0}^{u}
      \int_{0}^{u} 
      R(z, u-z, v, u-v) \,  h(v) {\rm d}v  {\rm d}z {\rm d}u\,.
  \end{align}
We see, just as for the usual Boltzmann equation, that $R(x',y',x,y)$
is symmetric with respect to the changes $(x,y)\to (y,x)$ and anti
symmetric with respect to changing $(x',y',x,y) \to (x,y,x',y')$, and
therefore the righthand side of equation~(\ref{eq:100}) is
\begin{align}
  \label{eq:101}
   \frac{1}{4}\int_{0}^{\infty} \frac{1}{u}\int_{0}^{u}
      \int_{0}^{u} 
     &  R(z, u-z, v, u-v) \, \left(  h(v)+ h(u-v) - h(z) - h(u-z)
  \right) {\rm d}v  {\rm d}z {\rm d}u =  \nonumber \\
& \frac{1}{8}\int_{0}^{\infty} \int_{0}^{\infty}\int_{-1}^{1}
    R(x',y',x,y) ( h(x)+ h(y) -h(x')-h(y') \,{\rm d}\xi  {\rm d}x {\rm d}y\,,
\end{align}
which implies~(\ref{eq:thm44:1}). To prove~(\ref{eq:thm44:2}), we
write $  R(x',y',x,y) $ as 
\begin{align}
  \label{eq:rdef2}
\alpha^2 g(x') g(y')g(x) g(y)\left(
             \frac{\Pi(\alpha g(x))}{\alpha g(x)}\frac{ \Pi(\alpha
                 g(y))}{\alpha g(y)}-
                   \frac{\Pi(\alpha g(x'))}{\alpha g(x')}\frac{ \Pi(\alpha
                 g(y'))}{\alpha g(y')}
   \right)\,.
\end{align}
Next we take $g(x)=f_{\alpha}(x)$ and define
\begin{align}
  \label{eq:103def}
  r(x) &= \log \frac{\Pi(\alpha f_{\alpha}(x))}{\alpha f_{\alpha}(x)}
         = -\log\left(\alpha f_{\alpha}(x) \right) +
             \log\left(1- \alpha f_{\alpha}(x) \right) - \frac{\alpha
         f_\alpha(x) }{1-\alpha f_\alpha(x) } \,.
\end{align}
Then
\begin{align}
  \label{eq:rprim}
  r'(x) &= - f_{\alpha}'(x)
          \left(\frac{1}{f_{\alpha}(x)} +
          \frac{\alpha}{1-\alpha f_\alpha(x)}
          + \frac{\alpha}{(1-\alpha f_\alpha(x))^2}
          \right) \nonumber \\
       &=    - f_{\alpha}'(x)\frac{1}{f_{\alpha}(x) (1-\alpha f_\alpha(x))^2}\,.
\end{align}
On the other hand $f_{\alpha}(x)$ satisfies
\begin{align}
  f_{\alpha}(x)&= \frac{1}{\phi'(F(x))}\,,
\end{align}
where $F(x) = \int_{0}^{x}f(y)\,{\rm d}y$ and 
\begin{align}
  \phi(\xi) &= \left(1-\frac{\alpha}{2}\right) \log\frac{1}{1-\xi} +
              \alpha \xi\,.
\end{align}
Therefore
\begin{align*}
  \frac{1}{f_{\alpha}(x)} &= \phi'(F(x)) =
                           \left(1-\frac{\alpha}{2}\right)
                           \frac{1}{1-F(x)} + \alpha\quad \mbox{and}
  \\
  f_{\alpha}'(x) &= -\frac{\phi''(F(x))}{\phi'(F(x))^2}f(x) =
         - \phi''(F(x)) f(x)^3\,,          
\end{align*}
which when inserted into~(\ref{eq:rprim}) gives
\begin{align}
  r'(x) &= \left(1-\frac{\alpha}{2}\right)
       \frac{1}{(1-F(x))^2} 
     \frac{1}{\left(\frac{1}{f_{\alpha}(x) }-\alpha \right)^2} =
          \left(1-\frac{\alpha}{2}\right)^{-1} \,.
\end{align}
Hence $r(x)$ is a linear function, and because the parenthesis in
equation~(\ref{eq:rdef2}) is
\begin{align}
  \exp( r(x)+r(y)) - \exp( r(x')+r(y') )
\end{align}
and $x+y=x'+y'$ we see that $R(x',y',x,y)$ vanishes when
$g(x)=f_\alpha(x)$. Therefore not only does $Q[f_\alpha](x)$ vanish,
but the whole integrand, which is to say that the collision process
satisfies a detailed balance condition also after passing to the
limit.

Finally, to prove~(\ref{eq:thm44:3}) we write
\begin{align}
  \frac{\partial}{\partial t} \left( g \log\left(\frac{\alpha
  g}{1-\alpha g}\right)\right)&
 = \frac{\partial g}{\partial t}
  \left( \log\left(\frac{\alpha g}{1-\alpha g}\right) +
         \frac{\alpha g}{1-\alpha g}\right) = - Q[g](x) r_{g}(x)\,,
\end{align}
where $r_{g}(x)$ is the expression in~(\ref{eq:103def}) with
$f_\alpha$ replaced by $g$. Using the expression in~(\ref{eq:101}), we
then find
\begin{align*}
  \frac{d}{dt}\int_{0}^{\infty}
  & g(x,t)
    \log\left(\frac{\alpha  g(x,t)}{1-\alpha g(x,t)}\right)\, {\rm d}x  = \\
  &  -\frac{1}{8} \int_{0}^{\infty}\int_{0}^{\infty}\int_{-1}^{1}
    R(x',y',x,y)\left( r_g(x)+r_g(y) - r_g(x')-r_g(y')\right){\rm d}\xi {\rm d}x {\rm d}y =\\
  & -\frac{1}{8} \int_{0}^{\infty}\int_{0}^{\infty}\int_{-1}^{1}
    g(x) g(y) g(x')g(y') \left(e^{r_g(x)+r_g(y)}-
    e^{r_g(x')+r_g(y')}\right) \times \\  
   & \qquad\qquad\qquad\qquad \qquad\qquad   \left( r_g(x)+r_g(y) -
     r_g(x')-r_g(y')\right)){\rm d}\xi {\rm d}x {\rm d}y \le 0\,,
\end{align*}
which proves~(\ref{eq:thm44:3}). Therefore
\begin{align*}
  \int_{0}^{\infty} g(x) \log\left(\frac{\alpha g(x)}{1-\alpha g(x)}\right)\,dx
\end{align*}
is an entropy for the Boltzmann equation~(\ref{eq:35}). 
\end{proof}

\section{ Simulation results}
\label{sec:sim}

We present here simulations to illustrate the results presented in the
previous sections, and to provide support for the conjecture that the
Kac process on $\SEnes$ propagates detailed chaos according to
Definition~\ref{def2}, and
moreover to investigate the long time behavior of solutions of
different types.

The sampling of initial data has been done as described by
Theorem~\ref{chacon} and Theorem~\ref{thm:chaosexp}. A very large number
of random numbers have been used, and in particular for simulations
with a large number of particles, it is necessary to use random
numbers with hight precision. We have generated random numbers with 64
bits precision, using routines from the GNU Scientific Library~\cite{GSL2018}.

In order to avoid having to compute the distance
between the new energy of a particle, $x_j^*$ with the energies of all
other particles, which would imply a computational cost of $\bigoh(n)$
for each jump, the $x_j$ are kept in an ordered list, which is
implemented as minor modification of the AVL-tree as described by
Ben Pfaff~\cite{Pfaff_avl}. In this way the computational cost of one
collision grows as $\bigoh\left(\log(n)\right)$.

In the first example the initial distributions are
$(\alpha,f_\alpha)$-chaotic, {\em i.e.} chosen to converge to the
equilibrium distribution with $\alpha=1$. We compare sampling initial
data that are equidistributed
(which corresponds to taking samples as in Remark~\ref{rem:37} with
$K\rightarrow\infty$ in Equation~(\ref{dir1b}), and a 
Dirichlet distributed initial data with $K=0.02$. These are
shown at $t=0$ (Fig.~\ref{fig:sim1a}), $t=0.1$
(Fig.~\ref{fig:sim1b}), and $t=10.0$ (Fig.~\ref{fig:sim1c}), showing
that although the initial distributions  are equilibrium-chaotic,
they are not at equilibrium for this jump
process. Fig.~\ref{fig:sim1d} shows the gap distribution at $t=0$ and
$t=10$, and includes also the result when the initial distribution is a
true equilibrium for this process, with asymptotically exponentially
distributed energy gaps. For these cases, the gap distribution is very
close to exponential at $t=10$, which supports our conjecture that this
property is propagated in time.  More simulation results can be found
in the supplementary material.

\begin{figure}[!h]
  \centering
  \includegraphics[width=0.5\textwidth]{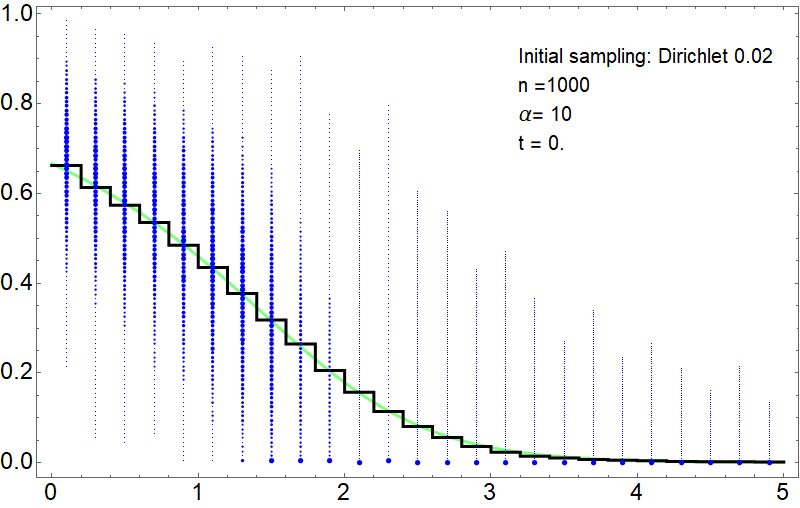}\;%
   \includegraphics[width=0.5\textwidth]{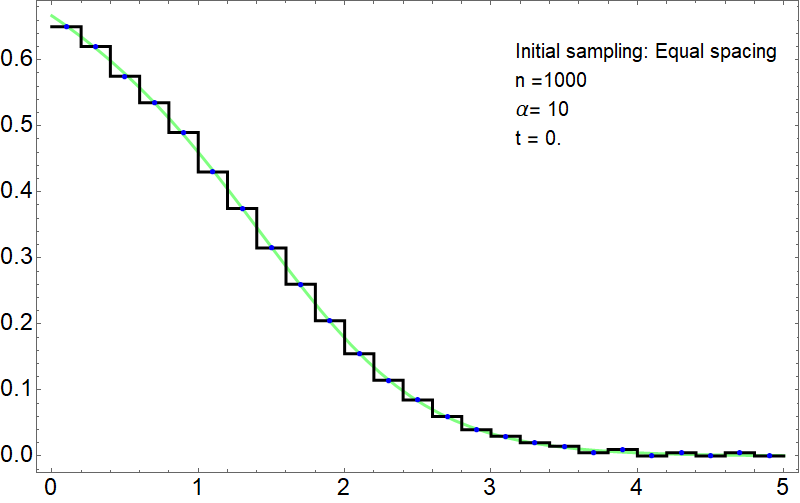}\;%
  \caption{\small The graphs show the equilibrium distribution (green)
  and the result from 5000 independent samples of the empirical
  distribution with $n=1000$, counted in bins of width $0.2$. The
  black step function shows the mean outcome, and the blue dots
  illustrate the distribution of counts in the bins, with the area of
  the dots proportional to the number of samples with the same
  count. The exclusion parameter $\alpha=1.0$}
  \label{fig:sim1a}
\end{figure}

\begin{figure}[!h]
  \centering
  \includegraphics[width=0.5\textwidth]{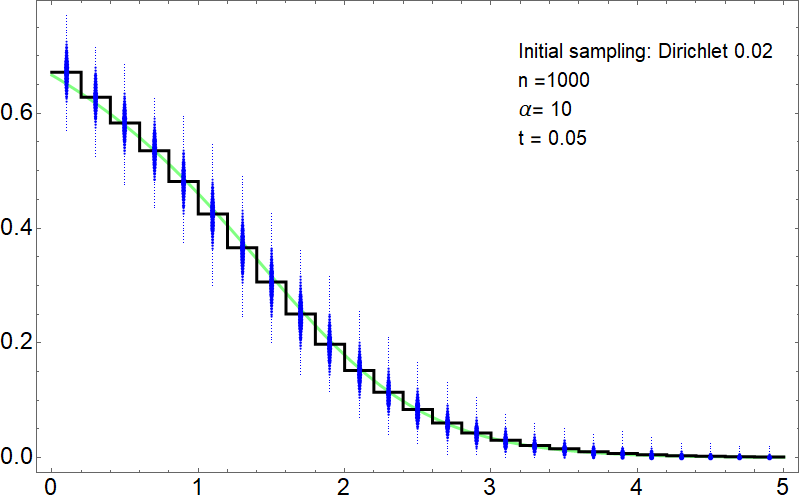}\;%
   \includegraphics[width=0.5\textwidth]{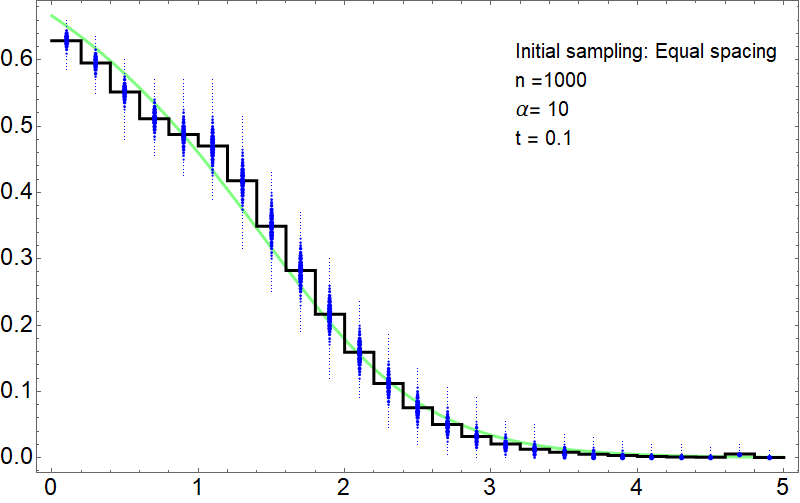}\;%
  \caption{\small The graphs show the outcome of the same simulation
    as in Figure~\ref{fig:sim1a} at time $t=0.1$. This shows that
    although the initial data in both cases are equilibrium chaotic,
    the non-equilibrium state gives quite different behavior of the evolution.}
  \label{fig:sim1b}
\end{figure}

\begin{figure}[!h]
  \centering
  \includegraphics[width=0.5\textwidth]{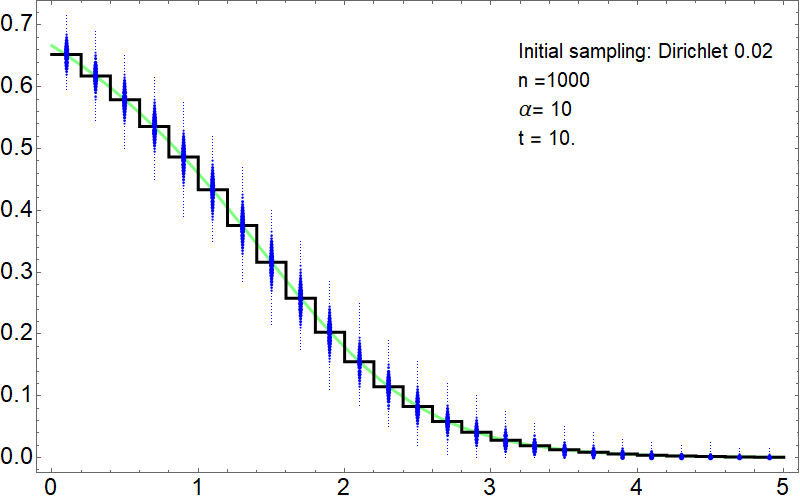}\;%
   \includegraphics[width=0.5\textwidth]{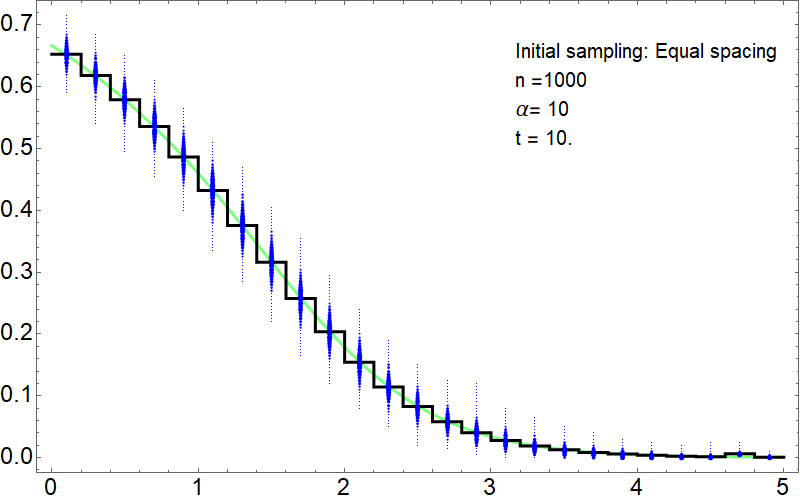}\;%
  \caption{\small These graphs represent the solutions of the same
    simulation as in Figure~\ref{fig:sim1a} at time $t=10.0$. Here thee
  two simulations give the same result, a convergence to the true equilibrium.}
  \label{fig:sim1c}
\end{figure}

\begin{figure}[!h]
  \centering
  \includegraphics[width=0.5\textwidth]{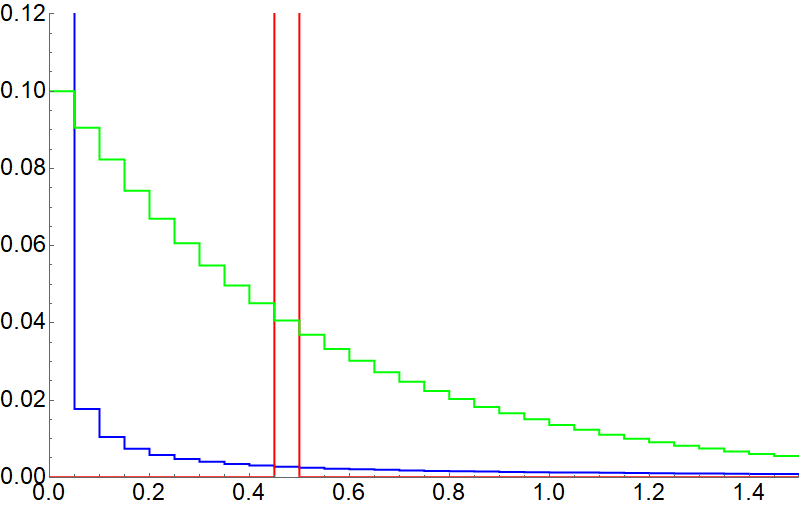}\;%
   \includegraphics[width=0.5\textwidth]{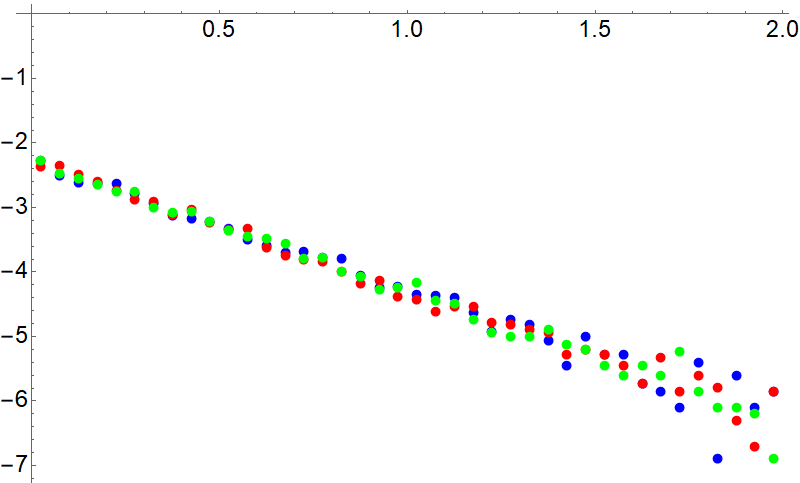}\;%
  \caption{\small The graphs show the gap distribution at time $t=0$ (left)
  and $t=2$ (right) for initial data with equal spacing of the excess
  energy (red) and the Dirichlet 0.01- distribution (blue). At $t=2$
  this is presented in logarithmic scale to show that the distribution
becomes exponential as conjectured. The red and blue dots almost
overlap here. At $t=0$ the red curve represents a Dirac measure, and the
blue curve shows that with the Dirichlet 0.01- distribution, most gaps
are very close to zero, and the excess energy is essentially
distributed to a few very large gaps.}
  \label{fig:sim1d}
\end{figure}

\section{Acknowledgements}

This work was begun in Fall 2016 when  E.C. was visiting the
University of Gothenburg and Chalmers Institute of Technology, to whom
he is grateful for hospitality. Both E.C. and B.W. were able to meet
again at Institute Mittag-Leffler and the Oberwolfach Mathematical
Research Institute, and both authors thank these institutes for their
hospitality. The work of E.C. was partially supported by
U.S. N.S.F. grant DMS-1501007. The work of B.W. was partially
supported by the Swedish Research Council and the Swedish Foundation
for Strategic Resarch.



\begin{thebibliography}{10}

\bibitem{Ahsanullah_etal_2013}
Mohammad Ahsanullah, Valery~B. Nevzorov, and Mohammad Shakil.
\newblock {\em An introduction to order statistics}, volume~3 of {\em Atlantis
  Studies in Probability and Statistics}.
\newblock Atlantis Press, Paris, 2013.

\bibitem{Benedetto_etal2007}
Dario~Benedetto, Fran\c{c}ois~Castella, Raffaele~Esposito, and Mario~Pulvirenti.
\newblock A short review on the derivation of the nonlinear quantum {B}oltzmann
  equations.
\newblock {\em Commun. Math. Sci.}, (suppl. 1):55--71, 2007.

\bibitem{Carlen_etal2010}
Eric~A. Carlen, Maria~C. Carvalho, Jonathan Le~Roux, Michael Loss, and C\'edric
  Villani.
\newblock Entropy and chaos in the {K}ac model.
\newblock {\em Kinet. Relat. Models}, 3(1):85--122, 2010.

\bibitem{CiprianiZeindler2015}
Alessandra Cipriani and Dirk Zeindler.
\newblock The limit shape of random permutations with polynomially growing
  cycle weights.
\newblock {\em ALEA Lat. Am. J. Probab. Math. Stat.}, 12(2):971--999, 2015.

\bibitem{Colangeli_etal2015}
Matteo Colangeli, Federica Pezzotti, and Mario Pulvirenti.
\newblock A {K}ac model for fermions.
\newblock {\em Arch. Ration. Mech. Anal.}, 216(2):359--413, 2015.

\bibitem{GSL2018}
Mark~Galassi et~al.
\newblock {\em GNU Scientific Library Reference Manual (3rd Ed.)}, 2018.
\newblock ISBN 0954612078 {\tt http://www.gnu.org/software/gsl/}.

\bibitem{GirouxFerland2008}
Gaston Giroux and Ren\'{e} Ferland.
\newblock Global spectral gap for {D}irichlet-{K}ac random motions.
\newblock {\em J. Stat. Phys.}, 132(3):561--567, 2008.

\bibitem{Kac1956}
Mark Kac.
\newblock Foundations of kinetic theory.
\newblock In {\em Proceedings of the {T}hird {B}erkeley {S}ymposium on
  {M}athematical {S}tatistics and {P}robability, 1954--1955, vol. {III}}, pages
  171--197, Berkeley and Los Angeles, 1956. University of California Press.

\bibitem{Nordheim1928}
Lothar Nordheim.
\newblock On the kinetic methods in the new statistics and its application in
  the electron theory of conductivity.
\newblock {\em Proc. Roy. Soc. London Ser. A}, 119(783):689--698, 1928.

\bibitem{PalPitman2008}
Soumik Pal and Jim Pitman.
\newblock One-dimensional {B}rownian particle systems with rank-dependent
  drifts.
\newblock {\em Ann. Appl. Probab.}, 18(6):2179--2207, 2008.

\bibitem{Pfaff_avl}
  \begin{flushleft}
Ben Pfaff.
\newblock {\em An Introduction to Binary Search Trees and Balanced Trees},
2004.
\end{flushleft}
\newblock {\tt https://www.gnu.org/software/avl/}.

\bibitem{Reygner2015}
Julien Reygner.
\newblock Chaoticity of the stationary distribution of rank-based interacting
  diffusions.
\newblock {\em Electron. Commun. Probab.}, 20:no. 60, 20, 2015.

\bibitem{Shkolnikov2012}
Mykhaylo Shkolnikov.
\newblock Large systems of diffusions interacting through their ranks.
\newblock {\em Stochastic Process. Appl.}, 122(4):1730--1747, 2012.

\bibitem{Sznitman1991}
Alain-Sol Sznitman.
\newblock Topics in propagation of chaos.
\newblock In {\em \'{E}cole d'\'{E}t\'e de {P}robabilit\'es de {S}aint-{F}lour
  {XIX}---1989}, volume 1464 of {\em Lecture Notes in Math.}, pages 165--251.
  Springer, Berlin, 1991.

\bibitem{UehlingUhlenbeck1933}
Edwin~A. Uehling and George~E. Uhlenbeck.
\newblock Transport phenomena in einstein-bose and fermi-dirac gases.~i.
\newblock {\em Physical Review}, 43(7):552--561, 1933.

\bibitem{VershikYakubovich2003}
Anatoli\u{\i}~M. Vershik and Yu.~V. Yakubovich.
\newblock Asymptotics of the uniform measure on simplices, and random
  compositions and partitions.
\newblock {\em Funktsional. Anal. i Prilozhen.}, 37(4):39--48, 95, 2003.

\end{thebibliography}

\def\cprime{$'$} \def\cprime{$'$}
  \def\polhk#1{\setbox0=\hbox{#1}{\ooalign{\hidewidth
  \lower1.5ex\hbox{`}\hidewidth\crcr\unhbox0}}}


\newpage

\section{Supplementary material available on line: additional numerical results}

{ \em  This part is intended as supplementary material in the
  published version of the paper }

\medskip

In this supplementary section we present numerical
simulations intented to illustrate the results of the paper. The first
examples show the stationary distribution for several 
values of $n$ and the parameter $\alpha$. These were obtained by
samplig from the invariant density as described in Section~\ref{sec:empirical}.
The influence of the
boundary $x=0$ is clearly visible, and extends to an 
interval , $x \le C \varepsilon$, where $C$
depends on $\alpha$ . Not surprisingly, the influence of
the boundary is much stronger when $\alpha$ is close to $2$, and
the density is very high.  In this case the distance between two
partcicles is not much larger than the exclusion distance $\varepsilon$.
When  $\alpha$ is close to $2$ the distribution at low energies
is essentially discrete, 
as shown by the very oscillatory behavior of the distribution.  This
is also illustrated  showing the distribution of $x_j$, for
$j=1,2,3,4$ and for 
some larger numbers. When $\alpha=1.8$, for example, it is only at
$j\sim 50$, that the distribution of  
$x_j$ and $x_{j+1}$ begin to overlap. But as the width of the interval
is proportional to $\varepsilon$ the oscillations of the density
$f_n(x)$ for $x$ larger than any 
fixed value $x_0>0$ disappears when $n\rightarrow\infty$. Some results
illustrating this are presented in Fig. ~\ref{fig:n10a10} to
Fig.~\ref{fig:n1000a18}.  

\begin{figure}[H]
  \centering
   \includegraphics[width=0.5\textwidth]{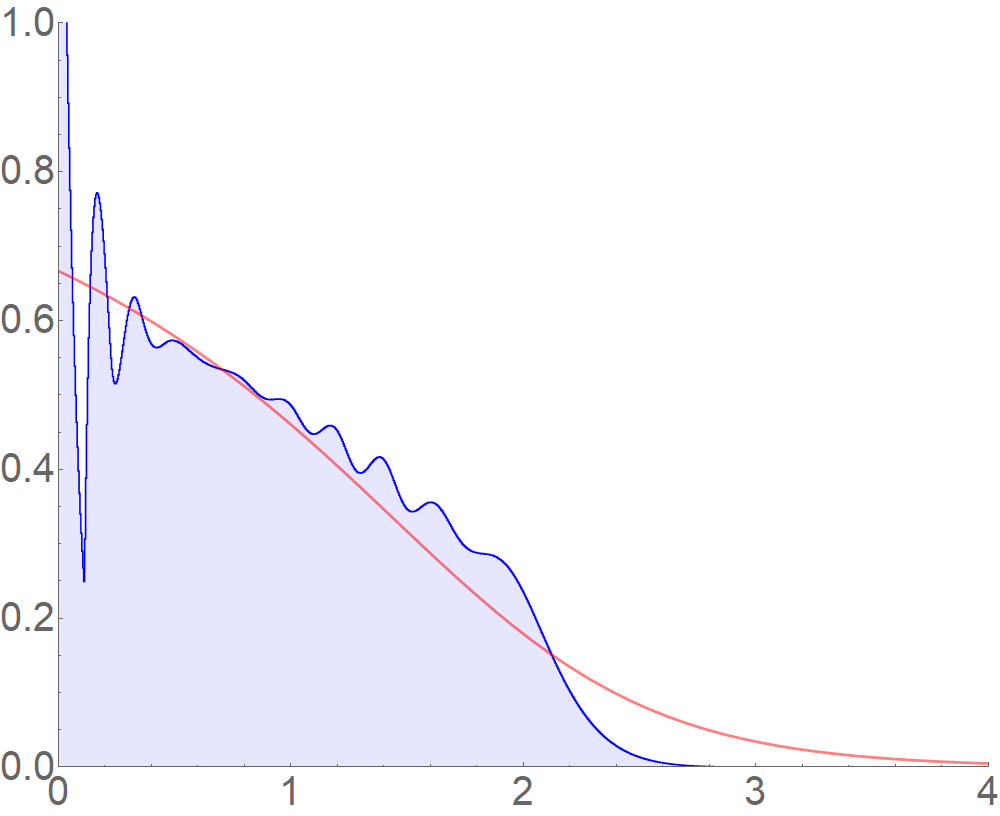}\;
   \begin{minipage}[b]{0.45\linewidth}

    \ 
     
    \includegraphics[width=1.0\textwidth]{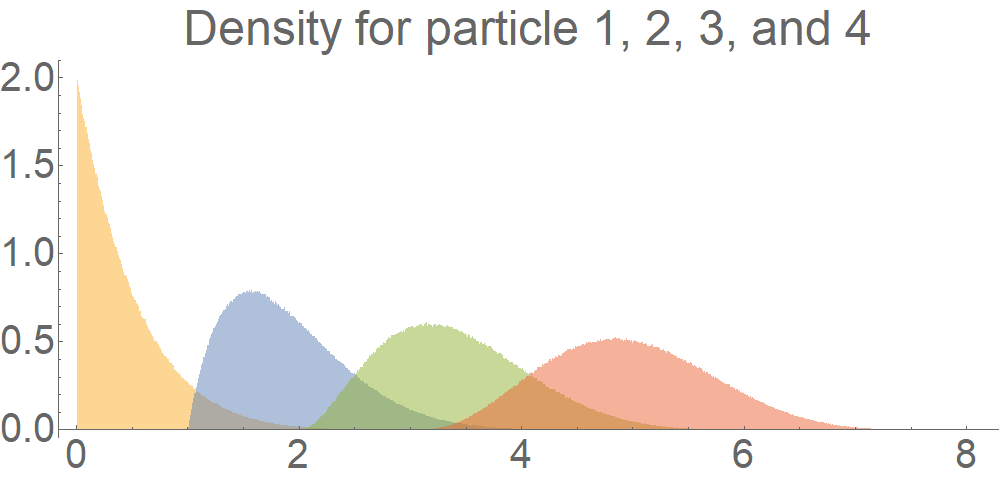}
    
    \includegraphics[width=1.0\textwidth]{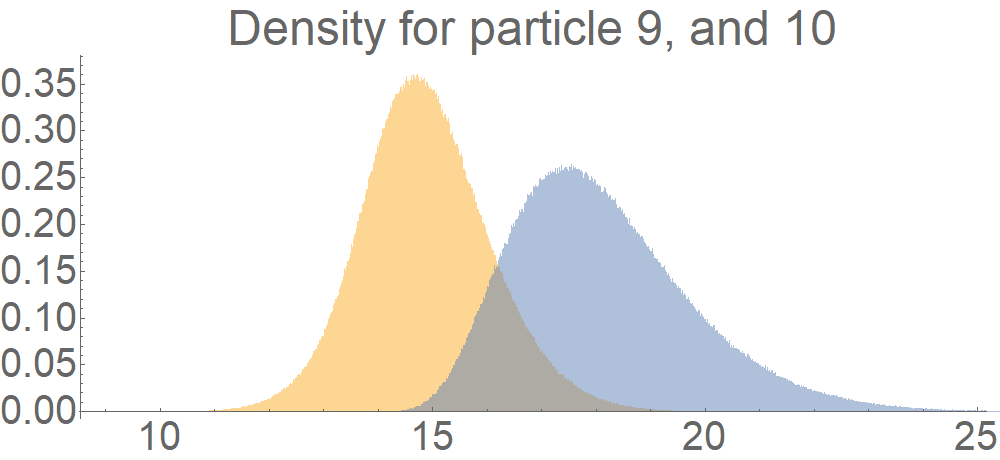}   
  \end{minipage}
 
  \caption{The empirical distribution from $4\times 10^8$ samples from
    $\SEnes$ with $n=10, \alpha=1.0$ (left). In all these images the
    red curve shows the limiting distribution when
    $n\rightarrow\infty$. Right: the distribution of the four
    particles with lowest energy, and with the highest energy taken
    from $2\times 10^6$ samples. The unit in the  $x$-axis is $\epsilon$, the
    minimal energy gap between particles. }  
  \label{fig:n10a10}
\end{figure}

\begin{figure}[H]
  \centering
   \includegraphics[width=0.5\textwidth]{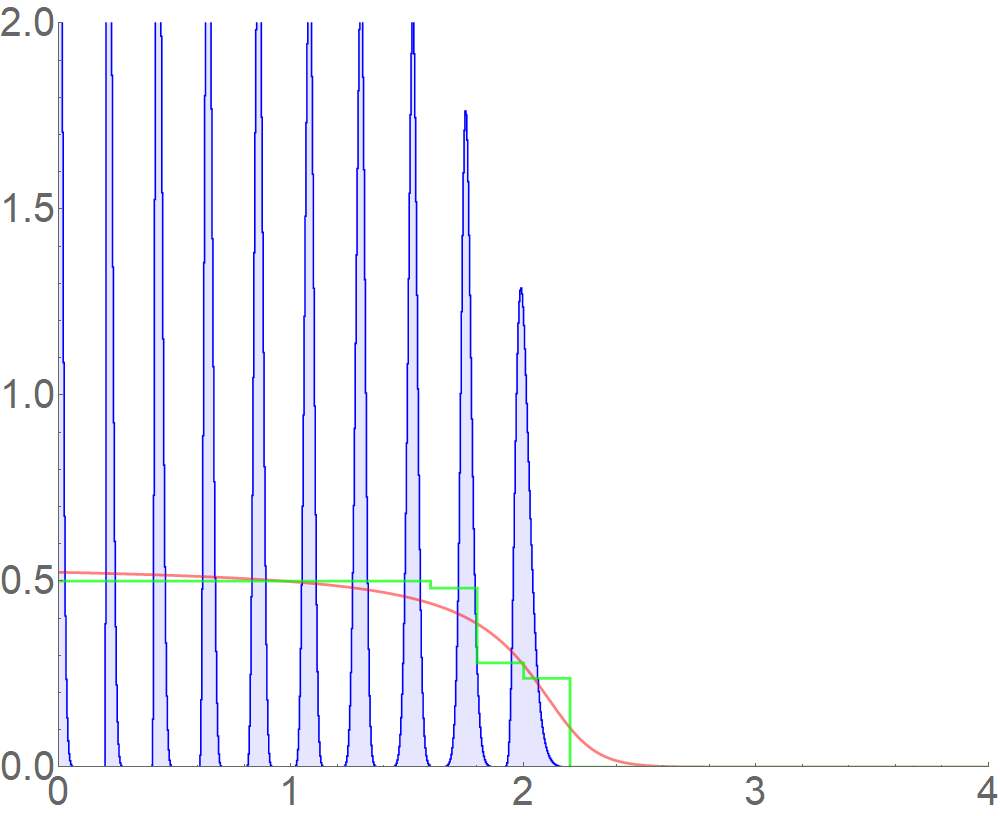}\;
   \begin{minipage}[b]{0.45\linewidth}

    \ 
     
    \includegraphics[width=1.0\textwidth]{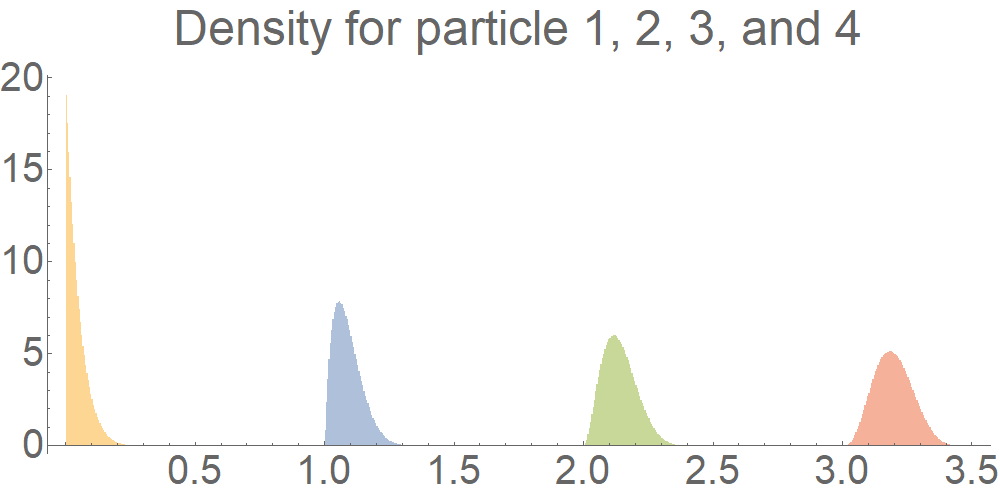}
    
    \includegraphics[width=1.0\textwidth]{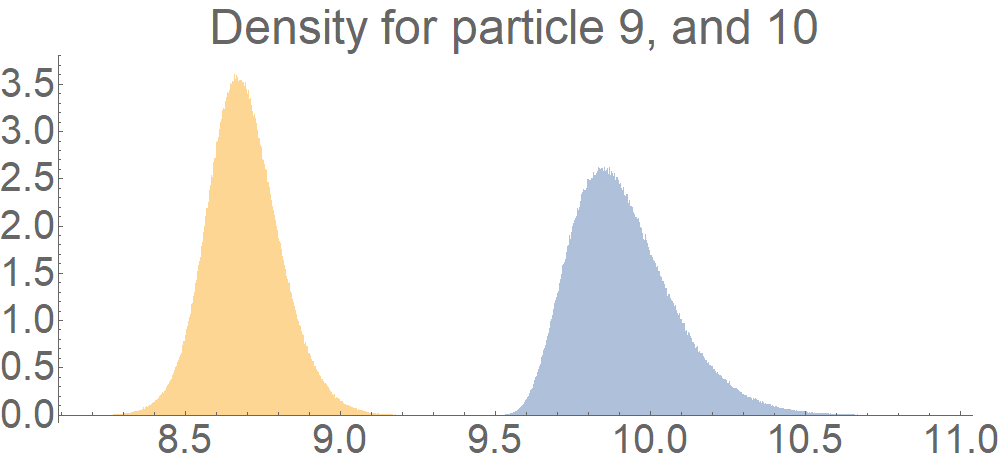}   
  \end{minipage}
 
  \caption{The empirical distribution from $4\times 10^8$ samples from
    $\SEnes$ with $n=10, \alpha=1.8$ (left).
    The green curve shows averages of the
    empirical distribution over intervals of size $0.2$. Right: the distribution of the four
    particles with lowest energy, and with the highest energy taken
    from $2\times 10^6$ samples.}  
  \label{fig:n10a18}
\end{figure}

\begin{figure}[H]
  \centering
   \includegraphics[width=0.5\textwidth]{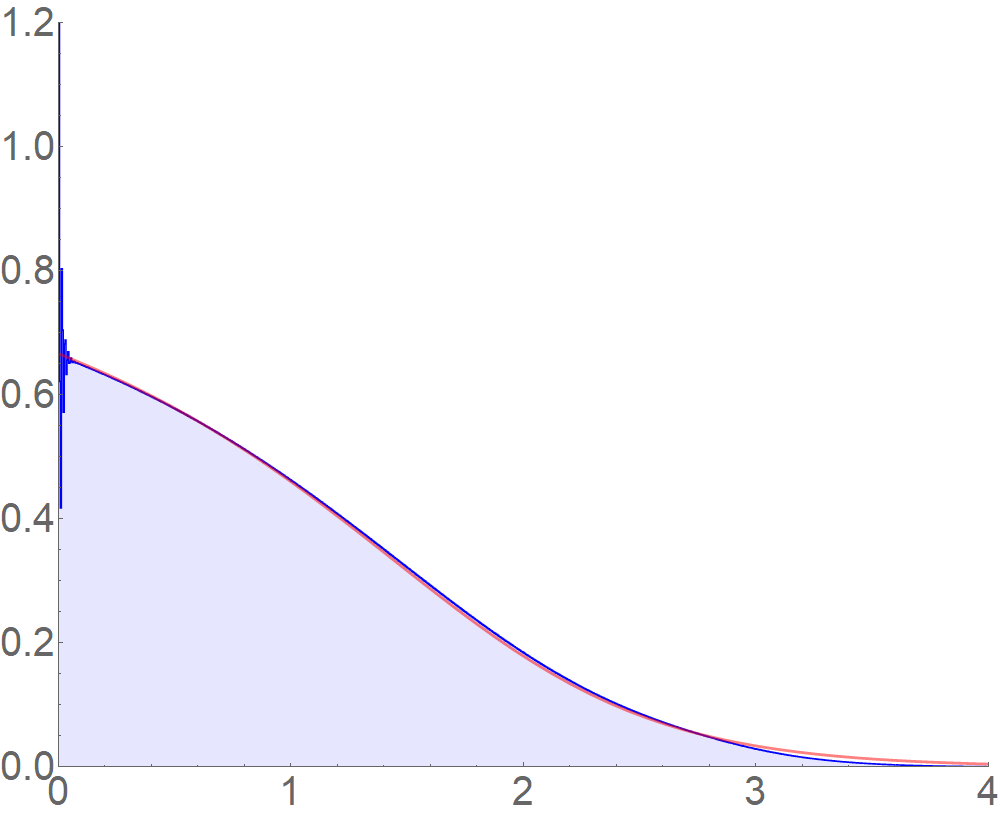}\;
   \begin{minipage}[b]{0.45\linewidth}

    \ 
     
    \includegraphics[width=1.0\textwidth]{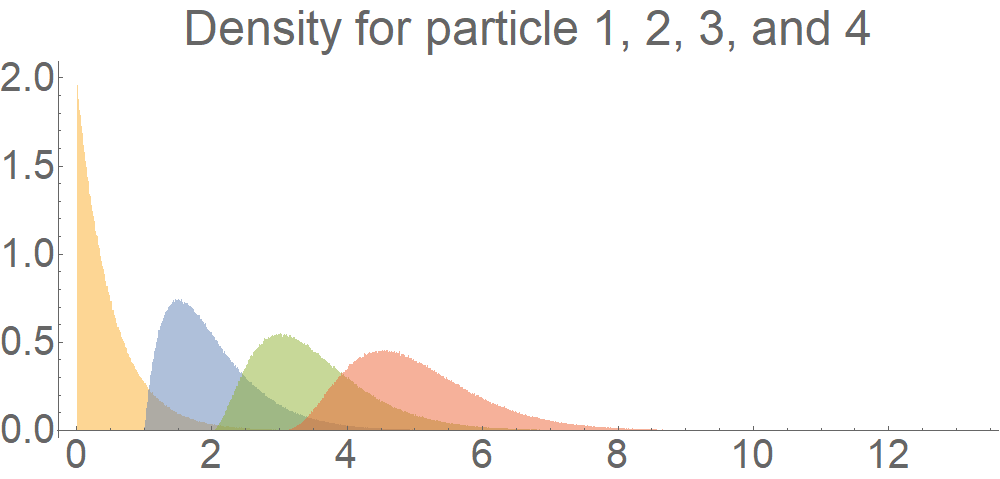}
    
    \includegraphics[width=1.0\textwidth]{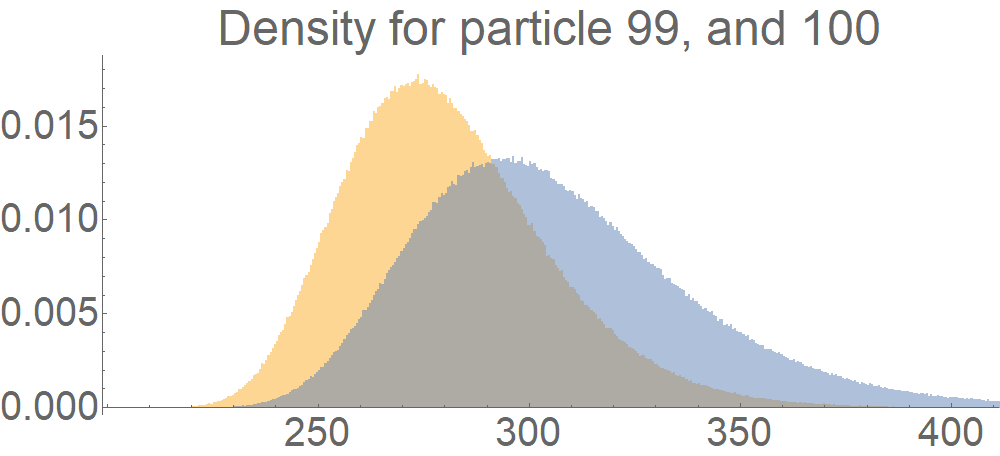}   
  \end{minipage}
 
  \caption{The empirical distribution from $ 10^8$ samples from
    $\SEnes$ with $n=100, \alpha=1.0$ (left). Right: the distribution of the four
    particles with lowest energy, and with the highest energy taken
    from $2\times 10^6$ samples.}  
  \label{fig:n100a10}
\end{figure}

\begin{figure}[H]
  \centering
   \includegraphics[width=0.5\textwidth]{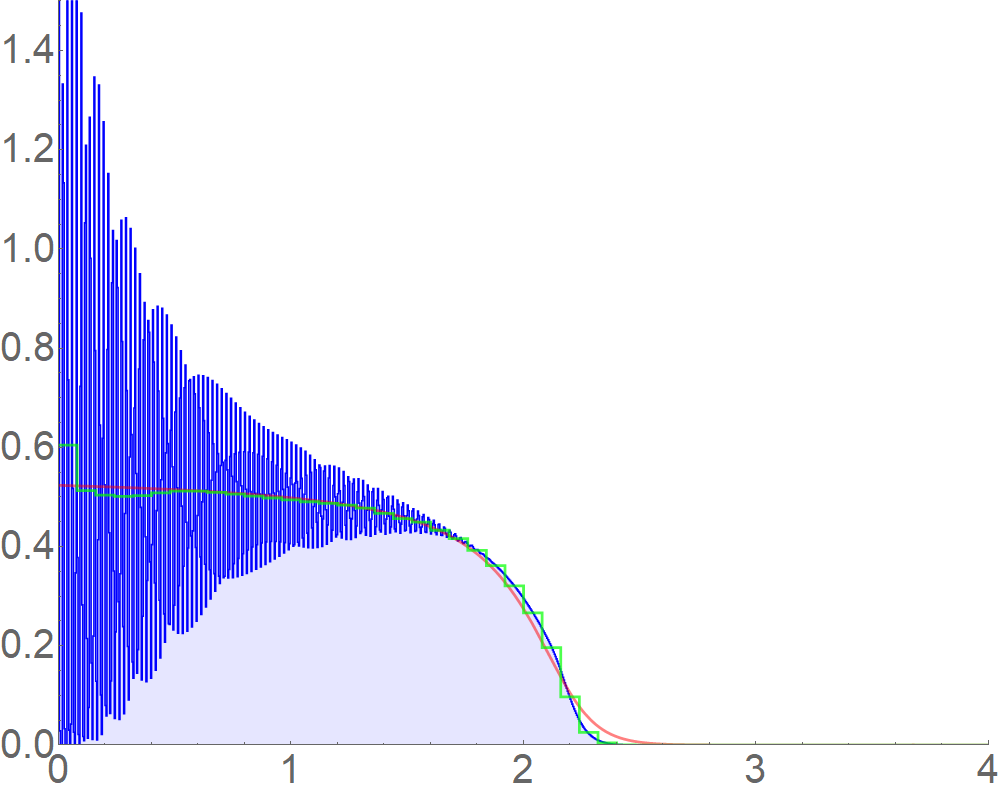}\;
   \begin{minipage}[b]{0.45\linewidth}

    \ 
     
    \includegraphics[width=1.0\textwidth]{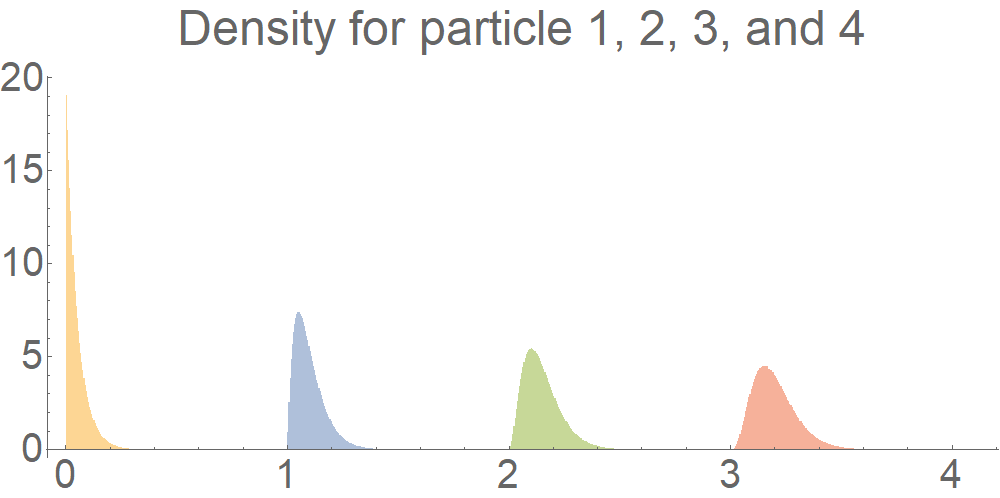}
    
    \includegraphics[width=1.0\textwidth]{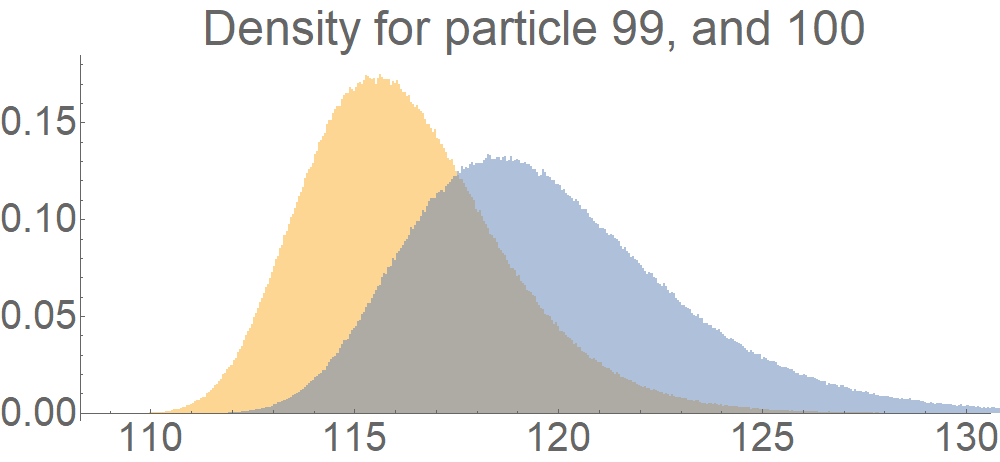}   
  \end{minipage}
 
  \caption{The empirical distribution from $ 10^8$ samples from
    $\SEnes$ with $n=100, \alpha=1.8$ (left).  The green curve shows
    averages of the 
    empirical distribution over intervals of size $0.03$.  Right: the
    distribution of the four 
    particles with lowest energy, and with the highest energy taken
    from $2\times 10^6$ samples.}  
  \label{fig:n100a18}
\end{figure}


\begin{figure}[H]
  \centering
   \includegraphics[width=0.5\textwidth]{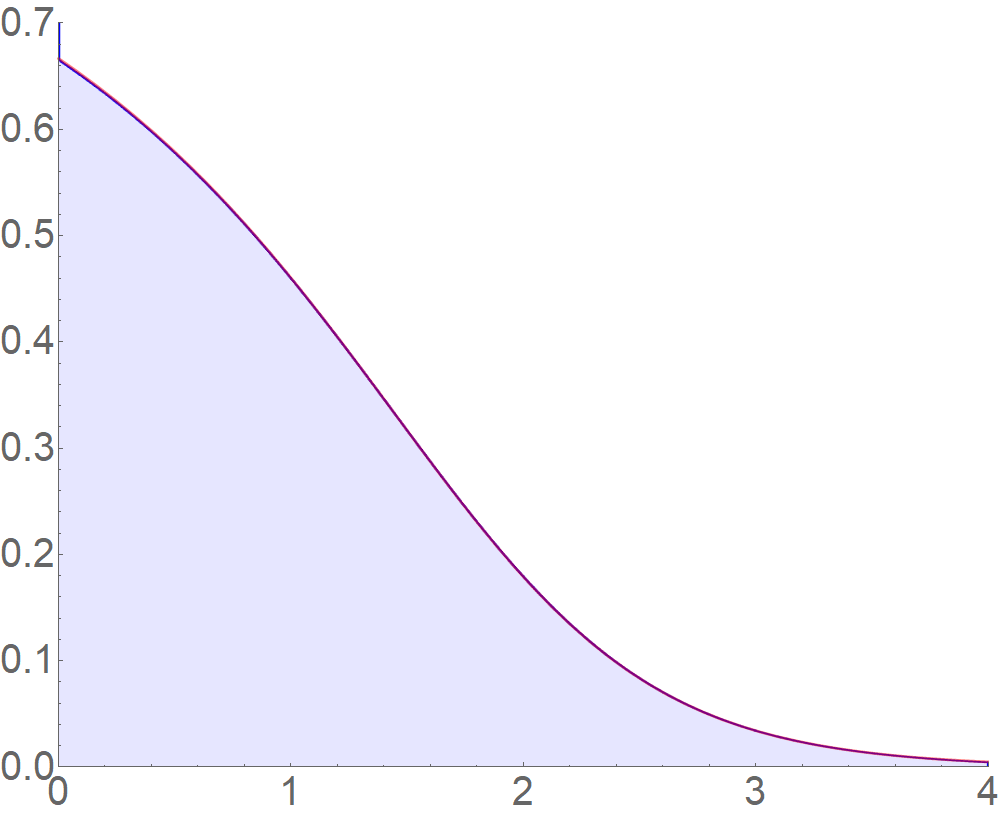}\;
   \begin{minipage}[b]{0.45\linewidth}

    \ 
     
    \includegraphics[width=1.0\textwidth]{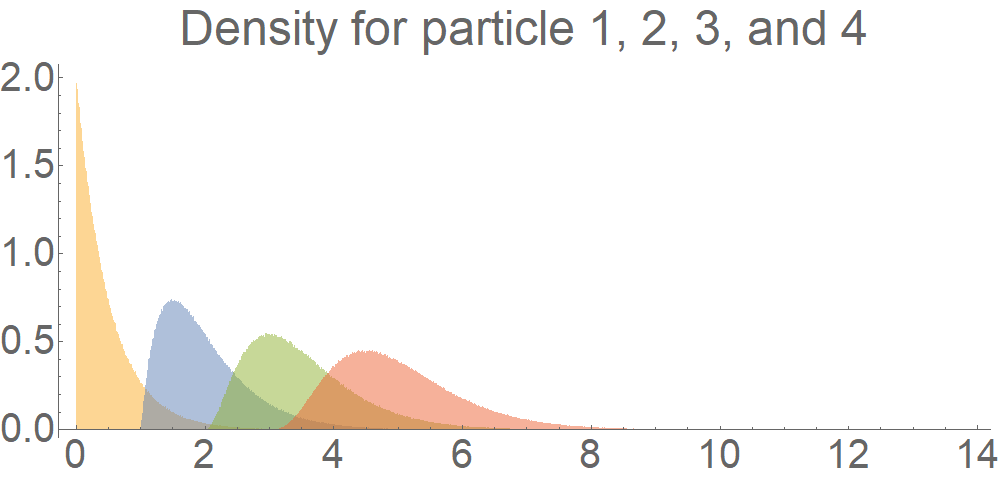}
    
    \includegraphics[width=1.0\textwidth]{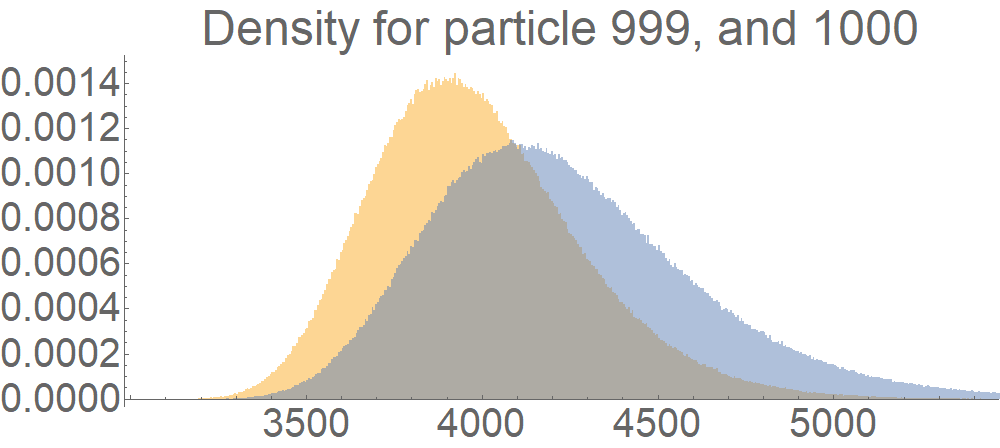}   
  \end{minipage}
 
  \caption{Left: The empirical distribution from $ 10^8$ samples from
    $\SEnes$ with $n=1000, \alpha=1.0$.   Right: the
    distribution of the four 
    particles with lowest energy, and with the highest energy taken
    from $2\times 10^6$ samples.}  
  \label{fig:n1000a10}
\end{figure}

\begin{figure}[H]
  \centering
   \includegraphics[width=0.5\textwidth]{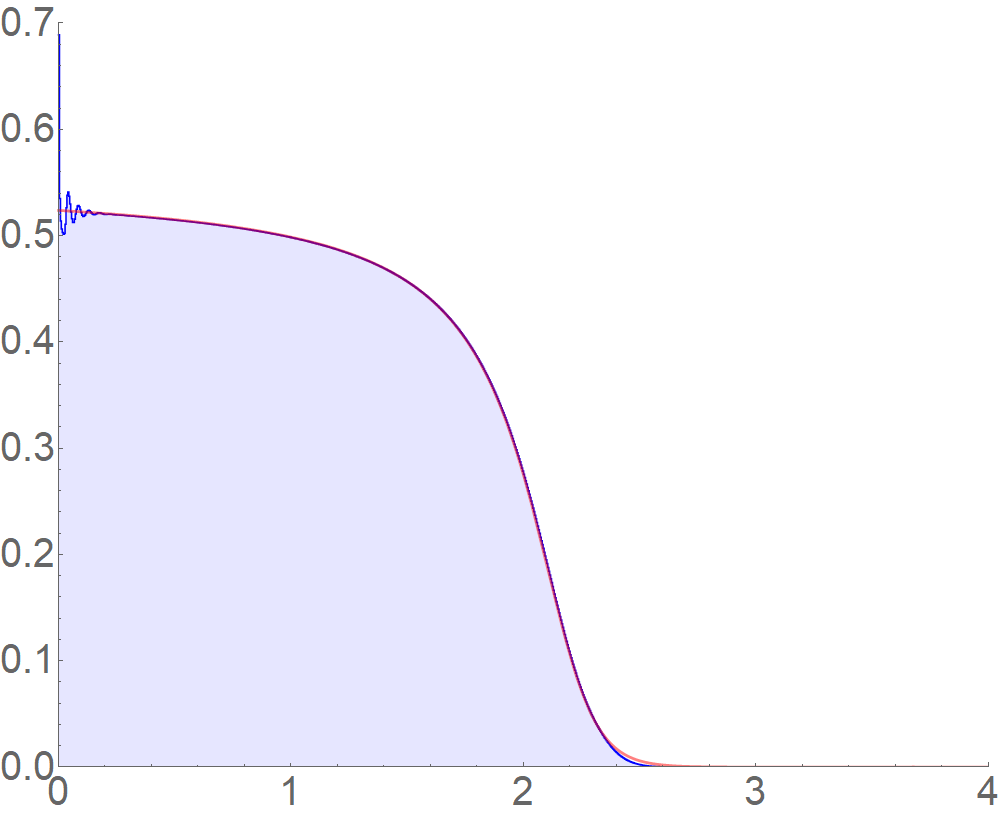}\;
   \begin{minipage}[b]{0.45\linewidth}

    \ 
     
    \includegraphics[width=1.0\textwidth]{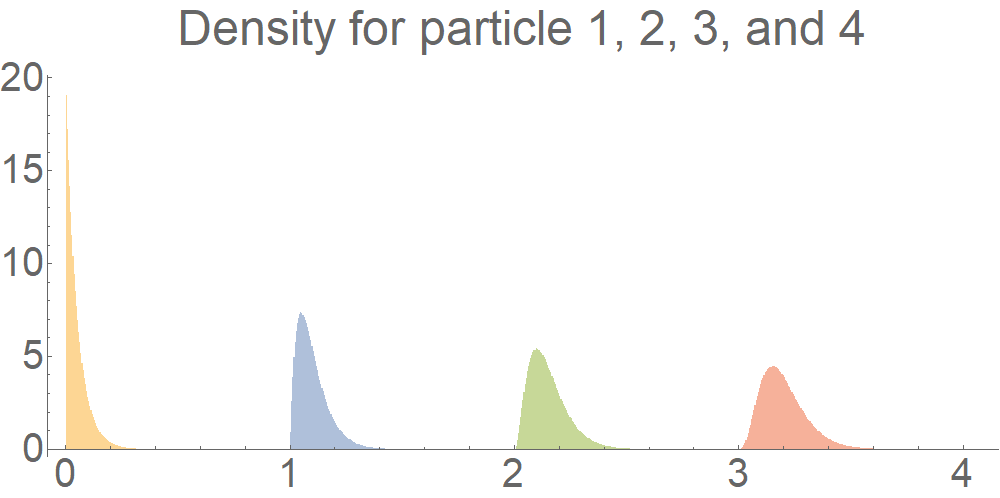}
    
    \includegraphics[width=1.0\textwidth]{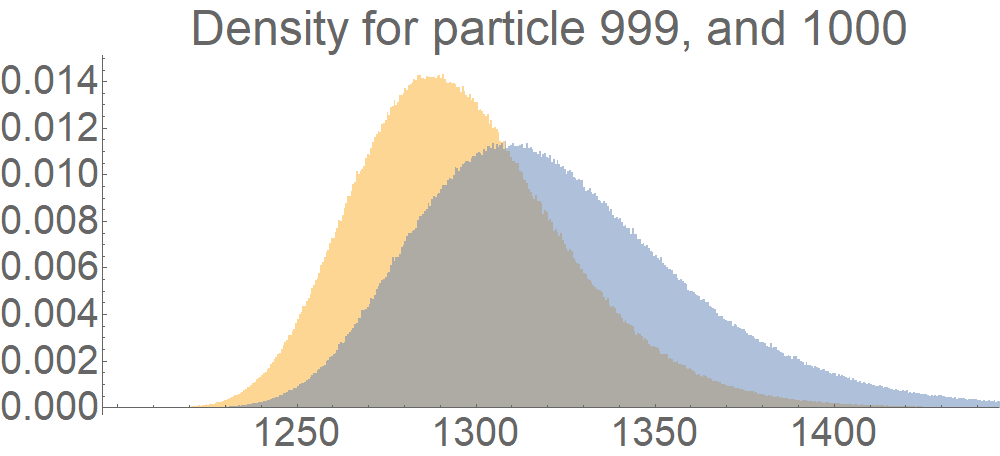}   
  \end{minipage}
 
  \caption{Left: The empirical distribution from $ 10^8$ samples from
    $\SEnes$ with $n=1000, \alpha=1.8$.   Right: the
    distribution of the four 
    particles with lowest energy, and with the highest energy taken
    from $2\times 10^6$ samples.}  
  \label{fig:n1000a18}
\end{figure}

The following  simulation results concern the dynamical process close to
equilibrium. These simulations are motivated partially by the fourth
question presented in the introduction:

{\em 
(4) At which energy levels in equilibrium do collisions occur at a
rate bounded away from zero, and at which energy levels are the
collisions ``frozen out''.
}

The first series of figures compare simulations around the stationary
state with $\alpha=1.8$ and with $\alpha=1.0$. The simulations were
carried out with the number of particles $n=1000$, a number which is
sufficiently large to give a good agreement with the limiting
distribution, and a reasonable computational time. The initial
configuration was sampled from the uniform distribution on
$\SEnes$. The simulation was then run up to time $ T_{\mbox{end}}=10^6$,
and which with $ n=10^3$ implies that the number of (attempted) jumps
during the simulation was $10^{9}$. With $\alpha=1.8$, the particle
distribution is very dense, and therefore the number of successful
jumps was only $5.6\times 10^5$, while with $\alpha=1.0$ the number of
successful jumps was $6.6 \times 10^7$.

To begin, Figure~\ref{fig:n1000dyneqA} shows a time series of the four
particles with lowest  
energy (left: $\alpha=1.8$, right: $\alpha=1.0$), and
Figure~\ref{fig:n1000dyneqB} 
shows histograms of the particle positions at these times, to be
compared with Fig~\ref{fig:n1000a18} and Fig~\ref{fig:n1000a10}.

\begin{figure}[H]
  \centering
   \includegraphics[width=0.45\textwidth]{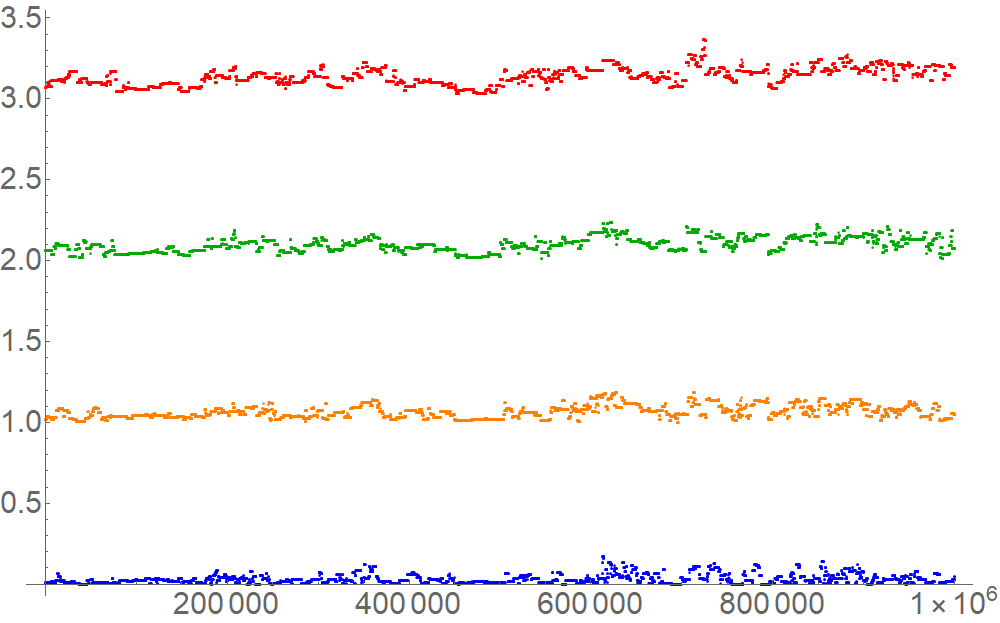}\qquad
   \includegraphics[width=0.45\textwidth]{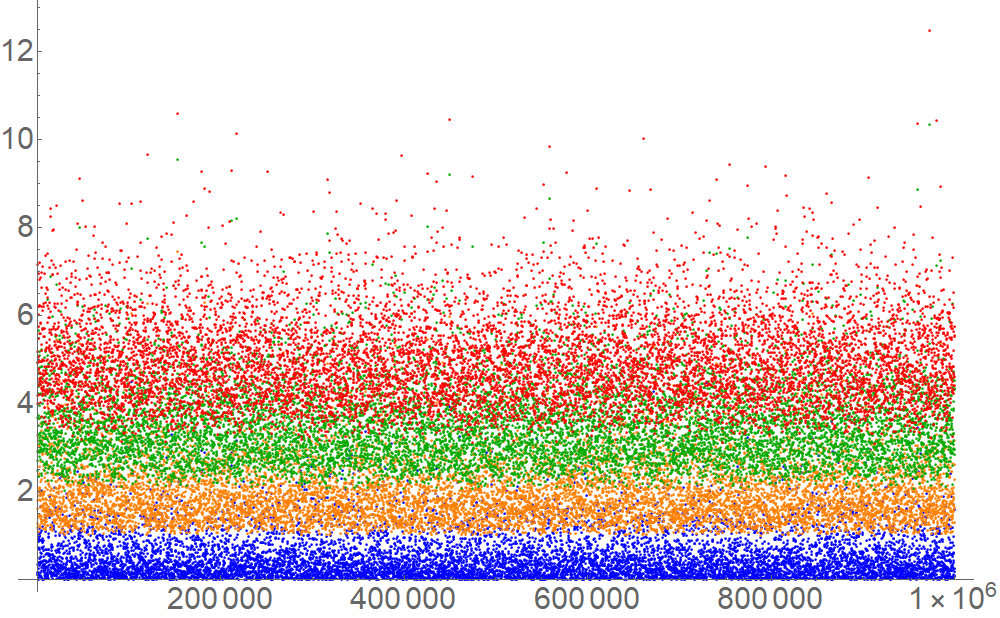}   
  \caption{The state of the four particles with lowest energy for
    $\alpha=1.8$ (left) and $\alpha=1.0$ (right). The state was
    sampled with an interval $100$, and hence each particle is
    represented by 10000 points in the graphs. The unit of the
    vertical axis is $\epsilon$, the minimal energy gap.}   
  \label{fig:n1000dyneqA}
\end{figure}

\begin{figure}[H]
  \centering
   \includegraphics[width=0.45\textwidth]{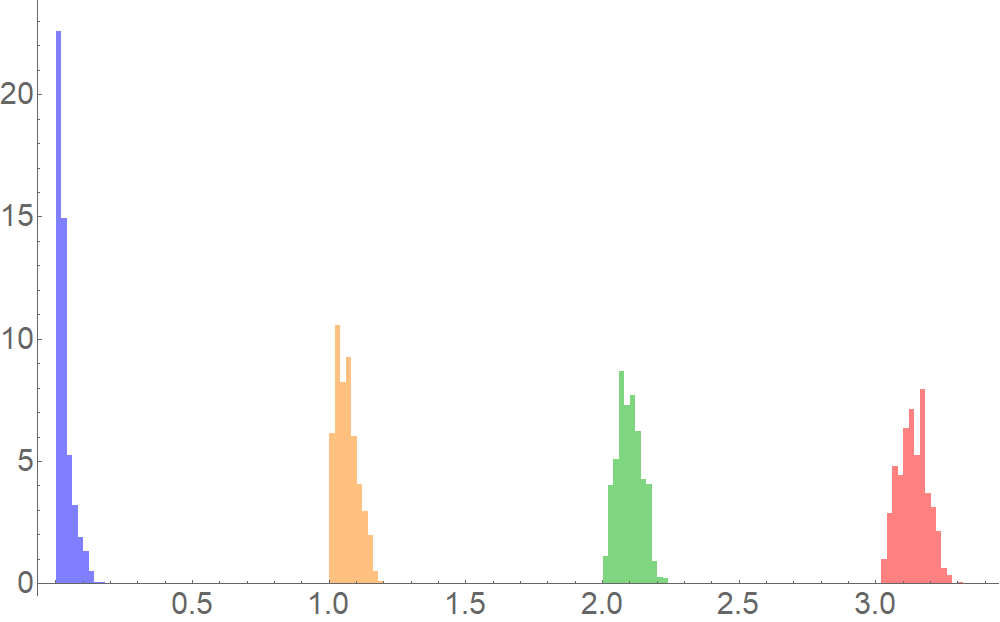}\qquad
   \includegraphics[width=0.45\textwidth]{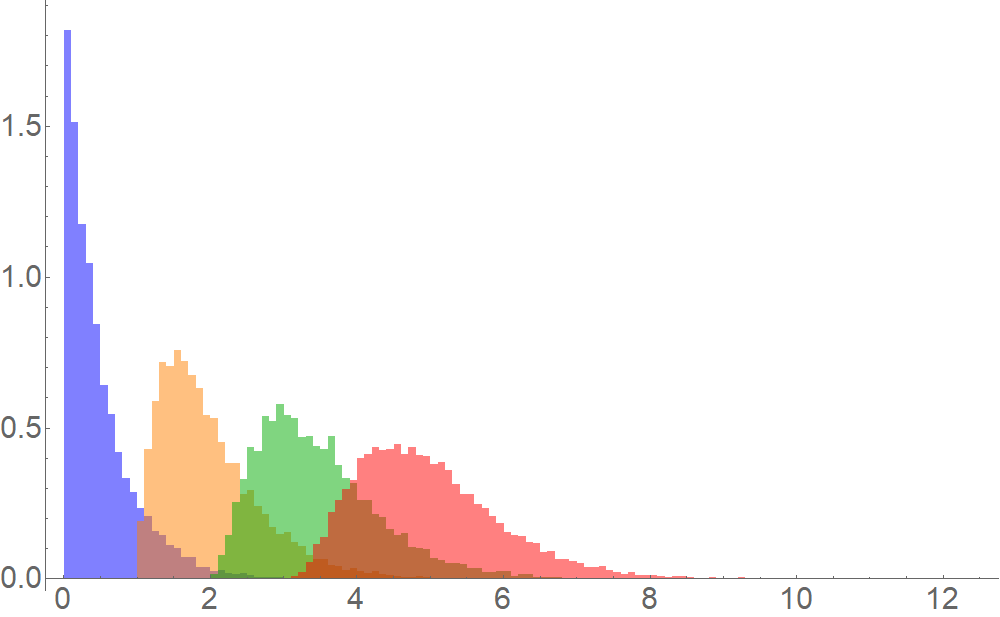}   
  \caption{Histograms of the 10000 samples of each particle energy, as
    shown in  Fig:~\ref{fig:n1000dyneqA}. These graphs should be
    compared with Fig.~\ref{fig:n1000a18} and Fig.~\ref{fig:n1000a10}.}  
  \label{fig:n1000dyneqB}
\end{figure}

If instead of keeping track of the particles ordered from lowest
energy, we follow some tagged particles along the flow, a very
different picture is seen. Fig.~\ref{fig:n1000dyneqC} shows the paths
of the particles that initially had the lowest and highest energy, and
the corresponding histograms are given in
Fig.~\ref{fig:n1000dyneqD}. While in the picture to the right, with
$\alpha=1$, the long time behavior of the two tagged particles are
the same, and generate histograms just like the equilibrium density,
the pictures to the left, with $\alpha=1.8$ the particle with lowest
initial energy always stays at the bottom, and the particle that
initially had the highest energy, while moving around quite a lot
certainly does not regenerate the equilibrium density.   

\begin{figure}[H]
  \centering
   \includegraphics[width=0.45\textwidth]{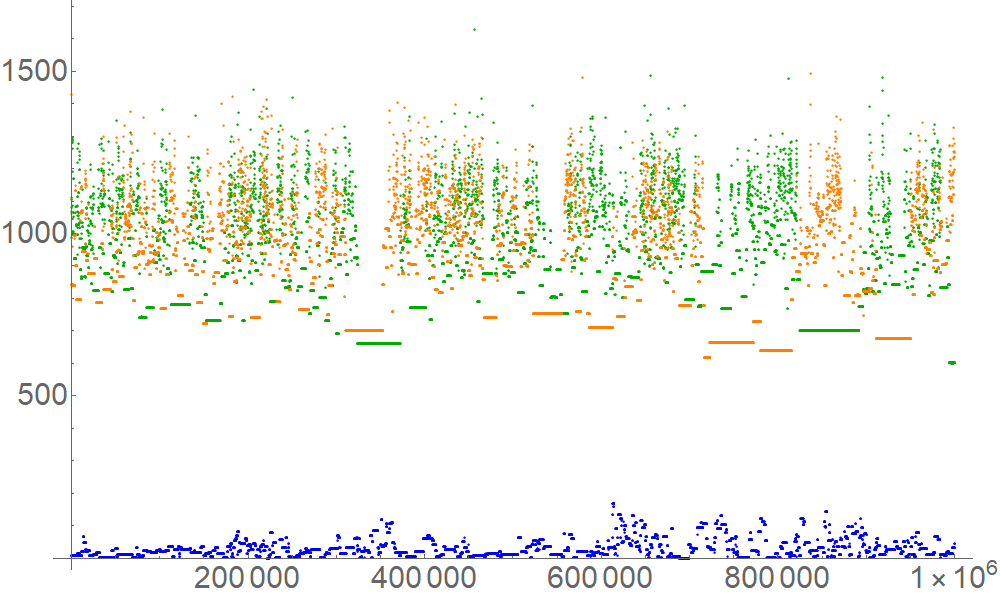}\qquad
   \includegraphics[width=0.45\textwidth]{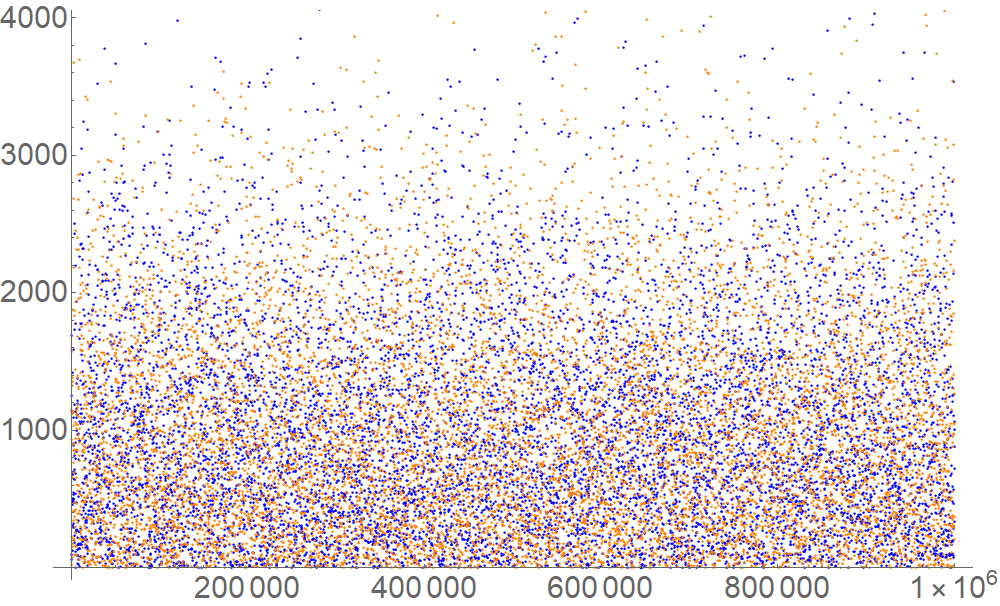}   
  \caption{Sample paths of the tagged  particles with lowest (blue)  and highest
    energy (orange) in the  initial configuration, for $\alpha=1.8$
    (left) and $\alpha=1.0$ (right). To the left, also the path of the
  second highest energy is plotted (green), and the blue path is
  multiplied by a factor 1000 for clarity.}  
  \label{fig:n1000dyneqC}
\end{figure}

\begin{figure}[H]
  \centering
   \includegraphics[width=0.45\textwidth]{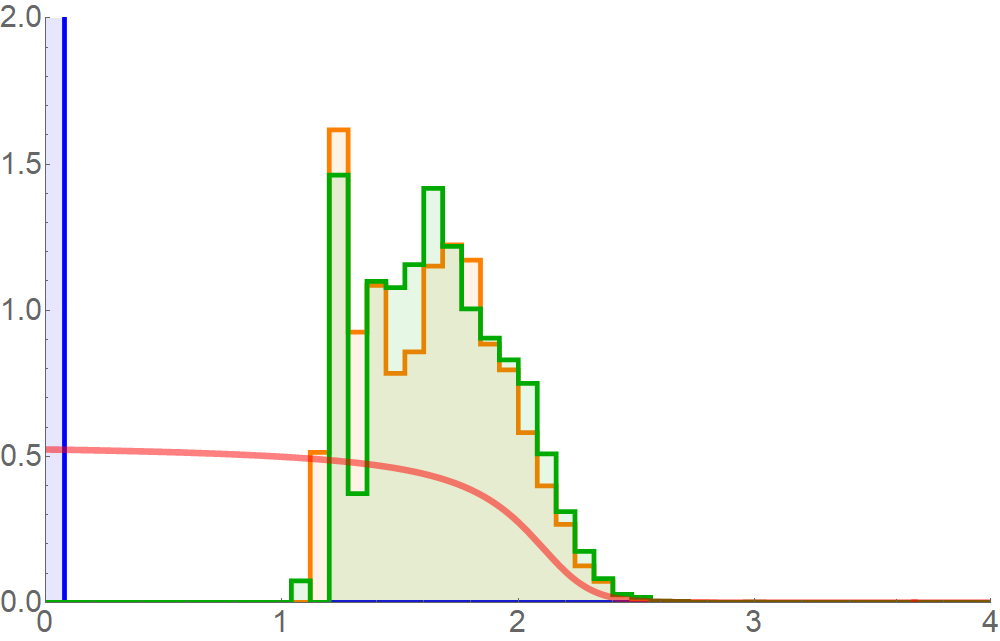}\qquad
   \includegraphics[width=0.45\textwidth]{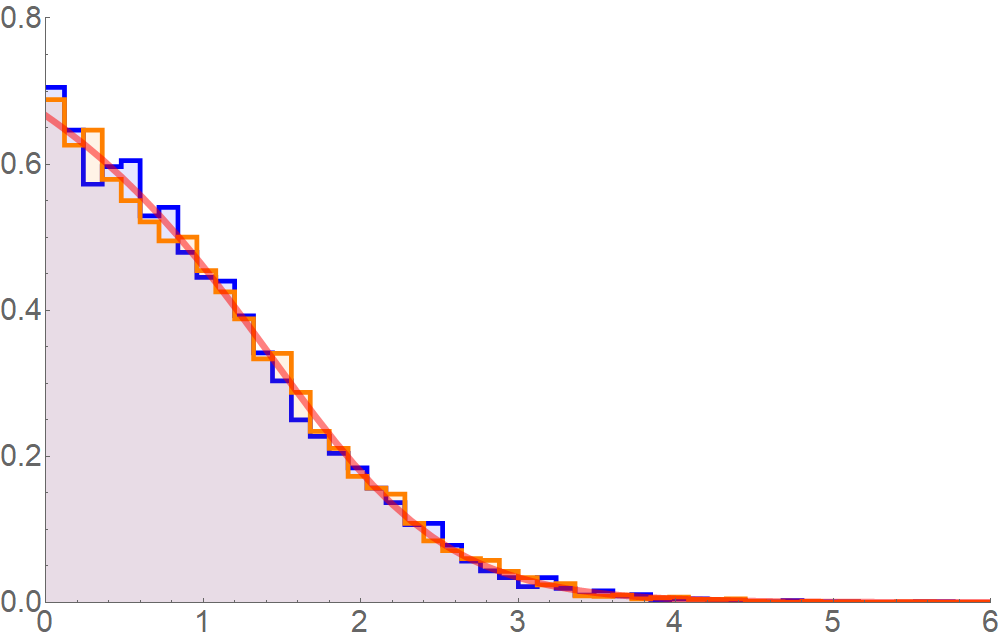}   
  \caption{Histograms of the sample paths shown in
    Fig.~\ref{fig:n1000dyneqC}. For $\alpha=1.0$ (right), the two path
    histograms correspond very well with the theoretically computed equilibrium
    density for $n=\infty$ (red curve), whereas for$\alpha=1.8$,
    neither of the tagged particles has a histogram which is similar to
    the equilibrium density\,.  }  
  \label{fig:n1000dyneqD}
\end{figure}

How this comes about is partially illustrated in
Fig.~\ref{fig:n1000dyneqE} which shows two scatterplots of the points
$(x_k,x_k*)$, {\em i.e.} the energy before and after a collision for
one of the particles involved in a collision. The graphs contain
200000 points, and clearly show that with $\alpha=1.8$ there is not
enough space at low energies for other outcomes of a collision than to
fall back in the same energy gap as where the particle was before the
collision.  

\begin{figure}[H]
  \centering
   \includegraphics[width=0.45\textwidth]{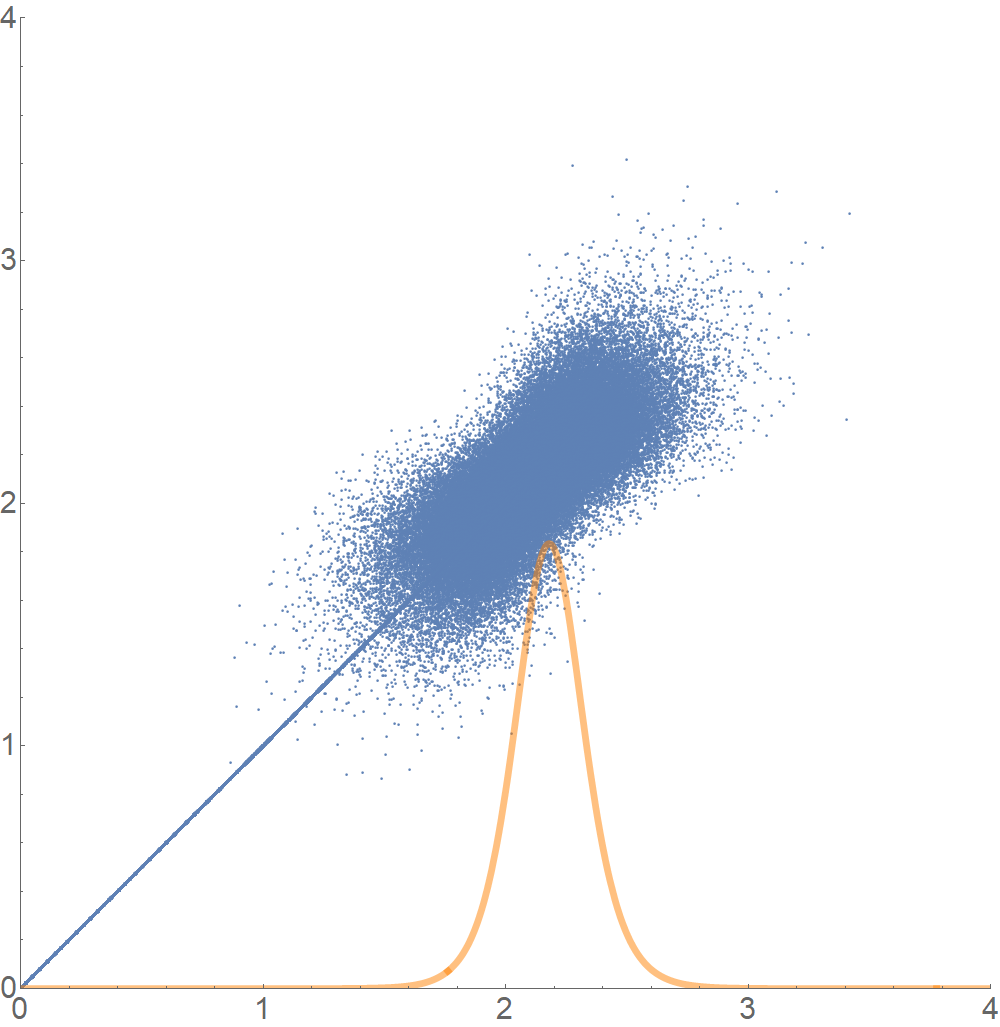}\qquad
   \includegraphics[width=0.45\textwidth]{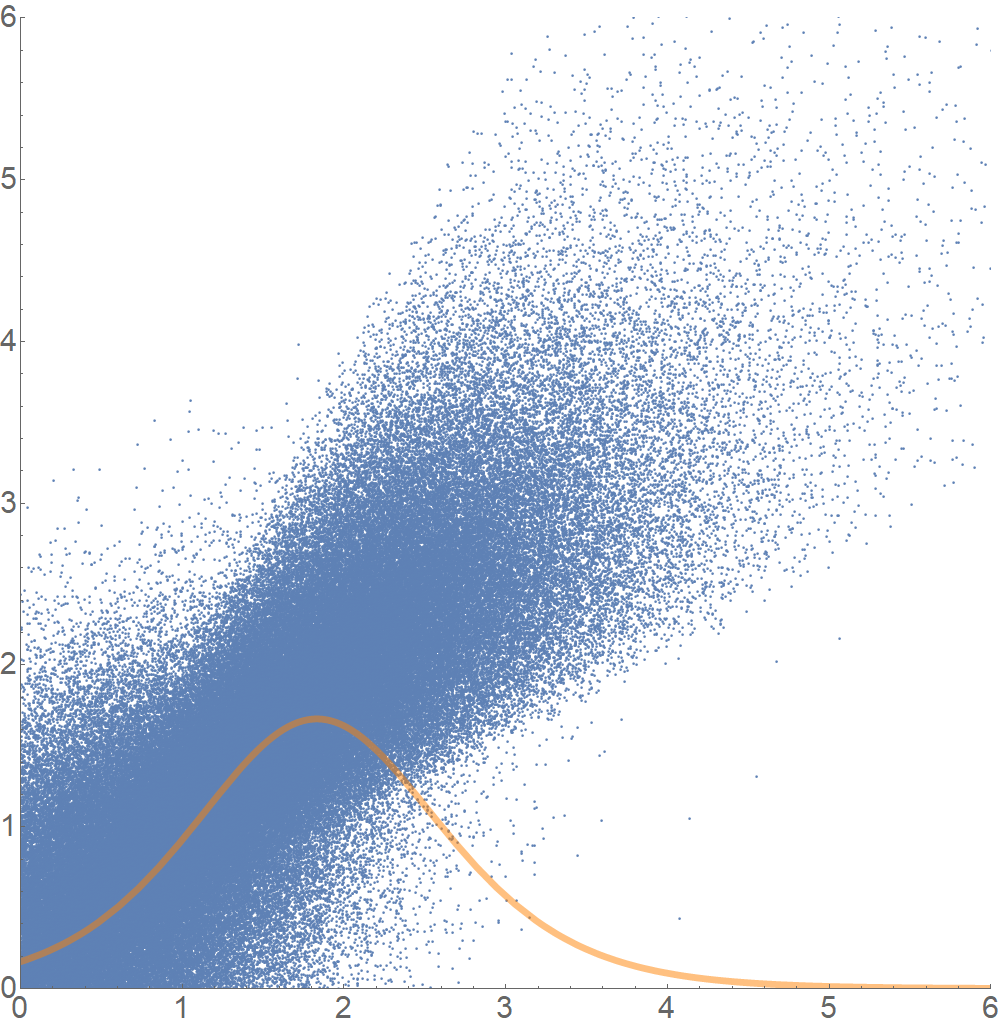}   
  \caption{A plot of one of the pairs $(x_j,x_j^*)$, the energy before
  and after collision for one of the particles involved in the first
  100000 collisions in a simulation. In the left, $\alpha=1.8$, and to
  the right, $\alpha=1.0$. The orange curve shows
  $f(x) \Pi(\alpha f(x))$. The plots appear to be symmetric around
  $x=x^*$, as one should expect if detailed balance holds. }   
  \label{fig:n1000dyneqE}
\end{figure}

The following results come from  simulations starting from initial
distributions far from equlibrium. The dynamical process depends
strongly on the value of $\alpha$, 
because in areas with high particle densites, very few collisions
attempts will actually result in a change. The figures~\ref{fig:ginit_t0},
\ref{fig:ginit_t01}, and~\ref{fig:ginit_t10} are the results of
simulations where the initial data are taken as the function $g(x)$ as
shown to the right of Figure~\ref{fig:excess} of the main article. The initial
function is plotted in red and the equilibrium density in green. The
black step functions shows the empirical histgram from the simulation,
and the blue dots illustrate the variation around the histogram
means. With a higher density ($\alpha$ large) the rate of convergence
to equlibrium is slower.

%
%
%
\begin{figure}[H]
  \centering
 
    \includegraphics[width=0.44\textwidth]{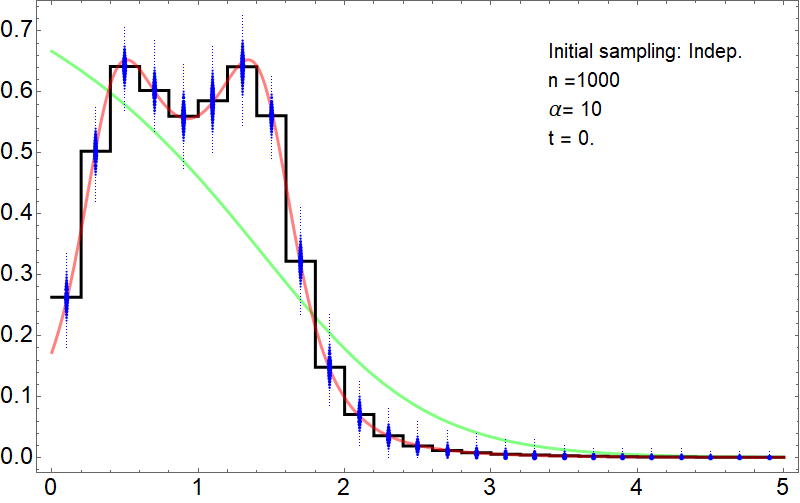}
\quad   
    \includegraphics[width=0.44\textwidth]{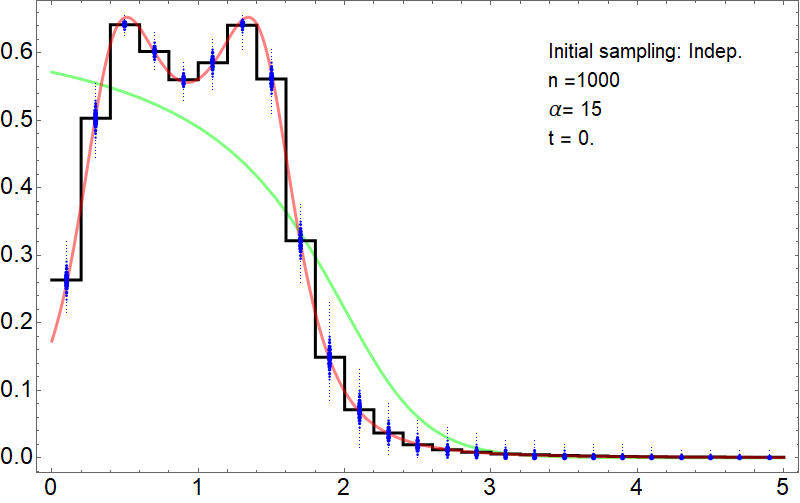}   

  \caption{The initial data is the function $g$ as in
    Fig.~\ref{fig:excess}, with $\alpha=1.0$ to the left and
    $\alpha=1.5$ to the right. The blue points show the distribution
    of simulation results for each histogram bin. The number of
    particls is 1000 and the 
    number independent samples is 5000. }  
  \label{fig:ginit_t0}
\end{figure}

\begin{figure}[H]
  \centering
 
    \includegraphics[width=0.45\textwidth]{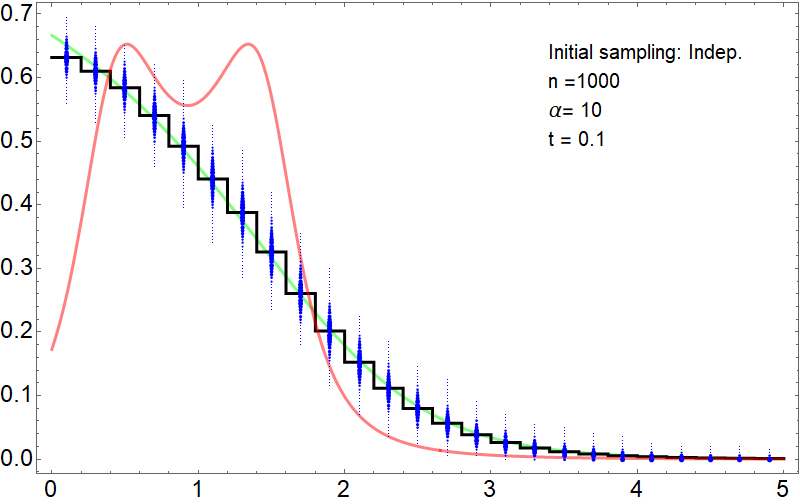}
\quad   
    \includegraphics[width=0.45\textwidth]{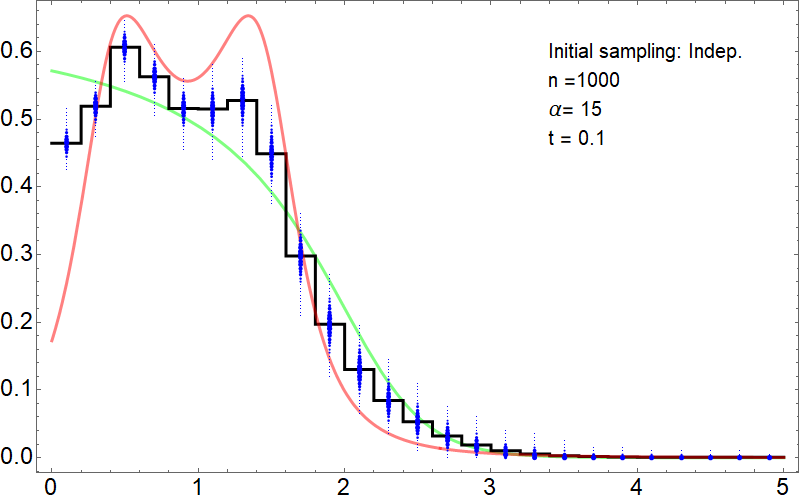}   

  \caption{At time $t=0.1$, the result is already very close to
    equilibrium when $\alpha=1.0$ (left), but not when $\alpha=1.5$ (right) }  
  \label{fig:ginit_t01}
\end{figure}

\begin{figure}[H]
  \centering
 
    \includegraphics[width=0.4\textwidth]{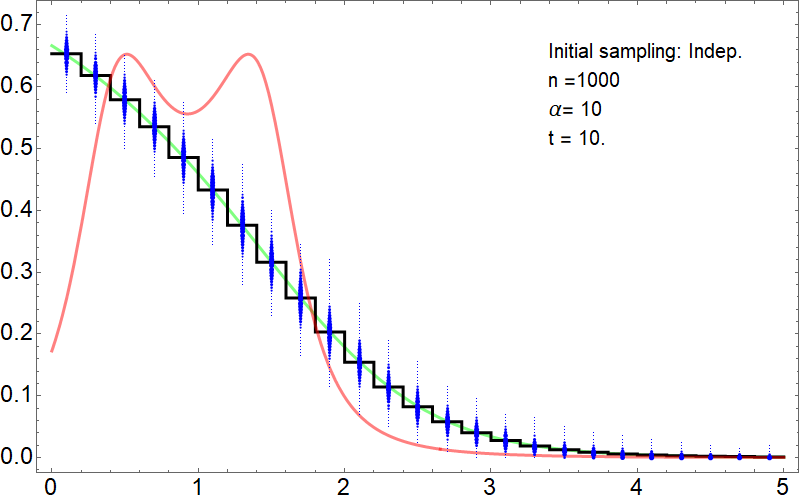}
    \includegraphics[width=0.4\textwidth]{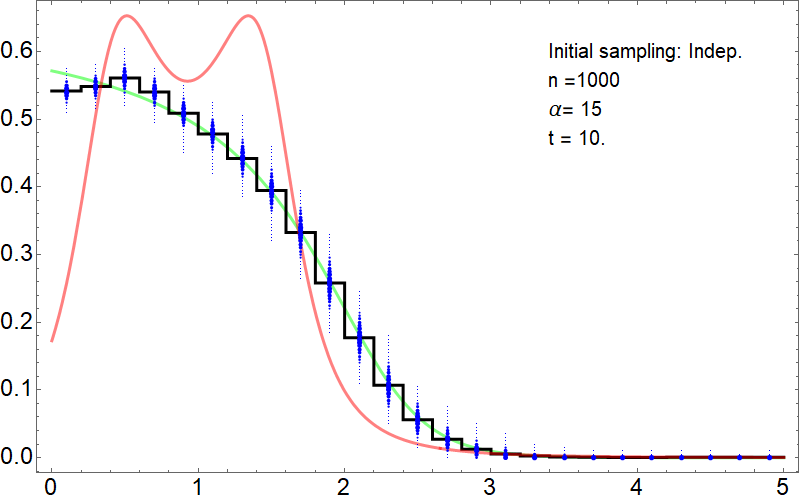}   

  \caption{At time $t=10$, the solution has still not converged to
    equilibrium  when $\alpha=1.5$ (right), because here, for small
    $x$ the particle density is close to its maximum, and there is not
  much space for collisions to take place.}  
  \label{fig:ginit_t10}
\end{figure}

An more extreme situation is shown in the final two examples. In
Figure~\ref{fig:dynnoneq1N1000}, the initial 
density is a single
step function, where all the particles are equally spaced, and pushed
as far towards the lower energy as possible. As discussed in
Section~\ref{sec:prechaos} of the paper, such sequences are chaotic
but not in the detailed sense.  The simulations were carried out with $n=10000$
particles. And we see that with $\alpha=1.8$ ( left of
Fig.~\ref{fig:dynnoneq1N1000} ),
the single step function is nearly a frozen state, there is hardly any
change over the simulation period, which in this case is
$T=1000$. Initially the only possible jumps 
are those where the two involved particles fall back to essentially
the same energy level they had before the collision. This is a slow
diffusive motion, that may be compared with, and perhaps possible to
analyze in the same way as models for competing particle systems and
rank based interacting
diffusions~\cite{PalPitman2008,Shkolnikov2012,Reygner2015}. On the
other hand, with $\alpha=1.0$, there is enough space between the
particles to allow for long jumps, and the convergence towards
equilibrium is much faster so that already after time $T=10$ the
distribution is close to equilibrium (the right side of the same figure).

In the final example, shown in Figure~\ref{fig:dynnoneq2N1000} the
initial configuration is one where the first $n-1$ are pushed even closer together to 
lower energies, and the $n^{\mbox{\tiny th}}$ particle is given sufficiently
high energy to give the same total energy.  Here we see a convergence towards
equilibrium also when $\alpha=1.8$ (left), and faster for
$\alpha=1.0$. What happens is that when 
the particle with highest energy (outside the range of the graph)
interacts with one of the other particles, the group of particles will
be torn apart, leaving space for long jumps, and then the convergence
to equilibrium can be seen also there.

Of course one can construct a
sequence of interactions that transforms the initial data of
Fig.~\ref{fig:dynnoneq1N1000} into the initial data of
Fig.~\ref{fig:dynnoneq2N1000}, and therefore one would expect
convergence to equilibrium also from this initial configuration, but
it may take a very long time to happen.

\begin{figure}[H]
  \centering
   \includegraphics[width=0.45\textwidth]{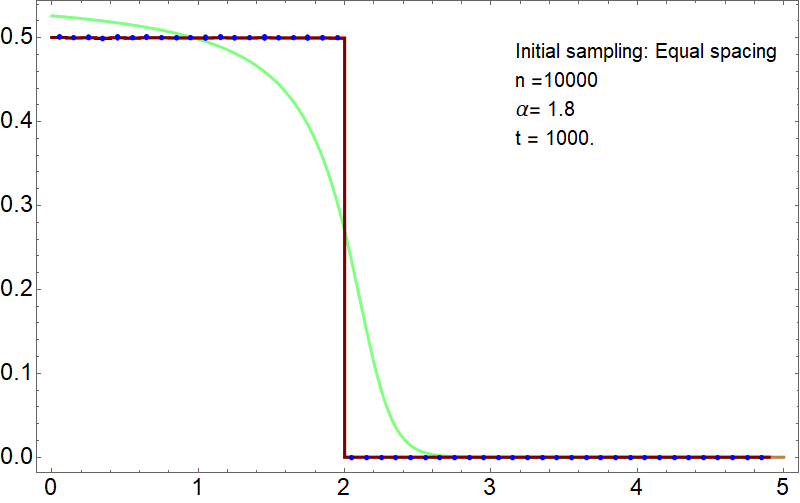}\qquad
   \includegraphics[width=0.45\textwidth]{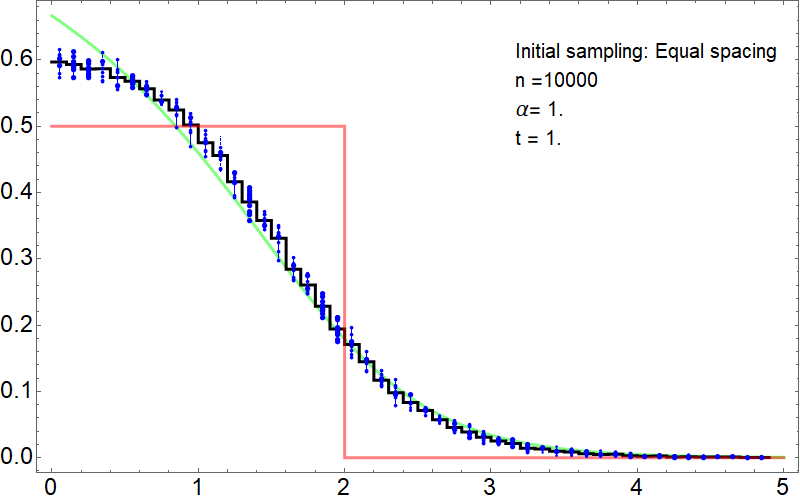}   
  \caption{The black step function shows the empirical histogram after
  time $T=1000$ for $\alpha=1.8$ (left), and after time $T=1.0$ for
  $\alpha=1.0$ (right). In the initial configuration the $x_i$ are put
at equal distance, corresponding to the step function represented in red.}  
  \label{fig:dynnoneq1N1000}
\end{figure}

\begin{figure}[H]
  \centering
   \includegraphics[width=0.45\textwidth]{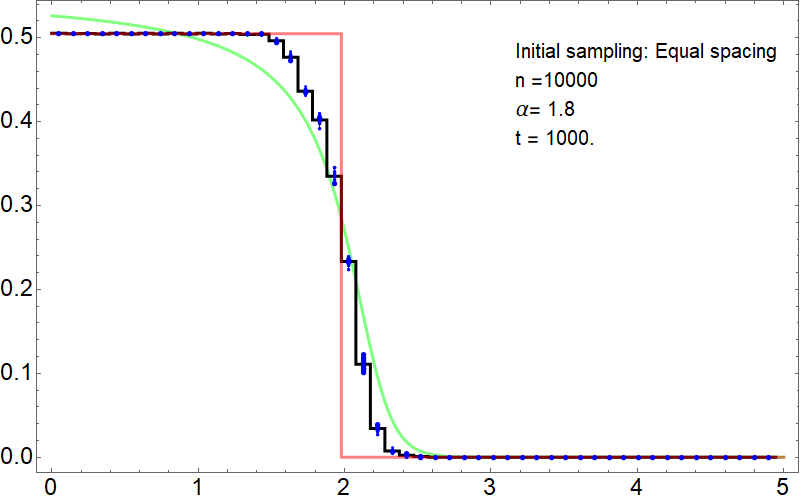}\qquad
   \includegraphics[width=0.45\textwidth]{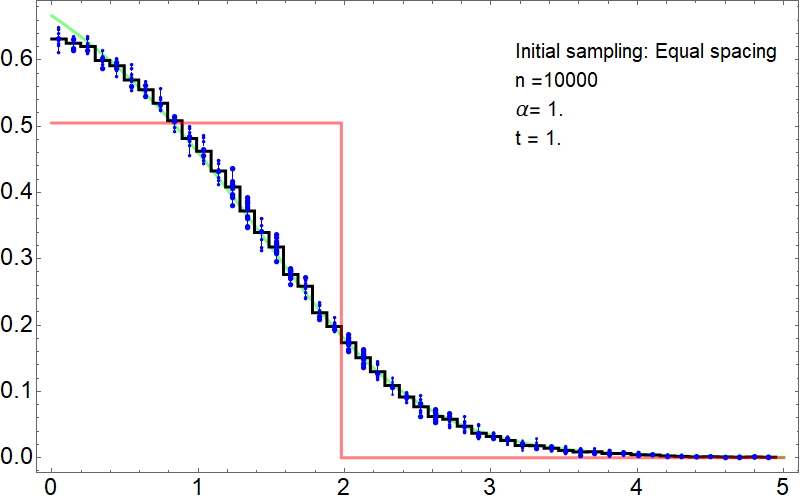}   
  \caption{The graphs show simulations results as in
    Fig.~\ref{fig:dynnoneq1N1000} , except that the initial data
    consists of a initial step function, constructed as above, but
    with 99\% of the mass compressed to a tighter configuration, and
    the remaining particles put at a higher energy (near $x=11$) to
    keep the same initial energy. Here we observe a convergence
    towards equilibrium also when $\alpha=1.8$, although this happens
    very slowly}
  \label{fig:dynnoneq2N1000}
\end{figure}

\end{document}